\documentclass[10pt,a4paper,oneside]{article}
\usepackage[letterpaper,left=1in,right=1in,top=1.5in,bottom=1.5in]{geometry}
\usepackage{tabularx}
\usepackage{booktabs} 
\usepackage{url}
\usepackage{hyperref}
\usepackage{xcolor}
\definecolor{lightgreen}{rgb}{.20,.60,.22}
\hypersetup{
	colorlinks=true, 
	linkcolor=purple,  
	citecolor=lightgreen, 
	urlcolor=cyan,   
}
\usepackage{pdfpages}

\usepackage{setspace}

\usepackage{microtype}
\usepackage{amssymb,amsmath,amsopn,amsxtra,amsthm,amsfonts}
\usepackage[mathcal]{euscript}
\let\mathscr\mathcal

\usepackage[bb=px]{mathalfa}

\usepackage[capitalise,nameinlink]{cleveref}
\usepackage[backend=biber,style=alphabetic,maxnames=99,giveninits=true]{biblatex}  
\DeclareFieldFormat[article,unpublished,misc]{title}{#1} 
\usepackage[T1]{fontenc}

\DeclareFontFamily{U}{min}{}
\DeclareFontShape{U}{min}{m}{n}{<-> udmj30}{}

\usepackage{comment}
\usepackage{todonotes}

\usepackage{enumitem}
\setlist[enumerate,1]{label={(\arabic*)},itemsep=\parskip} 
\setlist[itemize,1]{itemsep=\parskip} 
\newlist{thmlist}{enumerate}{2}
\setlist[thmlist,1]{label={\em(\roman*)},ref={(\roman*)},%
	itemsep=\parskip,leftmargin=*,align=left}
\setlist[thmlist,2]{label={\em(\alph*)},ref={(\alph*)},%
	itemsep=\parskip,leftmargin=*,align=left,topsep=0.1cm}
\newlist{defnlist}{enumerate}{2}
\setlist[defnlist,1]{label={(\roman*)},ref={(\roman*)},itemsep=\parskip,%
	leftmargin=*,align=left}
\setlist[defnlist,2]{label={(\alph*)},ref={(\alph*)},itemsep=\parskip,%
	leftmargin=*,align=left,topsep=0.1cm}

\newtheorem{thm}{Theorem}
\newtheorem{cor}[thm]{Corollary}
\newtheorem{lem}[thm]{Lemma}
\newtheorem{prop}[thm]{Proposition}

\theoremstyle{definition}

\newtheorem{de}[thm]{Definition}
\newtheorem*{de*}{Definition}

\newtheorem{rem}[thm]{Remark}

\newtheorem{ex}[thm]{Example}

\newtheorem{nota}[thm]{Notation}

\renewcommand{\eqref}[1]{(\ref{#1})}

\numberwithin{equation}{subsection}

\usepackage{tikz}
\usetikzlibrary{matrix}
\usepackage{tikz-cd}

\usepackage{xspace}

\usepackage{xifthen}

\usepackage{xparse}

\usepackage{mathtools}

\newcommand{\nc}{\newcommand}
\nc{\renc}{\renewcommand}
\nc{\ssec}{\subsection}
\nc{\sssec}{\subsubsection}
\nc{\on}{\operatorname}
\nc{\term}[1]{#1\xspace}

\nc{\sA}{\ensuremath{\mathcal{A}}\xspace}
\nc{\sB}{\ensuremath{\mathcal{B}}\xspace}
\nc{\sC}{\ensuremath{\mathcal{C}}\xspace}
\nc{\sD}{\ensuremath{\mathcal{D}}\xspace}
\nc{\sE}{\ensuremath{\mathcal{E}}\xspace}
\nc{\sF}{\ensuremath{\mathcal{F}}\xspace}
\nc{\sG}{\ensuremath{\mathcal{G}}\xspace}
\nc{\sH}{\ensuremath{\mathcal{H}}\xspace}
\nc{\sI}{\ensuremath{\mathcal{I}}\xspace}
\nc{\sJ}{\ensuremath{\mathcal{J}}\xspace}
\nc{\sK}{\ensuremath{\mathcal{K}}\xspace}
\nc{\sL}{\ensuremath{\mathcal{L}}\xspace}
\nc{\sM}{\ensuremath{\mathcal{M}}\xspace}
\nc{\sN}{\ensuremath{\mathcal{N}}\xspace}
\nc{\sO}{\ensuremath{\mathcal{O}}\xspace}
\nc{\sP}{\ensuremath{\mathcal{P}}\xspace}
\nc{\sQ}{\ensuremath{\mathcal{Q}}\xspace}
\nc{\sR}{\ensuremath{\mathcal{R}}\xspace}
\nc{\sS}{\ensuremath{\mathcal{S}}\xspace}
\nc{\sT}{\ensuremath{\mathcal{T}}\xspace}
\nc{\sU}{\ensuremath{\mathcal{U}}\xspace}
\nc{\sV}{\ensuremath{\mathcal{V}}\xspace}
\nc{\sW}{\ensuremath{\mathcal{W}}\xspace}
\nc{\sX}{\ensuremath{\mathcal{X}}\xspace}
\nc{\sY}{\ensuremath{\mathcal{Y}}\xspace}
\nc{\sZ}{\ensuremath{\mathcal{Z}}\xspace}

\nc{\bA}{\ensuremath{\mathbf{A}}\xspace}
\nc{\bB}{\ensuremath{\mathbf{B}}\xspace}
\nc{\bC}{\ensuremath{\mathbf{C}}\xspace}
\nc{\bD}{\ensuremath{\mathbf{D}}\xspace}
\nc{\bE}{\ensuremath{\mathbf{E}}\xspace}
\nc{\bF}{\ensuremath{\mathbf{F}}\xspace}
\nc{\bG}{\ensuremath{\mathbf{G}}\xspace}
\nc{\bH}{\ensuremath{\mathbf{H}}\xspace}
\nc{\bI}{\ensuremath{\mathbf{I}}\xspace}
\nc{\bJ}{\ensuremath{\mathbf{J}}\xspace}
\nc{\bK}{\ensuremath{\mathbf{K}}\xspace}
\nc{\bL}{\ensuremath{\mathbf{L}}\xspace}
\nc{\bM}{\ensuremath{\mathbf{M}}\xspace}
\nc{\bN}{\ensuremath{\mathbf{N}}\xspace}
\nc{\bO}{\ensuremath{\mathbf{O}}\xspace}
\nc{\bP}{\ensuremath{\mathbf{P}}\xspace}
\nc{\bQ}{\ensuremath{\mathbf{Q}}\xspace}
\nc{\bR}{\ensuremath{\mathbf{R}}\xspace}
\nc{\bS}{\ensuremath{\mathbf{S}}\xspace}
\nc{\bT}{\ensuremath{\mathbf{T}}\xspace}
\nc{\bU}{\ensuremath{\mathbf{U}}\xspace}
\nc{\bV}{\ensuremath{\mathbf{V}}\xspace}
\nc{\bW}{\ensuremath{\mathbf{W}}\xspace}
\nc{\bX}{\ensuremath{\mathbf{X}}\xspace}
\nc{\bY}{\ensuremath{\mathbf{Y}}\xspace}
\nc{\bZ}{\ensuremath{\mathbf{Z}}\xspace}

\nc{\dA}{\ensuremath{\mathds{A}}\xspace}
\nc{\dB}{\ensuremath{\mathds{B}}\xspace}
\nc{\dC}{\ensuremath{\mathds{C}}\xspace}
\nc{\dD}{\ensuremath{\mathds{D}}\xspace}
\nc{\dE}{\ensuremath{\mathds{E}}\xspace}
\nc{\dF}{\ensuremath{\mathds{F}}\xspace}
\nc{\dG}{\ensuremath{\mathds{G}}\xspace}
\nc{\dH}{\ensuremath{\mathds{H}}\xspace}
\nc{\dI}{\ensuremath{\mathds{I}}\xspace}
\nc{\dJ}{\ensuremath{\mathds{J}}\xspace}
\nc{\dK}{\ensuremath{\mathds{K}}\xspace}
\nc{\dL}{\ensuremath{\mathds{L}}\xspace}
\nc{\dM}{\ensuremath{\mathds{M}}\xspace}
\nc{\dN}{\ensuremath{\mathds{N}}\xspace}
\nc{\dO}{\ensuremath{\mathds{O}}\xspace}
\nc{\dP}{\ensuremath{\mathds{P}}\xspace}
\nc{\dQ}{\ensuremath{\mathds{Q}}\xspace}
\nc{\dR}{\ensuremath{\mathds{R}}\xspace}
\nc{\dS}{\ensuremath{\mathds{S}}\xspace}
\nc{\dT}{\ensuremath{\mathds{T}}\xspace}
\nc{\dU}{\ensuremath{\mathds{U}}\xspace}
\nc{\dV}{\ensuremath{\mathds{V}}\xspace}
\nc{\dW}{\ensuremath{\mathds{W}}\xspace}
\nc{\dX}{\ensuremath{\mathds{X}}\xspace}
\nc{\dY}{\ensuremath{\mathds{Y}}\xspace}
\nc{\dZ}{\ensuremath{\mathds{Z}}\xspace}

\nc{\bbA}{\ensuremath{\mathbb{A}}\xspace}
\nc{\bbB}{\ensuremath{\mathbb{B}}\xspace}
\nc{\bbC}{\ensuremath{\mathbb{C}}\xspace}
\nc{\bbD}{\ensuremath{\mathbb{D}}\xspace}
\nc{\bbE}{\ensuremath{\mathbb{E}}\xspace}
\nc{\bbF}{\ensuremath{\mathbb{F}}\xspace}
\nc{\bbG}{\ensuremath{\mathbb{G}}\xspace}
\nc{\bbH}{\ensuremath{\mathbb{H}}\xspace}
\nc{\bbI}{\ensuremath{\mathbb{I}}\xspace}
\nc{\bbJ}{\ensuremath{\mathbb{J}}\xspace}
\nc{\bbK}{\ensuremath{\mathbb{K}}\xspace}
\nc{\bbL}{\ensuremath{\mathbb{L}}\xspace}
\nc{\bbM}{\ensuremath{\mathbb{M}}\xspace}
\nc{\bbN}{\ensuremath{\mathbb{N}}\xspace}
\nc{\bbO}{\ensuremath{\mathbb{O}}\xspace}
\nc{\bbP}{\ensuremath{\mathbb{P}}\xspace}
\nc{\bbQ}{\ensuremath{\mathbb{Q}}\xspace}
\nc{\bbR}{\ensuremath{\mathbb{R}}\xspace}
\nc{\bbS}{\ensuremath{\mathbb{S}}\xspace}
\nc{\bbT}{\ensuremath{\mathbb{T}}\xspace}
\nc{\bbU}{\ensuremath{\mathbb{U}}\xspace}
\nc{\bbV}{\ensuremath{\mathbb{V}}\xspace}
\nc{\bbW}{\ensuremath{\mathbb{W}}\xspace}
\nc{\bbX}{\ensuremath{\mathbb{X}}\xspace}
\nc{\bbY}{\ensuremath{\mathbb{Y}}\xspace}
\nc{\bbZ}{\ensuremath{\mathbb{Z}}\xspace}

\nc{\mrm}[1]{\ensuremath{\mathrm{#1}}\xspace}

\nc{\mbf}[1]{\ensuremath{\mathbf{#1}}\xspace}
\nc{\mcal}[1]{\ensuremath{\mathcal{#1}}\xspace}
\nc{\msc}[1]{\ensuremath{\mathscr{#1}}\xspace}
\nc{\mfr}[1]{\ensuremath{\mathfrak{#1}}\xspace}

\renc{\bar}[1]{\overline{#1}}
\nc{\sub}{\subset}
\nc{\too}{\longrightarrow}
\nc{\hook}{\hookrightarrow}
\nc*{\hooklongrightarrow}{\ensuremath{\lhook\joinrel\relbar\joinrel\rightarrow}}
\nc{\hooklong}{\hooklongrightarrow}
\nc{\twoheadlongrightarrow}{\relbar\joinrel\twoheadrightarrow}
\nc{\shiso}{\approx}
\nc{\isoto}{\xrightarrow{\sim}}
\nc{\isofrom}{\xleftarrow{\sim}}
\renc{\ge}{\geqslant}
\renc{\le}{\leqslant}

\nc{\id}{\mathrm{id}}

\DeclareMathOperator{\Hom}{\on{Hom}}
\nc{\uHom}{\underline{\smash{\Hom}}}

\DeclareMathOperator{\End}{\on{End}}

\nc{\uEnd}{\underline{\smash{\End}}}

\renc{\lim}{\varprojlim}
\makeatletter
\newcommand{\colim@}[2]{%
	\vtop{\m@th\ialign{##\cr
			\hfil$#1\operator@font colim$\hfil\cr
			\noalign{\nointerlineskip\kern1.5\ex@}#2\cr
			\noalign{\nointerlineskip\kern-\ex@}\cr}}%
}
\newcommand{\colim}{%
	\mathop{\mathpalette\colim@{\rightarrowfill@\textstyle}}\nmlimits@
}
\makeatother

\nc{\Cofib}{\on{Cofib}}
\nc{\Fib}{\on{Fib}}
\nc{\initial}{\varnothing}

\renewbibmacro{in:}{}

\def\mc{\mathcal}
\def\mb{\mathbf}

\def\op{\operatorname}
\def\ein{$\mathbb{E}_{\infty}$}

\newcommand{\enu}[1]{ \begin{enumerate}[label=(\arabic*),font=\normalfont]
		#1
	\end{enumerate}
}

\newcommand{\alg}{\operatorname{Alg}}
\newcommand{\calg}{\operatorname{CAlg}}
\newcommand{\ccalg}{\operatorname{cCAlg}}
\newcommand{\cmon}{\operatorname{CMon}}
\newcommand{\modu}{\operatorname{Mod}}

\newcommand{\mapp}{\operatorname{Map}}

\newcommand{\funct}{\operatorname{Fun}}

\newcommand{\prl}{\ensuremath{\mc{P}\mathrm{r}^L}\xspace}

\newcommand{\pr}{\ensuremath{\mc{P}\mathrm{r}}\xspace}
\newcommand{\prad}{\ensuremath{\mc{P}\mathrm{r}_{\op{ad}}}\xspace}

\newcommand{\praddbl}{\ensuremath{\mc{P}\mathrm{r}_{\op{ad}}^{\op{dbl}}}\xspace}

\newcommand{\prlst}{\ensuremath{\mc{P}\mathrm{r}_{\op{st}}^L}\xspace}

\newcommand{\prst}{\ensuremath{\mc{P}\mathrm{r}_{\op{st}}}\xspace}
\newcommand{\prstdbl}{\ensuremath{\mc{P}\mathrm{r}_{\op{st}}^{\op{dbl}}}\xspace}
\newcommand{\prvdbl}{\ensuremath{\mc{P}\mathrm{r}_{\mc{V}}^{\op{dbl}}}\xspace}

\newcommand{\spgeq}{\ensuremath{\operatorname{Sp}_{\geq0}}\xspace}

\newcommand{\cotimes}{\ensuremath{\mc{C}^\otimes}\xspace}
\newcommand{\totimes}{\ensuremath{\mc{T}^\otimes}\xspace}
\newcommand{\votimes}{\ensuremath{\mc{V}^\otimes}\xspace}

\newcommand{\einf}{\term{\mathbb{E}_\infty}}

\newcommand{\infcat}{\term{$\infty$-category}}
\newcommand{\infcats}{\term{$\infty$-categories}}
\newcommand{\syminfcat}{\term{symmetric monoidal $\infty$-category}}

\newcommand{\shv}{\operatorname{Shv}}

\newcommand{\cidl}{(\mc{C}_{/\mb{1}})_{\leq-1}}

\newcommand{\idl}{\op{Idl}}

\newcommand{\idlc}{\op{Idl}(\mc{C})}
\newcommand{\idlcotimes}{\op{Idl}(\mc{C})^{\otimes}}
\newcommand{\im}{\op{Im}}

\newcommand{\mca}{\ensuremath{\mathcal{A}}\xspace}
\newcommand{\mcb}{\ensuremath{\mathcal{B}}\xspace}
\newcommand{\mcc}{\ensuremath{\mathcal{C}}\xspace}
\newcommand{\mcd}{\ensuremath{\mathcal{D}}\xspace}
\newcommand{\mce}{\ensuremath{\mathcal{E}}\xspace}

\newcommand{\mci}{\ensuremath{\mathcal{I}}\xspace}

\newcommand{\mck}{\ensuremath{\mathcal{K}}\xspace}
\newcommand{\mcl}{\ensuremath{\mathcal{L}}\xspace}
\newcommand{\mcm}{\ensuremath{\mathcal{M}}\xspace}
\newcommand{\mcn}{\ensuremath{\mathcal{N}}\xspace}

\newcommand{\mcs}{\ensuremath{\mathcal{S}}\xspace}
\newcommand{\mct}{\ensuremath{\mathcal{T}}\xspace}
\newcommand{\mcu}{\ensuremath{\mathcal{U}}\xspace}
\newcommand{\mcv}{\ensuremath{\mathcal{V}}\xspace}
\newcommand{\mcw}{\ensuremath{\mathcal{W}}\xspace}
\newcommand{\mcx}{\ensuremath{\mathcal{X}}\xspace}

\newcommand{\Cat}{\op{Cat}}
\newcommand{\spec}{\op{Spec}}
\newcommand{\dlat}{\mathsf{DLat}}
\newcommand{\Top}{\mathsf{Top}^{\op{op}}}
\newcommand{\topsob}{\Top_{\op{sob}}}
\newcommand{\topcoh}{\mathsf{Top}_{\mathsf{coh}}^{\op{op}}}

\newcommand{\frmsp}{\mathsf{Frm}_{\mathsf{spa}}}
\newcommand{\frmcoh}{\mathsf{Frm}_{\mathsf{coh}}}
\newcommand{\frm}{\mathsf{Frm}}
\newcommand{\pfrm}{\mathsf{PFrm}}

\newcommand{\zar}{\op{Zar}}

\newcommand{\coker}{\op{coker}}
\newcommand{\catper}{\op{Cat}^{\op{perf}}}

\newcommand{\lint}{\ensuremath{\prl_{\mc{T}}}}

\newcommand{\lintdbl}{\ensuremath{\pr_{\mc{T}}^{\op{dbl}}}}
\newcommand{\linvdbl}{\ensuremath{\pr_{\mc{V}}^{\op{dbl}}}}
\newcommand{\opsp}{\ensuremath{\op{Sp}}}
\newcommand{\cidem}{\ensuremath{\op{cIdem}}\xspace}

\newcommand{\cidemt}{\ensuremath{\op{cIdem}(\mathcal{T})}\xspace}
\newcommand{\cidemv}{\ensuremath{\op{cIdem}(\mathcal{V})}\xspace}

\usepackage{mathtools} 

\nc{\rnc}{\renewcommand}
\nc{\dmo}{\DeclareMathOperator}

\rnc{\emptyset}{\varnothing}
\nc{\Mid}{\,\big|\,}
\nc{\SET}[2]{\big\{\,#1\Mid#2\,\big\}} 
\nc{\xra}{\xrightarrow}

\nc{\downset}[1]{\downarrow #1}
\nc{\ie}{{i.e.}, }
\nc{\cf}{{cf.~}}

\title{Smashing, Balmer, Zariski spectra: an ideal approach}

\author{Jiacheng Liang and Changhan Zou}   
\date{}
\begin{document}

\let\mathbb=\mathbf
\maketitle
\begin{abstract}
We introduce the Zariski frame of any presentably symmetric monoidal $\infty$-category. This allows us to unify several spectral theories arising in higher algebra. The Zariski frame is coherent whenever the category is compactly generated, and the associated spectral space recovers both the classical Zariski spectrum of a commutative ring and the Hochster dual of the Balmer spectrum of a commutative $2$-ring. Moreover, the smashing frame of any stable presentably symmetric monoidal $\infty$-category can be identified with the Zariski frame of its category of dualizable modules. This construction is based on the principle that ideals in a symmetric monoidal $\infty$-category should be understood as monomorphisms into the unit object. In suitable contexts, this notion recovers the kinds of ideals appearing in the preceding examples, including thick ideals and smashing ideals, and it also accommodates the smashing ideals of non-stable $\infty$-categories. We also study the problem of forming quotients by ideals, which is subtle in the setting of higher algebra. To address this, we introduce two properties of pointed $\infty$-categories, called $\Sigma$-triviality and $\Sigma$-exactness. These conditions ensure that quotienting by ideals behaves well. As an application, we construct quotients of $\mathbb{E}_\infty$-semirings.
\end{abstract}

\tableofcontents
\section{Introduction}
The geometric classification of ideals is a fundamental tool across modern algebraic geometry, modular representation theory, and stable homotopy theory. This perspective builds on the classical Zariski spectrum, which organizes the prime ideals of a commutative ring into a topological space. Balmer’s tensor triangular geometry \cite{Balmer05a} extends this logic to the categorical setting. By classifying the thick ideals of an essentially small tensor triangulated category, the Balmer spectrum recovers classical schemes in algebraic geometry, as well as support varieties in modular representation theory and chromatic structures in stable homotopy theory.

The smashing ideals of a big tt-$\infty$-category $\mct$---that is, a presentably symmetric monoidal stable $\infty$-category---give another spectral construction. By a smashing ideal we mean the kernel of a smashing localization of $\mct$. In \cite{BalmerKrauseStevenson20}, Balmer--Krause--Stevenson showed that, when $\mct$ is rigidly-compactly generated, the poset of smashing ideals is a frame. The smashing frame is closely related to the notion of a categorified locale introduced by Clausen and Scholze in their condensed approach to analytic geometry \cite{clausen2026condensed}.

Although Balmer spectra and smashing frames both extract geometric information from tensor-triangulated categories, they have largely been studied in separate frameworks. In this work, we unify these constructions by introducing a Zariski frame for any presentably symmetric monoidal $\infty$-category, of which both the Balmer spectrum and the smashing frame arise as special cases. More precisely, we define a functor
\[
\op{Zar}: \calg(\prl) \to \frm
\]
from presentably symmetric monoidal $\infty$-categories to frames, which we call the \textbf{Zariski frame functor}. By post-composing with the functor
\[
\op{pt} \colon \frm \to \Top
\]
that takes frame-theoretic points, we obtain the \textbf{Zariski spectrum functor}
\[
\spec: \calg(\prl) \to \Top.
\]

Our construction starts from a notion of \textbf{categorical ideal} for presentably symmetric monoidal $\infty$-categories. For $\mcc^\otimes \in \calg(\prl)$, a categorical ideal of $\mcc^\otimes$ is a subobject $I\hookrightarrow\mb1$ of the tensor unit. We show that the lattice $\idlc$ of categorical ideals of $\mcc^\otimes$ naturally has the structure of a preframe---a suplattice with a compatible monoidal structure such that the monoidal unit is the top element. Unlike a frame, a preframe may not satisfy the distributive law of finite meets over arbitrary joins. Nevertheless, the radical elements in any preframe form a frame. Applying this construction to $\idlc$, we obtain our Zariski frame $\zar(\mcc)$. In other words, the Zariski frame $\zar(\mcc)$ is the frame of radical ideals of $\mcc$. In the compactly generated case, $\zar(\mcc)$ has the expected finiteness property:
\begin{thm}[\cref{thm:zar-coh}]\label{thm:a}
Let $\cotimes\in\calg(\prl)$ be compactly generated. The Zariski frame $\op{Zar}(\mcc)$ is coherent. Consequently, the Zariski spectrum $\spec(\mcc)$ is a coherent (or spectral) space.
\end{thm}

As a result, we obtain the following commutative diagram
\[
\begin{tikzcd}
\calg(\prl_\omega) \arrow[d, hook] \arrow[r, "\op{Zar}"] \arrow[rr, "\spec", bend left] & \frmcoh \arrow[d, hook] \arrow[r, "\op{pt}", "\sim"'] & \topcoh \arrow[d, hook] \\
\calg(\prl) \arrow[r, "\op{Zar}"] \arrow[rr, "\spec"', bend right]                      & \frm \arrow[r, "\op{pt}"]                  & \Top           
\end{tikzcd}
\]
where the equivalence is Stone duality.

The construction recovers its principal motivating examples from the following module categories.

First, let $R$ be an ordinary commutative ring and take $\mcc^\otimes = \modu_{R}(\op{Ab})^\otimes$, the ordinary category of $R$-modules. Its categorical ideals are precisely the ideals of $R$. Hence, $\zar(\mcc)$ is the frame of radical ideals of $R$, and $\spec(\mcc)$ is the classical Zariski spectrum of $R$. See \cref{ex:Zariski}.

Next, let $\mck^\otimes \in \calg(\op{Cat}^{\op{perf}})$ be a commutative $2$-ring and take $\mcc^\otimes = \modu_{\mck}(\op{Cat}^{\op{perf}})^\otimes$, the $\infty$-category of $\mck$-modules. Its categorical ideals correspond to the thick ideals of $\mck$. Thus, $\zar(\mcc)$ is the frame of radical thick ideals of $\mck$, while $\spec(\mcc)$ is the Hochster dual of the Balmer spectrum of $\mck$. See \cref{ex:Balmer}.

Finally, for a big tt-$\infty$-category $\totimes\in\calg(\prlst)$. Consider $\mcc^\otimes = \modu_\mct(\prlst)^{\op{dbl},\otimes}$, the $\infty$-category of dualizable $\mct$-modules and internal left adjoints. The categorical ideals of $\mcc$ are the smashing ideals of $\mct$, and $\zar(\mcc)$ recovers the smashing frame $\op{Sm}(\mct)$. See \cref{dblideal}.

The results above are summarized in the table below:
\begin{center}
\begin{tabularx}{\textwidth}{p{2.5cm} p{3cm} X X}
\toprule
$\mcc^\otimes$ 
& Ideals of $\mcc$
& $\zar(\mcc)$
& $\spec(\mcc)$\\
\midrule

$\modu_{R}(\op{Ab})^\otimes$ 
& ideals of $R$ 
& frame of radical ideals of $R$
& Zariski spectrum of $R$\\

\addlinespace

$\modu_{\mck}(\op{Cat}^{\op{perf}})^\otimes$ 
& thick ideals of $\mck$ 
& frame of radical thick ideals of $\mck$
& Hochster dual of Balmer spectrum of $\mck$ \\

\addlinespace

$\modu_\mct(\prlst)^{\op{dbl},\otimes}$ 
& smashing ideals of $\mct$
& smashing frame of $\mct$
& smashing spectrum of $\mct$\\

\bottomrule
\end{tabularx}
\end{center}

As an application, \cref{thm:a} allows us to apply point-free topological results simultaneously to commutative algebra and tensor triangular geometry. We prove the following theorem, which characterizes when the spectral space of a coherent frame is Noetherian.
\begin{thm}[\cref{thm:cohen}]\label{thm:intro-cohen}
Let $F$ be a coherent frame. The following are equivalent:
\enu{
\item The spectral space $\op{pt}(F)$ is Noetherian;
\item Every element of $F$ is finite;
\item Every prime element of $F$ is finite.
}
\end{thm}
Combined with \cref{thm:a}, this gives a Cohen-type criterion for the Zariski spectrum of any compactly generated presentably symmetric monoidal $\infty$-category $\mcc$. Applied to $\mcc^\otimes = \modu_{\mck}(\op{Cat}^{\op{perf}})^\otimes$, \cref{thm:intro-cohen} recovers the criterion of \cite{Barthel26}: the Hochster dual of the Balmer spectrum of $\mck$ is Noetherian if and only if every prime thick ideal---or equivalently, every radical thick ideal---is the radical of a finitely generated thick ideal. Appied to $\mcc^\otimes = \modu_{R}(\op{Ab})^\otimes$, it gives the analogous criterion for the topological Noetherianity of $R$.

In addition to unifying familiar spectral theories, our Zariski frame functor also yields other interesting constructions, thanks to the flexible input it takes. For instance:
\begin{thm}[\cref{thm:cidem-zar-ccalg}]\label{thm:c}
Let $\votimes \in \calg(\prl)$.	The Zariski frame $\zar(\ccalg(\mcv))$ of the cocommutative coalgebras in $\mcv$ recovers the frame $\cidem(\mcv)$ of coidempotent objects introduced in \cite{aoki2023sheaves}.
\end{thm}
When $\mcv=\mct\in\calg(\prlst)$ is stable, the frame $\cidem(\mct)$ recovers the usual smashing frame $\op{Sm}(\mct)$ of \cite{BalmerKrauseStevenson20}. For a general $\mcv$, it can be regarded as a non-stable version of the smashing frame; see \cite{aoki2023sheaves}. It is an open question whether the frame $\op{Sm}(\mct)$ is coherent---and hence has an associated spectral space by Stone duality---even for rigidly-compactly generated $\mct$. Nevertheless, we can prove that $\op{Sm}(\mct)$ is $\omega_1$-coherent:
\begin{thm}[\cref{cor:omega1-coherent}]\label{thm:d}
Let $\totimes\in\calg^{\op{rig}}(\prlst)$ be a rigid $\opsp$-algebra. The smashing frame $\op{Sm}(\mct)$ is $\omega_1$-coherent.
\end{thm}
This theorem is obtained by identifying $\op{Sm}(\mct)$ with the Zariski frame of dualizable modules over $\mct$:
\begin{thm}[\cref{dblideal}]\label{thm:intro-sm-zar}
Let $\totimes\in\calg(\prlst)$. We have
\[
\op{Sm}(\mct) \simeq \zar(\modu_\mct(\prlst)^{\op{dbl}}).
\]
\end{thm}
\Cref{thm:intro-sm-zar} also suggests an unstable generalization of smashing frames. For $\votimes \in \calg(\prl)$, we define a smashing ideals of $\mcv$ to be a fully faithful categorical ideal $i:\mci\to \mcv$ in $\linvdbl\coloneqq\modu_\mcv(\prl)^{\op{dbl},\otimes}$. We show in \cref{smideal} that the smashing ideals of $\mcv$ are precisely the coidempotent objects of $\linvdbl$, and therefore they form a frame
\[
\op{Sm}(\mcv) \coloneqq \cidem(\linvdbl).
\]
The following theorem compares this ideal-theoretic definition with the intrinsic $\cidemv$.
\begin{thm}[\cref{smidsmcoloc}]
Let $\votimes\in \calg(\prl)$. There is an equivalence of frames
$$\cidem(\mcv)\xrightarrow{\sim}\cidem(\linvdbl)=\op{Sm}(\mcv).$$
\end{thm}
Together with \cref{thm:c}, this gives natural equivalences
\[
\op{Sm}(\mcv) \simeq \cidemv \simeq \zar(\ccalg(\mcv)).
\]

The construction of the Zariski frame provides a geometric classification of ideals in a very general context, yet the utility of these ideals in higher algebra is fundamentally tied to our ability to form quotients. In classical commutative algebra, there is a satisfactory correspondence between the ideals of a ring $R$ and the surjective ring homomorphisms out of $R$. In \cref{sec:sigma} we will establish a robust higher categorical analog of this correspondence within our framework.

Consider a commutative ring object $R \in \calg(\mcc)$ for $\cotimes\in\calg(\prl_*)$ a pointed presentably symmetric monoidal $\infty$-category. Let $I\subseteq R$ be a categorical ideal of $\modu_{R}(\mcc)^\otimes$. Naively, we form the quotient $R/I$ in $\mcc$ as the pushout
\[
\begin{tikzcd}
I \arrow[d] \arrow[r] & R \arrow[d] \\
0 \arrow[r]           & R/I.
\end{tikzcd}
\]
However, this construction does not guarantee that the quotient $R/I$ inherits a unique commutative algebra structure that satisfies the expected universal property of a quotient. To rectify this, we introduce:
\begin{de}
Let $\mcc$ be a pointed $\infty$-category. We say that $\mcc$ is \textbf{$\Sigma$-trivial} if for any object $X\in\mcc$, the map $X\to0$ is an epimorphism.
\end{de}
The name $\Sigma$-trivial comes from the fact that if $\mcc$ admits cofibers, then $\mcc$ is $\Sigma$-trivial if and only if the suspension functor $\Sigma:\mcc\to \mcc$ is zero. The $\Sigma$-trivial condition ensures that the naive pushout $R/I$ always makes sense as an algebraic quotient:
\begin{thm}[\cref{quotidl}]\label{thm:intro-quotidl}
Let $\cotimes\in\calg(\prl_*)$ be $\Sigma$-trivial. Let $R\in\calg(\mcc)$ and $I\subseteq R$ be a categorical ideal of $\modu_{R}(\mcc)$. There exists a unique commutative $R$-algebra structure on $R/I$. Moreover, the map $R\to R/I$ is an epimorphism in $\calg(\mc{C})$ and satisfies the following universal property: For any $B\in \calg(\mc{C})$, the induced map $$\mapp_{\calg(\mc{C})}(R/I,B)\to \mapp_{\calg(\mc{C})}(R,B)$$ is a $(-1)$-truncated map whose image on $\pi_0$ consists of the maps $R\to B$ such that the composite $I\to R\to B$ is homotopic to $0$ in $\modu_R(\mcc)$.
\end{thm}
As an application, we construct quotients of $\einf$-semirings by ideals whose underlying $\einf$-spaces have trivial $\mathbb{E}_1$-group completion (\cref{thm:quotient-einf-semiring}). After passing to unital $\einf$-semirings, quotients exist for arbitrary ideals (\cref{thm:quotient-unital-einf-semiring}).


\cref{thm:intro-quotidl} shows that, under the $\Sigma$-triviality hypothesis, every ideal admits a quotient with the expected universal property. To formulate the resulting ideal-epimorphism correspondence, we must take account of a distinction that does not arise in ordinary abelian categories: not every monomorphism need be a kernel, and not every epimorphism need be a cokernel.

Recall that a monomorphism is called normal if it is a kernel, and an epimorphism is called normal if it is a cokernel. Let $\mcc$ be a $\Sigma$-trivial $\infty$-category. We call $\mcc$ \textbf{left $\Sigma$-exact} if every monomorphism is normal, \textbf{right $\Sigma$-exact} if every epimorphism is normal, and \textbf{$\Sigma$-exact} if both conditions hold. Our general correspondence is first stated for normal ideals and normal epimorphisms:

\begin{thm}[\cref{thm:ideal-epi-iden}]\label{thm:intro-ideal-epi}
Let $\mcc^\otimes$ be a presentably symmetric monoidal $\Sigma$-trivial $\infty$-category. We have a natural equivalence $$\idl(\mcc)^n\xrightarrow{\sim}\calg(\mcc)^n$$
between the normal ideals of $\mcc$ (\ie normal monomorphisms $I \to \mb1$) and the commutative algebras $S\in\calg(\mcc)$ such that $\mb1 \to S$ is a normal epimorphism in $\mcc$. 
\end{thm}

Thus, $\Sigma$-triviality is sufficient for the correspondence between normal ideals and normal epimorphisms. The stronger $\Sigma$-exactness conditions allow the normality restrictions to be removed: left $\Sigma$-exactness ensures that every ideal is normal, while right $\Sigma$-exactness ensures that every epimorphism is normal.

For the ordinary category $\modu_{R}(\op{Ab})^\otimes$, every monomorphism and epimorphism is normal, and \cref{thm:intro-ideal-epi} recovers the classcial correspondence between the ideals of $R$ and quotient maps out of $R$.

For a commutative $2$-ring $\mck$, the category $\modu_{\mck}(\op{Cat}^{\op{perf}})^\otimes$ is left $\Sigma$-exact, but not right $\Sigma$-exact (\cref{ex:Sigma-exact}): every thick ideal is normal, whereas the normal epimorphisms are precisely the monoidal Karoubi projections .

Finally, if $\mct$ is a big tt-$\infty$-category, then $\modu_\mct(\prlst)^{\op{dbl},\otimes}$ is $\Sigma$-exact (\cref{sigmaexact}), and \cref{thm:intro-ideal-epi} gives the usual correspondence between smashing ideals and smashing localizations. These examples are summarized in the following table.

\begin{center}
\begin{tabularx}{\textwidth}{p{3cm} p{4cm} X}
\toprule
$\mcc^\otimes \in \calg(\prl)$
& (normal) ideals of $\mcc^\otimes$
& (normal) epimorphisms \\  
\midrule

$\modu_{R}(\op{Ab})^\otimes$ 
& ideals $I \subseteq R$
& surjective ring homomorphisms $R \to R/I$ \\

\addlinespace

$\modu_{\mck}(\op{Cat}^{\op{perf}})^\otimes$ 
& thick ideals $\mci \subseteq \mck$ 
& monoidal Karoubi projections $\mck \to \mck/\mci$ \\

\addlinespace

$\modu_\mct(\prlst)^{\op{dbl},\otimes}$ 
& smashing ideals $\mcl \subseteq \mct$
& smashing localizations $\mct \to \mct/\mcl$ \\
\bottomrule
\end{tabularx}
\end{center}

Stability is essential in the last example. For a general presentably symmetric monoidal $\infty$-category $\mcv$, the category $\linvdbl=\modu_\mcv(\prl)^{\op{dbl},\otimes}$ need not be $\Sigma$-exact. Consequently, passing from a smashing ideal to the associated smashing localization and then taking its dualizable kernel need not recover the original smashing ideal; see \cref{rem:prdbl-non-sigma-exact}.

\subsection*{Outline}
The paper is organized as follows.

In \cref{sec:zar}, we lay the foundations for (categorical) ideals and Zariski frames. Starting with the lattice-theoretic mechanics of preframes, we define the Zariski frame as the frame of radical ideals (\cref{de:zariski-frame}). We then restrict our attention to compactly generated symmetric monoidal $\infty$-categories, showing that they yield coherent Zariski frames (\cref{thm:zar-coh}). This allows us to recover both the classical Zariski spectrum of a commutative ring and the Balmer spectrum of a commutative $2$-ring (\cref{ex:Zariski}, \cref{ex:Balmer}). We then prove \cref{thm:cohen}. We conclude this section by clarifying the relations between our Zariski frames and the coidempotent frame in \cite{aoki2023sheaves}.

\cref{sec:dualizable} is devoted to the study of smashing ideals and smashing spectra. We define a smashing ideal as a fully faithful categorical ideal within the category of dualizable modules, which allows us to define the smashing frame for non-stable presentably symmetric monoidal $\infty$-categories. We then establish that for a big tt-$\infty$-category $\mct$, the smashing frame of $\mct$ naturally coincides with the Zariski frame of the category of dualizable modules over $\mct$ (\cref{dblideal}). We conclude \cref{sec:dualizable} by showing that the smashing frame of a rigid big tt-$\infty$-category $\mct$ is inherently $\omega_1$-coherent (\cref{cor:omega1-coherent}).

In \cref{sec:sigma}, we turn to the study of $\Sigma$-trivial and $\Sigma$-exact $\infty$-categories, with the goal of making sense of quotient by ideals in our context. We introduce the $\Sigma$-trivialization of any pointed presentable $\infty$-category and show that it gives rise to a mode (\cref{thm:mode}). After establishing the ideal-epimorphism correspondence (\cref{thm:ideal-epi-iden}), we apply it to construct quotients of $\mathbb{E}_\infty$-semirings when the underlying ideal possesses a trivial $\mathbb{E}_1$-group completion. We also give a parametrized version of the ideal-epimorphism correspondence in \cref{thm:parametrized-ideal-epi}.
\subsection*{Notation}

The following notation will be used throughout.

\begin{enumerate}[label=(\alph*)]
\item We denote the \(\infty\)-category of spaces (i.e.\ \(\infty\)-groupoids) by \(\mathcal{S}\), and denote the \(\infty\)-category of spectra by \(\mathrm{Sp}\).

\item We write \(\mathcal{S}_{\leq n}\) for the full subcategory of \(\mathcal{S}\) spanned by the \(n\)-truncated spaces. In particular, \(\mathcal{S}_{\leq -1}\) is the full subcategory spanned by \(\varnothing\) and \(*\), and \(\mathcal{S}_{\leq -2}\) is the full subcategory spanned by \(*\).

\item We denote by \(\mathrm{Cat}_{\infty}\) the \(\infty\)-category of small $\infty$-categories, functors and natural transformations. We write $\Cat_{\infty}^*$ for the \infcat  of small pointed $\infty$-categories and point-preserving functors.

\item For functor categories \(\mathrm{Fun}(\mathcal{C},\mathcal{D})\) we use superscripts \text{lex}, \text{rex}, \text{ex}, \text{lim} and \text{colim} to indicate the full subcategories of functors preserving finite limits, finite colimits, finite limits and finite colimits (exact functors), small limits, and small colimits, respectively.

Superscripts \(L\) and \(R\) single out functors admitting left and right adjoints, respectively. When \(\mathcal{C}\) and \(\mathcal{D}\) carry symmetric monoidal structures, the superscript \(\otimes\) denotes the \(\infty\)-category of functors equipped with a strong symmetric monoidal structure with morphisms given by strongly symmetric monoidal natural transformations.

\item We write \(\prl\) for the \(\infty\)-category of presentable \(\infty\)-categories with left adjoint functors. Its default symmetric monoidal structure is the Lurie tensor product, which corepresents functors out of the Cartesian product that are colimit-preserving in each variable. 

For a regular cardinal \(\kappa\), we use \(\prl_{\kappa}\subseteq\prl\) to denote the (non-full) subcategory of \(\kappa\)-accessible presentable \(\infty\)-categories with functors that preserve \(\kappa\)-compact objects.

\item  We use the superscript \(\cotimes\) to denote a \syminfcat. If $\cotimes$ is Cartesian monoidal, we will alternatively denote it by $\mcc^\times$. The $\infty$-category of commutative monoids in a symmetric monoidal $\infty$-category $\cotimes$ will be denoted by $\mathrm{CAlg}(\mathcal{C})$ and we refer to its objects as commutative algebras or \ein-algebras in $\mathcal{C}$. For $\mathcal{C}=\mathrm{Sp}$ equipped with its natural symmetric monoidal structure, we usually say $\mathbb{E}_{\infty}$-ring spectrum or $\mathbb{E}_{\infty}$-ring instead of commutative algebra.
\item  A symmetric monoidal presentable $\infty$-category $\mathcal{C}^\otimes$ is called presentably symmetric monoidal if the monoidal structure $\otimes$ preserves colimits separately in each variable.
\end{enumerate}

\subsection*{Acknowledgments}

We would like to thank Tobias Barthel, Martin Gallauer, David Gepner, Piotr Pstr\k{a}gowski, Maxime Ramzi, Gabriel Ribeiro, Beren Sanders, Vova Sosnilo, Germ\'an Stefanich, and Greg Stevenson for many helpful conversations at various stages of this project. We are especially grateful to Beren Sanders for his detailed feedback on an earlier draft of this manuscript.

\numberwithin{thm}{subsection}
\section{Categorical Ideals and Zariski Frames}\label{sec:zar}

\subsection{Spectra of preframes}
In this section, we introduce the spectrum $\spec(Q)$ of a preframe $Q$. As a set, $\spec(Q)$ consists of the prime elements of $Q$; see \cref{def:prime} below. We show that these prime elements are in bijection with the frame-theoretic points of $Q_{\op{rad}}$, the frame obtained by taking radical elements of $Q$ (\cref{thm:specQ}). The main reference for this section is \cite[Appendix A]{anel2023left}. We also refer the reader to \cref{app:stone-duality} for basic concepts in pointless topology.

\begin{de}
A (small) poset is said to be a \textbf{suplattice} if it admits all joins. By the adjoint functor theorem, a suplattice also has all meets and is thus a complete lattice. A morphism of suplattices is a map preserving joins.
\end{de}

\begin{de}
A \textbf{tensor-suplattice} (or \textbf{$\otimes$-suplattices}) is a suplattice $Q$ equipped with a commutative monoidal structure $(Q, \star, 1)$ such that the product $\star \colon Q \times Q \to Q$ preserves joins in each variable. A map $\phi \colon Q \to Q'$ of $\otimes$-suplattices is a morphism if it is both a morphism of suplattices and a morphism of monoids.
\end{de}

\begin{de}\label{de:preframe}
A \textbf{preframe}\footnote{Preframes are called $\otimes$-frames in \cite{anel2023left}. Note that our preframe is different from the ``preframe'' that appeared in \cite[Definition 3.3]{aoki2023sheaves}. There a ``preframe'' refers to a $\otimes$-poset in which the tensor distributes with directed suprema.} is a $\otimes$-suplattice $(Q, \star, 1)$ where the unit $1$ is the top element of $Q$. A map between preframes is a morphism if it is a morphism of $\otimes$-suplattices. We denote by $\pfrm$ the category of preframes.
\end{de}

\begin{rem}
Recall that a frame is a complete lattice for which finite meets distribute over arbitrary joins. A morphism of frames is a map preserving arbitrary joins and finite meets. We denote by $\frm$ the category of frames. A frame has a natural structure of a $\otimes$-suplattice with $x \star y \coloneq x \wedge y$, and with the top element serving as the unit. Therefore, any frame is naturally a preframe. We thus have a fully faithful embedding
\[
\frm \hookrightarrow \pfrm.
\]
\end{rem}
\begin{rem}[Categorification]\label{rem:categorification}
Let $\prl_{0}$ denote the full subcategory spanned by the presentable $0$-categories. Recall that a $0$-category is an $\infty$-category whose mapping spaces are either empty or contractible; equivalently, it is a poset. Under this identification, presentable $0$-categories correspond to suplattices, and colimit-preserving functors correspond to maps preserving arbitrary joins. Hence there are natural equivalences
\[
\begin{aligned}
\prl_{0} &\simeq \mathsf{SupLat} \\
\calg(\prl_{0}) &\simeq \otimes\text{-}\mathsf{SupLat} \\
\calg(\prl_0)_{\op{ter}} &\simeq \pfrm \\
\calg(\prl_0)_{\op{Car}} &\simeq \frm.
\end{aligned}
\]
Here the subscripts $\op{ter}$ and $\op{Car}$ indicate, respectively, that the monoidal unit is terminal and that the monoidal structure is Cartesian. Under the final equivalence, compactly generated Cartesian monoidal 0-categories correspond to coherent frames:
\begin{equation}\label{idencohfrm}
\calg(\prl_{\omega,0})_{\op{Car}}\simeq \frmcoh.
\end{equation}
\end{rem}

\begin{rem}
A preframe might not be a frame. For example, the poset of ideals of a commutative ring $R$ is a preframe but might not be a frame: Take $R=\mathbb{Z}$. In the ideal preframe, the monoidal product is ideal multiplication. If it were a frame with this monoidal structure then multiplication would have to coincide with meet, and in particular every ideal would be idempotent. However, $(2)^2=(4)\neq(2)$. Nevertheless, the radical elements of any preframe form a frame, as shown below.
\end{rem}

\begin{de}\label{def:rad}
Let $Q$ be a preframe. We call an element $r \in Q$ \textbf{radical} if $x^2 \le r$ implies $x \le r$ for every $x \in Q$. We write $Q_{\op{rad}}$ for the poset of radical elements of $Q$. 
\end{de}

\begin{prop}\label{prop:sqrtQ-reflection}
Let $Q$ be a preframe. The radicalization map $\sqrt{-}\colon Q \to Q_{\op{rad}}$ exhibits $Q_{\op{rad}}$ as the localization of $Q$ in $\pfrm$ at the family of morphisms $\SET{x^2\le x}{x\in Q}$. Equivalently, precomposition with $\sqrt{-}$ induces a bijection
\[
\Hom_{\pfrm}(Q_{\op{rad}},Q') \xrightarrow{\sim} \SET{f\in\Hom_{\pfrm}(Q,Q')}{f(x^2)=f(x)\text{ for every }x\in Q}
\]
for every preframe $Q'$. The monoidal product on $Q_{\op{rad}}$ is given by $\sqrt{x_1} \cdot \sqrt{x_2} \coloneq \sqrt{x_1 \star x_2}$.
\end{prop}

\begin{proof}
See \cite[Proposition~A.4.3 and Proposition~A.2.1]{anel2023left}.
\end{proof}

\begin{rem}\label{operationsqrt}
The radical of any element $x \in Q$ can be computed as
\[
\sqrt x = \bigwedge\SET{y\in Q}{y\text{ is radical and }y\ge x}.
\]
Since the radical elements are the local objects for the relations $y^2 \leq y$, we can also construct $\sqrt{x}$ by a small object argument. Let us put
$$
\sqrt[1]{x}=\sup \left\{y \mid \exists k\colon y^k \leq x\right\}
$$
and $\sqrt[n+1]{x}=\sqrt[1]{\sqrt[n]{x}}$. Then the transfinite colimit of $x \leq \sqrt[1]{x} \leq \sqrt[2]{x} \leq \ldots$ is $\sqrt{x}$. When $Q$ is $\omega$-coherent in the sense of \cref{defkcoh} below, we have $\sqrt[1]{x}=\sqrt{x}$.
\end{rem}

\begin{thm}\label{thm:sqrtQ-frame}
Let $Q$ be a preframe. The monoidal product on $Q_{\op{rad}}$ coincides with the meet. In other words, we have $y_1 \cdot y_2 = y_1 \wedge y_2$ for all $y_1, y_2 \in Q_{\op{rad}}$. Consequently, $Q_{\op{rad}}$ is a frame.
\end{thm}

\begin{proof}
It suffices to show that
\[
y \le y_1\cdot y_2 \iff y\le y_1 \text{ and } y\le y_2
\]
for all $y, y_1, y_2 \in Q_{\op{rad}}$.
Since $y_1\cdot y_2 \le y_1\cdot 1 = y_1$ and $y_1\cdot y_2 \le 1\cdot y_2=y_2$, we have
\[
y \le y_1\cdot y_2 \implies y\le y_1 \text{ and } y\le y_2.
\]
For the other direction, suppose that $y\le y_1 \text{ and } y\le y_2$. Then $y\cdot y \le y_1\cdot y \le y_1\cdot y_2$. By \cref{prop:sqrtQ-reflection}, every element in $Q_{\op{rad}}$ is idempotent. Hence, $y = y\cdot y \le y_1\cdot y_2$. Finally, to see that $Q_{\op{rad}}$ is a frame, note that now $y\wedge-=y\cdot-$ preserves arbitrary joins for any $y\in Q_{\op{rad}}$, since $Q_{\op{rad}}$ is a $\otimes$-suplattice.
\end{proof}

\begin{cor}\label{cor:preframe}
The assignment $Q \mapsto Q_{\op{rad}}$ is left adjoint to the inclusion $\frm \hookrightarrow \pfrm$.
\end{cor}

\begin{proof}
This follows from \cref{prop:sqrtQ-reflection} and \cref{thm:sqrtQ-frame}.
\end{proof}

\begin{de}\label{def:prime}
Let $Q$ be a preframe. An element $p \in Q$ is called a \textbf{prime} if $p\neq1$ and $x_1 \star x_2 \le p$ implies either $x_1 \le p$ or $x_2 \le p$, for all $x_1, x_2 \in Q$. We write by $\spec(Q)$ the set of prime elements of $Q$. It will be given a topology in \cref{def:spec-q}.
\end{de}

\begin{rem}
Recall that a frame-theoretic point of a frame $F$ is a frame morphism $f\colon F \to \{0,1\}$. The set $\op{pt}(F)$ of points of $F$ admits a topology with open subsets given by $$U_y \coloneqq \SET{f\in \op{pt}(F)}{f(y)=1}$$
ranging over the elements $y$ of $F$. The assignment $F\mapsto\op{pt}(F)$ induces a contravariant functor from frames to topological spaces. See \cref{app:stone-duality} for more detail.
\end{rem}

\begin{thm}\label{thm:specQ}
Let $Q$ be a preframe. We have a bijection
\[
\op{pt}(Q_{\op{rad}}) \xra{\sim} \spec(Q)
\]
sending a frame-theoretic point $f\colon Q_{\op{rad}} \to \{0,1\}$ to the element $p_f \coloneq \bigvee f^{-1}(0) \in \spec(Q)$. The inverse sends a prime element $p$ to the frame theoretic-point $f_p\colon Q_{\op{rad}} \to \{0,1\}$ defined by
\[
f_p(y) =
\begin{cases}
0 & \text{if } y \le p \\
1 & \text{if } y \not\le p.
\end{cases}
\]
\end{thm}

\begin{proof}
Let $p$ be a prime element of $Q$. We first check that the map $f_p$ is a morphism of frames \ie preserves arbitrary joins and finite meets. Let $\{y_i\}_i$ be a set of elements in $Q_{\op{rad}}$. If $\bigvee_{i}y_i \le p$ then $y_i \le p$ for all $i$ and hence $\bigvee_if_p(y_i) = \bigvee_i0 = 0 = f_p(\bigvee_iy_i)$. If $\bigvee_{i}y_i \not\le p$ then $y_j \not\le p$ for some $j$ and hence $1 = f_p(y_j) \le \bigvee_if_p(y_i)$, so $\bigvee_if_p(y_i) = 1 = f_p(\bigvee_iy_i)$. This proves that $f_p$ preserves joins. To see that it preserves finite meets, it suffices to show
\[
\sqrt{x_1} \wedge \sqrt{x_2} \le p \iff \sqrt{x_1} \le p \text{ or } \sqrt{x_2} \le p
\]
for all $\sqrt{x_1}, \sqrt{x_2} \in Q_{\op{rad}}$. The direction ($\impliedby$) is immediate. For the ($\implies$) direction, note that
\[
\sqrt{x_1} \wedge \sqrt{x_2} = \sqrt{x_1}\cdot\sqrt{x_2} = \sqrt{x_1\star x_2}
\]
by \cref{prop:sqrtQ-reflection} and \cref{thm:sqrtQ-frame}.
Thus, if $\sqrt{x_1} \wedge \sqrt{x_2} \le p$ then $x_1\star x_2 \leq \sqrt{x_1\star x_2} \le p$. Hence, $x_1\le p$ or $x_2\le p$ since $p$ is a prime element. As prime elements are radical, we obtain $\sqrt{x_1} \le p$ or $\sqrt{x_2} \le p$. This verifies that $f_p$ preserves finite meets. Therefore, $f_p \in \op{pt}(Q_{\op{rad}})$.

Let $f$ be an arbitrary frame-theoretic point of $Q_{\op{rad}}$. We will show that $p_f$ is a prime element of $Q$. We first claim that $p_f\neq1$. Indeed, if $p_f=1$ then $f(1)=f(p_f)=f(\bigvee f^{-1}(0))=0$, which contradicts the fact that frame maps preserve top elements. Hence, the claim follows. Now assume that $x_1 \star x_2 \le p_f$ for $x_1, x_2 \in Q$. By definition, $p_f$ is radical, so $\sqrt{x_1} \wedge \sqrt{x_2} = \sqrt{x_1 \star x_2} \le p_f$. Since $f$ preserves arbitrary joins and finite meets, we have $f(p_f)=0$ and thus $f(\sqrt{x_1})\wedge f(\sqrt{x_2}) = f(\sqrt{x_1}\wedge \sqrt{x_2}) = 0$. Hence, either $f(\sqrt{x_1})=0$ or $f(\sqrt{x_2})=0$. We therefore have either $x_1 \le \sqrt{x_1} \le p_f$ or $x_2 \le \sqrt{x_2} \le p_f$. This establishes $p_f\in \spec(Q)$.

It remains to verify that the maps $p_{(-)}$ and $f_{(-)}$ are inverse to each other. Let $f \in \op{pt}(Q_{\op{rad}})$. Since $f(p_f)=0$, we have
\[
y \le p_f \iff f(y)=0
\]
for all $y\in Q_{\op{rad}}$. That is, $f_{p_{f}}=f$. Now, let $p\in\spec(Q)$. We have $p_{f_p} \le p$, since $f_p(p_{f_p})=0$. We also have $p\le p_{f_p}$, since $f_p(p)=0$. Therefore, $p_{f_p} = p$.
\end{proof}

\begin{de}\label{def:spec-q}
We equip $\spec(Q)$ with the topology whose open subsets are the sets
\[
U_r \coloneqq \SET{p\in \spec(Q)}{r\not\le p}
\]
ranging over all $r \in Q_{\op{rad}}$. We call $\spec(Q)$ the \textbf{spectrum} of $Q$.
\end{de}

\begin{rem}
In the case where $Q$ is already a frame (that is, $Q_{\op{rad}}=Q$ by \cref{cor:preframe}), \cref{thm:specQ} states that $\op{pt}(Q)\cong\spec(Q)$. This recovers \cite[Proposition~II.1.3]{Johnstone82}.
\end{rem}
\begin{de}
Let $Q$ be a $\otimes$-suplattice. By definition, for each $a\in Q$, the map $a \star(-): Q \rightarrow Q$ preserves joins:
$$
a \star\left(\bigvee_{x \in S} x\right)=\bigvee_{x \in S} (a \star x)
$$
for every subset $S \subseteq Q$. It follows that the map $a \star(-): Q \rightarrow Q$ has a right adjoint, which is denoted by $a \backslash(-): Q \rightarrow Q$, and called the \textbf{division} by $a$. For every $a, b, c \in Q$ we have
$$
a \star b \leq c \iff b \leq a \backslash c .
$$
\end{de}
\begin{prop}\label{homradical}
If  $Q$ is a preframe and $y\in Q$ is radical, then so is $x\backslash y$ for any $x\in Q$.
\end{prop}
\begin{proof}
Let $a\in Q$ be an element such that $a^n \leq x\backslash y$ for some $n\geq1$. We need to show that $a \in x\backslash y$. By adjunction, we have $a^n\star x\leq y$ and hence $(a\star x)^n=a^n\star x^n\leq y$. Since $y$ is radical, we get $a \in x\backslash y$.
\end{proof}
\begin{de}
Let $Q$ be a preframe. We define $Q^{\op{cidem}} \subseteq Q$ to be the subposet spanned by (co)idempotent elements, i.e. the elements $x$ such that $x^2=x$.
\end{de}
\begin{prop}
The poset $Q^{\op{cidem}}$ is a sub-preframe of $Q$. Moreover, $Q^{\op{cidem}}$ is a frame.
\end{prop}
\begin{proof}
The poset $Q^{\op{cidem}}$ is closed under the monoidal product since we have $(a \star b)^2=a^2 \star b^2=a \star b$ for all $a, b \in Q$. For any family of elements $a_i \in Q$, we have
$$
\bigvee a_i \geq\left(\bigvee a_i\right)^2=\bigvee a_i \star a_j \geq \bigvee a_i^2=\bigvee a_i
$$
and thus $\bigvee a_i=\left(\bigvee a_i\right)^2$. This shows that $Q^{\op{cidem}}$ is closed under small joins. Therefore, $Q^{\op{cidem}}$ is a sub-preframe of $Q$. It then follows from \cref{prop:sqrtQ-reflection} that $Q^{\op{cidem}}$ is a frame, since every element of $Q^{\op{cidem}}$ is coidempotent by definition.
\end{proof}
\begin{prop}[{\cite[Proposition A.4.8]{anel2023left}}]
The inclusion $Q^{\op{cidem}} \rightarrow Q$ is the universal coreflection of $Q$ into frames. 
\end{prop}
\begin{cor}\label{prefrmadjoints}
From the discussion above, we obtain the following adjunctions:
$$\begin{tikzcd}
\pfrm \arrow[r, "(-)_{\op{rad}}", shift left=2ex] \arrow[r, "(-)^{\op{cidem}}"', shift right=2ex]  & \arrow[l,"\perp","\perp"', hook'] \frm
\end{tikzcd}$$
or alternatively
$$\begin{tikzcd}
\calg(\prl_0)_{\op{ter}} \arrow[r, "(-)_{\op{rad}}", shift left=2ex] \arrow[r, "\cidem(-)"', shift right=2ex]  & \arrow[l,"\perp","\perp"', hook'] \calg(\prl_0)_{\op{Car}}
\end{tikzcd} .$$
\end{cor}

\begin{rem}
Since any frame is canonically a preframe whose monoidal product is given by the meet, the following generalizes the definition of $\kappa$-coherence of frames introduced in {\cite{krause2023spectrum}}.
\end{rem}

\begin{de}\label{defkcoh}
Let $\kappa$ be a regular cardinal. A preframe $Q$ is called \textbf{$\kappa$-coherent}, if the following hold:
\enu{
\item Every element is the join of $\kappa$-compact elements;
\item The greatest element is $\kappa$-compact;
\item The product of two $\kappa$-compact elements is $\kappa$-compact.
}
In the case when $\kappa = \omega$, we simply say that the preframe $Q$ is \textbf{coherent}.
\end{de}

\begin{rem}\label{rem:kcoh}
By definition, a preframe $Q$ is $\kappa$-coherent if and only if $Q$ regarded as a presentably symmetric monoidal $0$-category is $\kappa$-presentable. Similarly, a frame $F$ is $\kappa$-coherent if and only if the cartesian symmetric monoidal $\infty$-category $F^\times $ lies in $ \calg(\prl_{\kappa,0})$.
\end{rem}

\begin{rem}
When a preframe $Q$ is coherent, the spectrum $\spec(Q)$ is defined and studied in \cite{BuanKrauseSolberg07}. We remind the reader that a coherent preframe is called an ideal lattice in \cite{BuanKrauseSolberg07}.
\end{rem}

\subsection{Zariski frames}\label{sec2.1}
In this section, we construct the Zariski frame functor. Throughout, we fix a presentably symmetric monoidal \infcat $\cotimes\in\calg(\prl)$. We also fix a regular cardinal $\kappa$ such that $\cotimes\in\calg(\prl_\kappa)$. There always exists such a $\kappa$ by \cite[Lemma 5.3.2.12]{ha}.

\begin{de}\label{def:ideal}
We say that a morphism $i: I\to\mb{1}$ is a \textbf{categorical ideal} of $\mcc$ if $i$ is a monomorphism, \ie a $(-1)$-truncated morphism in $\mc{C}$. We often denote a categorical ideal $i: I\to\mb{1}$ by $I$ and simply call it an ideal. We write $\op{Idl}(\mcc)$ for the \infcat of ideals of $\mcc$.
\end{de}

\begin{rem}
By \cite[Remark~5.5.6.10]{htt}, a morphism $i: I\to\mb{1}$ is a $(-1)$-truncated morphism in $\mcc$ if and only if it is $(-1)$-truncated as an object in the overcatgory $\mcc_{/\mb{1}}$. Hence, $\op{Idl}(\mc{C})$ can be identified with the truncated overcategory $(\mc{C}_{/\mb{1}})_{\leq-1}$. By \cite[Remark~5.5.6.18]{htt}, the inclusion $\idlc\hookrightarrow \mc{C}_{/\mb{1}}$ admits a left adjoint which we denote by $\im$. Moreover, $\op{Idl}(\mc{C})$ inherits a symmetric monoidal structure from $\mc C^\otimes$, and the monoidal product by is given by $(I,J)\mapsto I\cdot J \coloneqq \im(I\otimes J)$; see \cite[\textsection2.2.2 and 4.8.2.15]{ha}. In particular, $\idlc^\otimes$ is a presentably symmetric monoidal $0$-category, \ie a $\otimes$-suplattice.
\end{rem}


\begin{rem}
Our next goal is to prove that $\idlc^\otimes$ is a $\kappa$-coherent preframe. We start by showing that $\mc{C}_{/\mb{1}}$ is $\kappa$-presentable as a special case of the following lemma.
\end{rem}

\begin{lem}\label{overcataccessible}
Let $\mcd\in\prl_{\kappa}$ be $\kappa$-presentable \infcat and $X\in\mcd$. Then $\mc{D}_{/X}\in\prl_{\kappa}$, and the forgetful functor $F:\mc{D}_{/X}\to \mc{D}$ preserves and reflects $\kappa$-compact objects. Moreover, if $S=\{X_\alpha\}$ is a generating set 
for \mcd, then the collection $S_{/X}=\{X_\alpha\to X\}$ of all maps $A\to X$, with $A\in S$, generates~$\mc{D}_{/X}$.
\end{lem}

\begin{proof}
Since the right adjoint $G$ to $F:\mc{D}_{/X}\to \mc{D}$ is given by $Y\to Y\times X$ and $\kappa$-filtered colimits commute with $\kappa$-small limits in \mcd, it follows that $G$ preserves $\kappa$-filtered colimits and hence $F$ preserves $\kappa$-compact objects. Now given a map $Z\to X$ in \mcd, consider the following pullback diagram of spaces:
$$
\begin{tikzcd}
{\mapp_{\mcd_{/X}}(Z,-)} \arrow[d] \arrow[r] & {\mapp_{\mcd}(Z,-)} \arrow[d] \\
* \arrow[r]                                  & {\mapp_{\mcd}(Z,X)}.       
\end{tikzcd}$$
Then $F$ reflects $\kappa$-compact objects because filtered colimits commute with finite limits in \mcs. However, for any object $Z\in \mcd_{/X}$, $Z$ can be written as a $\kappa$-filtered colimit of $\kappa$-compact objects in \mcd. It is therefore a $\kappa$-filtered colimit of $\kappa$-compact objects in $\mcd_{/X}$. We conclude that $\mc{D}_{/X}$ is $\kappa$-presentable. For the last statement, let $S=\{X_\alpha\}$ be a generating set 
for \mcd. It suffices to show that $$\{h_f=\mapp_{\mcd_{/X}}(X_\alpha,-): \mcd_{/X}\to\mcs \,|\,f\in S_{/X} \}$$ is jointly conservative. This follows from the pullback diagram
$$
\begin{tikzcd}
{\mapp_{\mcd_{/X}}(X_\alpha,-)} \arrow[d] \arrow[r] & {\mapp_{\mcd}(X_\alpha,-)} \arrow[d] \\
* \arrow[r,"f"]                                  & {\mapp_{\mcd}(X_\alpha,X)}   
\end{tikzcd}$$
for all $f\in S_{/X}$.
\end{proof}

\begin{prop}\label{idlccompgene}
The category $\op{Idl}(\mc{C})^\otimes$ is a commutative algebra object of $\prl_{\kappa,0}$, the $\infty$-category of $\kappa$-presentable $0$-categories. Moreover, an ideal $I$ is $\kappa$-compact in $\idlc$ if and only if there exists $\kappa$-compact object $X\in\mcc^\kappa$ and a map $X\to \mb{1}$ such that $\op{Im}(X)=I$.
\end{prop}

\begin{proof}
The inclusion $\cidl\to \mcc_{/\mb{1}}$ is closed under $\kappa$-filtered colimits by \cite[Lemma 5.5.6.15]{htt}. Thus, $\im \colon \mcc_{/\mb{1}}\to \cidl$ is a $\kappa$-accessible (monoidal) localization so $\idlcotimes\in\calg(\prl_{\kappa,0})$. Now given a $\kappa$-compact object $I$ in $\idlc$, by \cref{overcataccessible} we can write $I$ as a $\kappa$-filtered colimit of $\kappa$-compact objects $X_\alpha\to \mb{1}$ in $\mcc_{/\mb{1}}$. Thus $I \simeq \im(I) \simeq \im(\varinjlim X_\alpha) \simeq\varinjlim \im(X_\alpha)$. By $\kappa$-compactness, $I$ is a retract of some $\im(X_\alpha)$. Since retractions are equivalences in any $0$-category, $I=\im(X_\alpha)$.
\end{proof}

\begin{rem}\label{rem:generatingset}
In fact, if $S\subseteq \mcc$ is a generating set for \mcc, then $\im(S_{/\mb{1}})$ is a generating set of $\idlc$. Moreover, every ideal $I\in \idlc$ can be written as $I=\bigvee_\alpha K_\alpha$ where each $K_\alpha\in \im(S_{/\mb{1}})$, since one can check that such $I$ are closed under arbitrary joins. 
\end{rem}

\begin{cor}\label{cor:idlc-preframe}
The category $\op{Idl}(\mc{C})^\otimes$, regarded as a $\otimes$-suplattice, is a preframe. Furthermore, it is $\kappa$-coherent in the sense of \cref{defkcoh}.
\end{cor}
\begin{proof}
The first statement follows immediately from the clear fact that the unit in $\op{Idl}(\mc{C})^\otimes$ is terminal. The second statement is due to \cref{idlccompgene}, keeping in mind \cref{rem:categorification}.
\end{proof}

\begin{de}\label{def:radprimeideal}
Let $I\subseteq\mb{1}$ be an ideal of $\mc{C}$.
\enu{
\item We say $I$ is a \textbf{radical ideal} if $I$ is a radical element in $\idlc$ in the sense of \cref{def:rad}, that is, for any ideal $J\subseteq\mb{1}$ such that $J^n\subseteq I$ for some $n\geq 1$, we have $J\subseteq I$. 
\item We say $I$ is a \textbf{prime ideal} if $I$ is a prime element in $\idlc$ in the sense of \cref{def:prime}, that is, for any ideals $J_1, J_2 \subseteq\mb{1}$ with $J_1 J_2\subseteq I$, we have either $J_1 \subseteq I$ or $J_2 \subseteq I$.
\item For any set of ideals $\{I_\alpha\}$, we denote their meet in $\idlc$ by $\bigcap_\alpha I_\alpha$. The \textbf{nilpotent radical} $\sqrt{I}$ of an ideal $I \in \idlc$ is defined to be the meet of all radical ideals containing $I$ (see \cref{operationsqrt}):
\[
\sqrt I \coloneqq \bigcap\SET{J\in \idlc}{J\text{ is radical and }I \subseteq J}.
\]
}

\end{de}

\begin{prop}\label{prop:idlcrad-localization}
The full subcategory $\idlc_{\op{rad}}$ spanned by the radical ideals is a reflective subcategory of $\idlc$ and the corresponding localization is given by $I\mapsto \sqrt{I}$. Furthermore, the localization is symmetric monoidal, \ie we have $\sqrt{I\cdot J}=\sqrt{\sqrt{I}\cdot \sqrt{J}}$ for any ideals $I$ and $J$. It thereby induces a symmetric monoidal adjunction 
$$\begin{tikzcd}
\idlc^\otimes \arrow[r, "\sqrt{-}", shift left=1ex]   & \arrow[l,"\perp"', hook', shift left=1ex] \idlc_{\op{rad}}^\otimes
\end{tikzcd} .$$
\end{prop}

\begin{proof}
This follows from \cref{prop:sqrtQ-reflection} and \cref{operationsqrt} by \cref{cor:idlc-preframe}. A direct argument can be given as follows: The map $I\mapsto \sqrt{I}$ gives a reflection by definition. It then suffices to prove that for any ideals $I,J$ in $\mcc$ the natural map $\sqrt{I\cdot J}\to \sqrt{I\cdot \sqrt{J}}$ is an equivalence in $\idlc_{\op{rad}}$. It remains to show that any radical ideal $K$ containing $I\cdot J$ also contains $I\cdot \sqrt{J}$. This is a consequence of \cref{homradical}.
\end{proof}


\begin{thm}
The category $\idlc_{\op{rad}}^\otimes$ is cartesian monoidal and hence a frame, \ie we have $\sqrt{I\cdot J}=\sqrt{I}\cap\sqrt{J}$ for all radical ideals $I,J$.
\end{thm}
\begin{proof}
This can be deduced directly from \cref{thm:sqrtQ-frame}. We also give a straightforward argument: The inclusion $\sqrt{I\cdot J}\subseteq\sqrt{I}\cap\sqrt{J}$ is obvious. The reverse inclusion follows from the observation
\[
(\sqrt{I}\cap\sqrt{J})^2\subseteq \sqrt{\sqrt{I}\cdot \sqrt{J}}= \sqrt{I\cdot J}.\qedhere
\]
\end{proof}

\begin{de}\label{de:zariski-frame}
We call the frame $\zar(\mc C) \coloneqq \idlc_{\op{rad}}$ of radical ideals the \textbf{Zariski frame of $\mc{C}^\otimes$}. We call the topological space $\spec(\mc C) \coloneqq \spec(\op{Idl}(\mc{C}))$ the \textbf{Zariski spectrum of $\mc C^\otimes$}.
\end{de}

\begin{rem}\label{rem:spec-pt-zar}
By \cref{thm:specQ} we have
\[
\spec(\mc C) \cong \op{pt}(\op{Zar}(\mc{C})).
\]
\end{rem}

\begin{rem}
The Zariski frame construction $\mcc^\otimes\mapsto \zar(\mc C)$ induces a functor
\[
\op{Zar}\colon \calg(\prl)\to \frm
\]
by the following result.
\end{rem}

\begin{prop}\label{functorial}
Let $F^\otimes: \mcc^\otimes\to\mcd^\otimes$ be a morphism in $\calg(\prl)$. Then there exist unique  colimit-preserving symmetric monoidal functors (the dashed arrows below) making the following diagram commute:
$$
\begin{tikzcd}
\mcc_{/\mb{1}_\mcc}^\otimes \arrow[r, "F/_{\mb1}^\otimes"] \arrow[d, "\im"']                                 & \mcd_{/\mb{1}_\mcd}^\otimes \arrow[d, "\im"] \\
\idl(\mc{C})^\otimes \arrow[r, "\idl(F)^\otimes", dashed] \arrow[d, "\sqrt{-}"']             & \idl(\mc{D})^\otimes \arrow[d, "\sqrt{-}"]        \\
\zar(\mcc) \arrow[r, "\zar(F)", dashed] & \zar(\mcd)
\end{tikzcd}$$
where $F/_{\mb1}^\otimes$ is the colimit-preserving symmetric monoidal functor induced by $F^\otimes$.
\end{prop}
\begin{proof}
Let $F^R$ be the right adjoint of $F$. The right adjoint $(F/_{\mb1})^R$ of $F/_{\mb1}$ sends $Y\to\mb1_{\mcd}$ to
\[
F^R(-)\times_{F^R(\mb{1}_\mcd)}\mb{1}_\mcc.
\]
Since right adjoints preserve limits and monomorphisms are stable under pullback, $(F/_{\mb1})^R$ carries ideals of $\mcd$ to ideals of $\mcc$. It therefore restricts to a functor
\[
\idl(\mc{D})^\otimes \to \idl(\mc{C})^\otimes.
\]
Taking its left adjoint gives the required functor
\[
\idl(F)^\otimes\colon \idl(\mc{C})^\otimes \to \idl(\mc{D})^\otimes.
\]
Explicitly,
\[
\idl(F)^\otimes(I\to\mb1_{\mcc}) = \im(F(I)\to\mb1_{\mcd}).
\]
By construction, this functor makes the upper square commute. It is colimit-preserving and symmetric monoidal since $F/_{\mb1}$ and $\im$ are so. Finally, the universal property of the radicalization (\cref{prop:sqrtQ-reflection}) induces a unique frame morphism
\[
\zar(F) \colon \zar(\mcc) \to \zar(\mcd)
\]
satisfying
\[
\zar(F)(\sqrt{I})=\sqrt{\idl(F)(I)}.
\]
This gives the lower dashed arrow and proves uniqueness.
\end{proof}


\begin{lem}\label{lem:rad-criterion}
Let $S\subseteq \mcc$ be a generating set of \mcc. An ideal $I\subseteq\mb{1}$ is radical if and only if the following condition holds: whenever
\[
J = \im(X\to\mb1)
\]
for some $X\in S$ and $J^n\subseteq I$ for some $n\geq 1$, we have $J\subseteq I$.
\end{lem}

\begin{proof}
The forward implication is immediate. Conversely, let $K$ be an ideal such that $K^n\subseteq I$ for some $n\geq1$. By \cref{rem:generatingset} we may write
\[
K = \bigvee_\alpha K_\alpha
\]
where each $K_\alpha = \im(X_\alpha\to\mb1)$ for some $X_\alpha\in S$, and $K_\alpha\subseteq K$. It follows that
\[
K_\alpha^n\subseteq K^n\subseteq I.
\]
By hypothesis, $K_\alpha\subseteq I$ for every $\alpha$. Taking their join gives $K\subseteq I$, so $I$ is radical.
\end{proof}

\begin{prop}\label{radidprkl}
The category $\zar(\mc C)=\idlc_{\op{rad}}^\times$ a commutative algebra object of $\prl_{\kappa,0}$ and the functor $\sqrt{-} \colon \idlc \to \zar(\mcc)$ is a morphism in $ \calg(\prl_{\kappa,0})$. Moreover, a radical ideal $I$ is $\kappa$-compact in $\idlc_{\op{rad}}$ if and only if there exists a $\kappa$-compact ideal $J\in \idlc^\kappa$ such that $\sqrt{J}=I$.
\end{prop}

\begin{proof}
For the first statement, note that $\idlcotimes\to \idlc_{\op{rad}}^\times$ is a symmetric monoidal localization by \cref{prop:idlcrad-localization} and we have $\idlcotimes \in \calg(\prl_{\kappa})$ by \cref{idlccompgene}. It then suffices to show that the inclusion $\idlc_{\op{rad}}\subseteq \idlc$ is closed under $\kappa$-filtered colimits. This follows from \cref{lem:rad-criterion} by taking $S=\mc C^\kappa$. For the ``only if'' part of last statement, observe that any $\kappa$-compact object $I$ in $\idlc_{\op{rad}}$ can be written as a $\kappa$-filtered colimit of $\kappa$-compact ideals $I_\alpha$ in $\idlc$. Thus $I \simeq \sqrt I \simeq \sqrt{\varinjlim I_\alpha} \simeq\varinjlim \sqrt{I_\alpha}$. It then follows from $\kappa$-compactness that $I$ is a retract of some $\sqrt{I_\alpha}$. However, in a $0$-category, any retraction is an equivalence, therefore $I=\sqrt{I_\alpha}$. The ``if'' part is clear since we have established that $\sqrt{-}$ is a morphism in $\prl_{\kappa,0}$.
\end{proof}

\begin{cor}\label{kcohfrm}
For $\cotimes\in\calg(\prl_\kappa)$, the Zariski frame $\zar(\mc C)$ is $\kappa$-coherent.
\end{cor}
\begin{proof}
This follows from \cref{radidprkl} by \cref{rem:kcoh}.
\end{proof}

\subsection{The compactly generated case}\label{sec2.3}
We now restrict our attention to compactly generated categories. Throughout this section, we fix an $\omega$-presentably symmetric monoidal \infcat $\cotimes\in\calg(\prl_{\omega})$.

\begin{de}\label{def:fg-ideal}
Let $I\subseteq\mb{1}$ be an ideal of $\mc{C}$. We say $I$ is \textbf{finitely generated} if it is compact in $\idlc$, or equivalently, if there exists a compact object $X\in\mc{C}^\omega$ and a map $X\to \mb{1}$ such that $I=\op{Im}(X)$.
\end{de}

\begin{rem}
Let $I\subseteq\mb{1}$ be an ideal of $\mc{C}$. By \cref{operationsqrt}, we have $\sqrt{I}=\sqrt[1]{I}= \bigvee_{\alpha\in S}I_\alpha$ where $S\subseteq \idlc$ consists of the finitely generated ideals $I_\alpha$ such that $I^n_\alpha\subseteq I$ for some $n\geq 1$.
\end{rem}

\begin{thm}\label{thm:zar-coh}
For $\cotimes\in\calg(\prl_{\omega})$, the Zariski frame $\op{Zar}(\mc{C})$ is coherent, and hence the Zariski spectrum $\spec(\mc{C})$ is a coherent space.
\end{thm}
\begin{proof}
This follows from \cref{kcohfrm} and \eqref{idencohfrm}.
\end{proof}

\begin{rem}
By \cref{thm:zar-coh}, the functor $\op{Zar}\colon \calg(\prl)\to \frm$ restricts to a functor $$\op{Zar}\colon \calg(\prl_\omega) \to \frmcoh.$$
We thus obtain the following commutative diagram:
\[
\begin{tikzcd}
\calg(\prl_\omega) \arrow[d, hook] \arrow[r, "\op{Zar}"] \arrow[rr, "\spec", bend left] & \frmcoh \arrow[d, hook] \arrow[r, "\op{pt}", "\sim"'] & \topcoh \arrow[d, hook] \\
\calg(\prl) \arrow[r, "\op{Zar}"] \arrow[rr, "\spec"', bend right]                      & \frm \arrow[r, "\op{pt}"]                  & \Top.             
\end{tikzcd}
\]
\end{rem}

\begin{ex}\label{ex:Zariski}
Let $R$ be a commutative ring and consider the ordinary category of $R$-modules $\modu_{R}(\op{Ab})$. An ideal of $\modu_{R}(\op{Ab})$ in the sense of \cref{def:ideal} is exactly the same thing as an ideal of $R$; see \cref{modumono}. The Zariski frame $\op{Zar}(\modu_{R}(\op{Ab}))$ hence recovers the frame of radical ideals of $R$. Therefore, the Zariski spectrum $\spec(\modu_{R}(\op{Ab}))$ recovers the classical Zariski spectrum of $R$; see \cite[Lemma~1.3.1]{KockPitsch17}, for example.
\end{ex}

\begin{ex}\label{ex:Balmer}
Let $\mck$ be a commutative $2$-ring and consider the category of $\mc K$-modules $\modu_{\mc K}(\op{Cat}^{\op{perf}})$. An ideal of $\modu_{\mc K}(\op{Cat}^{\op{perf}})$ amounts to a thick ideal of $\mc K$ in the sense of \cite{Balmer05a}. One can see this by combining \cref{monoepicatperf} and \cref{modumono}. Therefore, the Zariski spectrum $\spec(\modu_{\mc K}(\op{Cat}^{\op{perf}}))$ recovers the Hochster dual of the Balmer spectrum of $\mc K$, by \cite[Corollary~3.4.2]{KockPitsch17}.
\end{ex}


\begin{prop}
The following diagram commutes:
$$	\begin{tikzcd}
\calg(\op{Ab}) \arrow[d] \arrow[rrd, "\spec_{\op{Zar}}"]               &  &                                  \\
\calg(\prl_\omega) \arrow[rr, "\op{Spec}"]                     &  & \topcoh \\
\calg(\op{Cat}^{\op{perf}}) \arrow[u] \arrow[rru, "\spec_{\op{Bal}}"'] &  &                                 
\end{tikzcd}$$
where $\spec_{\op{Zar}}$ denotes the classical Zariski spectrum, $\spec_{\op{Bal}}$ denotes the Hochster dual of the Balmer spectrum, and the two vertical maps send a commutative ring $R$ to $\modu_{R}(\op{Ab})^{\otimes}$ and a commutative $2$-ring $\mc K$ to $\modu_{\mc K}(\op{Cat}^{\op{perf}})^{\otimes}$, respectively.
\end{prop}

\begin{proof}
This follows from \cref{ex:Zariski} and \cref{ex:Balmer}.
\end{proof}


\subsection{Cohen's theorem for coherent frames}
We now prove a frame-theoretic version of Cohen’s theorem for arbitrary coherent frames. Applying it to the coherent Zariski frames of \cref{thm:zar-coh} gives criteria for the Noetherianity of Zariski spectra. We refer the reader to \cref{app:stone-duality} and \cite{DickmannSchwartzTressl19} for any unfamiliar terminology in this section.

\begin{lem}\label{lem:cohen}
Let $F$ be a coherent frame. An element of $F$ is finite if and only if its corresponding open subset of $\op{pt}(F)$ is constructible.
\end{lem}
\begin{proof}
By \cref{rem:stone-duality}, an element of $F$ is finite if and only if its corresponding open subset is quasi-compact. The result then follows from the general fact that a subset of a spectral space is quasi-compact open if and only if it is constructible and closed under generalization (\cite[Theorem~1.5.4(iii)]{DickmannSchwartzTressl19}).
\end{proof}

\begin{thm}\label{thm:cohen}
Let $F$ be a coherent frame. The following are equivalent:
\enu{
\item The spectral space $\op{pt}(F)$ is Noetherian;
\item Every element of $F$ is finite;
\item Every prime element of $F$ is finite.
}
\end{thm}

\begin{proof}
By \cite[Theorem~8.1.11]{DickmannSchwartzTressl19}, a spectral space is Noetherian if and only if every closed subset is constructible, or equivalently, every open subset is constructible. Therefore, the equivalence (1)$\iff$(2) follows from \cref{lem:cohen}. It remains to show (3)$\implies$(1). Note that the open subset corresponding to a prime element $x$ of $F$ is the complement of the closure of $x$ in $\op{pt}(F)$. Hence, (3) implies that the closure of every point in $\op{pt}(F)$ is constructible, by \cref{lem:cohen}. The latter is equivalent to (1) by \cite[Theorem~8.1.11]{DickmannSchwartzTressl19}.
\end{proof}

\begin{cor}\label{cor:cohen}
For $\cotimes\in\calg(\prl_{\omega})$, the following are equivalent:
\enu{
\item The spectral space $\spec(\mc C)$ is Noetherian;
\item Every radical ideal of $\mc C$ is the radical of some finitely generated ideal;
\item Every prime ideal of $\mc C$ is the radical of some finitely generated ideal.
}
\end{cor}

\begin{proof}
Apply \cref{thm:cohen} to the coherent frame $\op{Zar}(\mc{C})$. Keep in mind that $\spec(\mc C) \cong \op{pt}(\op{Zar}(\mc{C}))$ by \cref{rem:spec-pt-zar}.
\end{proof}



\begin{ex}
Specializing to $\mc{C}^\otimes=\modu_{\mc K}(\op{Cat}^{\op{perf}})^\otimes$, \cref{cor:cohen} recovers Cohen's theorem in tensor triangular geometry \cite[Theorem~B]{Barthel26}. Keep in mind that in this case $\op{Spec}(\mc C)$ is the \textit{Hochster dual} of the Balmer spectrum of $\mc K$. We caution the reader that in \cite{Barthel26} a radical or prime thick ideal is said to be finitely generated if it is the radical of some finitely generated thick ideal. This is weaker than saying that a radical thick ideal is finitely generated in the sense of \cref{def:fg-ideal}.
\end{ex}

\begin{ex}
For $\mc{C}^\otimes=\modu_{R}(\op{Ab})^\otimes$, \cref{cor:cohen} asserts that the ring $R$ is topologically Noetherian if and only if every radical ideal is the radical of some finitely generated ideal. This is also equivalent to the ascending chain condition on radical ideals. Note that this is different from the classical Cohen's theorem which states that the ring $R$ is Noetherian if and only if every radical ideal is finitely generated.
\end{ex}

\begin{rem}
With \cref{cor:cohen} we also obtain a generalization of \cite[Theorem~7]{Barthel26}:
\end{rem}

\begin{cor}
For $\cotimes\in\calg(\prl_{\omega})$, the following are equivalent:
\enu{
\item The spectral space $\op{Spec}(\mc C)$ is finite;
\item Every point in $\op{Spec}(\mc C)$ is locally closed and every radical ideal of $\mc C$ is the radical of some finitely generated ideal;
\item Every point in $\op{Spec}(\mc C)$ is locally closed and every prime ideal of $\mc C$ is the radical of some finitely generated ideal.
} 
\end{cor}
\begin{proof}
This is a direct consequence of \cref{cor:cohen} since \cite[Lemma~3]{Barthel26} establishes that a spectral space is finite if and only if it is Noetherian and every point in the space is locally closed.
\end{proof}

\subsection{Coidempotent frames and cocommutative coalgebras}
The purpose of this subsection is to compare the radicalization construction of §2.1 with the coidempotent frame of a presentably symmetric monoidal $\infty$-category. The link is provided by cocommutative coalgebra objects: their category is Cartesian monoidal, and its categorical ideals are precisely the coidempotent objects of the original category.

\begin{de}
Let $\votimes$ be a symmetric monoidal $\infty$-category.
A coidempotent object in $\mcv$ is an object $c: C \rightarrow \mb{1}$ of the overcategory $\mcv_{/ \mathbf{1}}$ such that \[
C \otimes c: C \otimes C \rightarrow C \otimes \mb{1} \simeq C
\]
is an equivalence. We denote by $\cidemv$ the full subcategory of $\mcv_{/\mb{1}}$ spanned by the coidempotent objects. We call $\cidemv$ the \textbf{coidempotent frame} of \mcv when $\votimes\in \calg(\prl)$. This terminology makes sense due to the following theorem.
\end{de}

\begin{thm}[{\cite[Theorem~3.29]{aoki2023sheaves}}]\label{shvsmadj}
If $\votimes\in \calg(\prl)$ then $\cidemv$ is a frame.
\end{thm}

\begin{de}
Let $\votimes$ be a symmetric monoidal \infcat. The \infcat of cocommutative coalgebras in $\mcv$ is defined to be 
$$\op{cCAlg}(\mcv) \coloneqq \calg(\mcv^{\op{op}})^{\op{op}}.$$
\end{de}

\begin{rem}
The \infcat $\op{cCAlg}(\mcv)$ inherits a symmetric monoidal structure $\op{cCAlg}(\mcv)^\otimes$ such that the forgetful functor $\op{cCAlg}(\mcv)^\otimes\to \votimes$ is symmetric monoidal by \cite[Proposition~3.2.4.3 and Example~3.2.4.4]{ha}. It preserves all existing colimits by \cite[Corollary 3.1.2]{ec1}. Moreover, $\op{cCAlg}(\mcv)^\otimes$ is Cartesian monoidal; see \cite[Theorem 2.4.3.18]{ha}.
\end{rem}

\begin{prop}
If $\votimes\in \calg(\prl)$ then $\op{cCAlg}(\mcv)^\otimes\in \calg(\prl)_{\op{Car}}$. In particular, the forgetful functor $$\op{cCAlg}(\mcv)^\otimes\to \votimes$$ is a morphism in $\calg(\prl)$.
\end{prop}

\begin{proof}
By \cite[Corollary 3.1.4]{ec1}, $\op{cCAlg}(\mcv)$ is presentable since $\mcv$ is. In view of the previous remark, it remains to show that $\op{cCAlg}(\mcv)^\otimes$ is presentably symmetric monoidal. This follows from the fact that the forgetful functor $\op{cCAlg}(\mcv)\to \mcv$ is conservative and our assumption that $\votimes$ is presentably symmetric monoidal. 
\end{proof}

\begin{thm}\label{thm:cidem-zar-ccalg}
Let $\votimes \in \calg(\prl)$.	We have a natural equivalence
$$\zar(\ccalg(\mcv)) \simeq \cidemv $$
which is functorial in $\mcv$.
\end{thm}

\begin{proof}
Since $\ccalg(\mcv)^\otimes$ is Cartesian monoidal, every ideal of $\ccalg(\mcv)^\otimes$ is idempotent. We thus have
$$
\zar(\ccalg(\mcv)) = \idl(\ccalg(\mcv)).
$$
Moreover, a morphism of cocommutative coalgebras $C\to \mb{1}$ is coidempotent if and only if it is a monomorphism in $\ccalg(\mcv)$. Hence,
$$ \idl(\ccalg(\mcv))=\cidem(\ccalg(\mcv)). $$
By the dual of \cite[Proposition 4.8.2.9]{ha}, we obtain an equivalence
$$
\cidem(\ccalg(\mcv))\xrightarrow{\sim}\cidem(\mcv).
$$
Therefore, we have $\zar(\ccalg(\mcv)) = \idl(\ccalg(\mcv)) \simeq \cidem(\mcv)$.
\end{proof}

\begin{rem}
In \cite[Theorem~3.19]{aoki2023sheaves}, there is an adjunction
\[
\begin{tikzcd}
\frm  \arrow[r, "\op{Shv}(-)^\otimes",hook, shift right=-1ex]  & \arrow[l,"\perp"', "\cidem(-)", shift right=-1ex] \calg(\prl) \end{tikzcd}
\]
where $\op{Shv}(-)^\otimes$ is the sheaves functor. Since every $\infty$-topos is Cartesian monoidal, the functor $\op{Shv}(-)^\otimes$ takes values in $\calg(\prl)_{\op{Car}}$. In the following we shall see that the inclusion $\calg(\prl)_{\op{Car}} \hookrightarrow \calg(\prl)$ admits a right adjoint given by taking cocommutative coalgebras.
\end{rem}

\begin{prop}\label{prop:ccalg-adj}
There is an adjunction
$$\begin{tikzcd}
\calg(\prl)_{\op{Car}}  \arrow[r, hook, shift right=-1ex]  & \arrow[l,"\perp"', "\op{cCAlg}(-)^\otimes", shift right=-1ex] \calg(\prl) \end{tikzcd}.$$
\end{prop}

\begin{proof}
The forgetful functor $\op{cCAlg}(\mcv)^\otimes\to \votimes$  satisfies the universal property of a counit by \cite[Corollary 2.4.1.8 and Theorem 2.4.3.18]{ha}.
\end{proof}

\begin{ex}\label{ex:cartesianidl}
For $\votimes \in \calg(\prl)_{\op{Car}}$, we have $\ccalg(\mcv) \simeq \mcv$ by \cref{prop:ccalg-adj} and thus
\[
\zar(\mcv)=\idl(\mcv)\simeq\op{cIdem}(\mcv).
\]
This recovers \cite[Example~2.12]{aoki2023sheaves}. For example, if $\mcx$ is an $\infty$-topos with Cartesian monoidal structure, then $\zar(\mcx)\simeq\mcx_{\leq -1}$ can be identified with its underlying locale.
\end{ex}

\begin{cor}
The adjunction
$$\begin{tikzcd}
\frm \arrow[r, "\op{Shv}(-)^\otimes", hook, shift right=-1ex]  & \arrow[l,"\perp"', "\cidem(-)", shift right=-1ex] \calg(\prl) \end{tikzcd}$$
decomposes into two adjunctions:
$$\begin{tikzcd}
\frm \arrow[r, "\op{Shv}(-)^\otimes", hook, shift right=-1ex]  & \arrow[l,"\perp"', "\zar(-)", shift right=-1ex] \calg(\prl)_{\op{Car}} \arrow[r, hook, shift right=-1ex]  & \arrow[l,"\perp"', "\op{cCAlg}(-)^\otimes", shift right=-1ex] \calg(\prl) \end{tikzcd}.$$
\end{cor}

\begin{proof}
The decomposition is due to \cref{thm:cidem-zar-ccalg}. The first ajunction follows from \cref{ex:cartesianidl}. The second adjunction is \cref{prop:ccalg-adj}.
\end{proof}

\begin{rem}
In \cite{aoki2023sheaves}, the frame $\cidemv\simeq \zar(\ccalg(\mcv))$ is called the unstable smashing spectrum of $\mcv$. In \cref{sec:dualizable} we will give an ideal-theoretic interpretation of this construction.
\end{rem}

\begin{rem}
Note that we always have $\ccalg(\mcv)\simeq\ccalg(\mcv_{/\mb{1}})$ for any $\votimes\in \calg(\prl)$. This observation allows us to factor the adjunction above through the category $\calg(\prl)_{\op{ter}}$ of terminally presentably symmetric monoidal $\infty$-categories. Recall that $\votimes\in \calg(\prl)$ is called terminally monoidal if the monoidal unit of $\mcv$ is a terminal object. We clearly have $$
\calg(\prl)_{\op{Car}}\subseteq\calg(\prl)_{\op{ter}}\subseteq\calg(\prl).
$$
\end{rem}

\begin{prop}
There exist adjunctions:
$$\begin{tikzcd}
\calg(\prl)_{\op{Car}}  \arrow[r, hook, shift right=-1ex]  & \arrow[l,"\perp"', "\op{cCAlg}(-)^\otimes", shift right=-1ex]  \calg(\prl)_{\op{ter}} 
\arrow[r,"\perp"', hook, shift right=-1ex] &
\arrow[l, "(-)^\otimes_{/\mb{1}}", shift right=-1ex]  \calg(\prl)
\end{tikzcd}.  $$
\end{prop}
\begin{proof}
The first adjunction follows from \cref{prop:ccalg-adj} and the previous remark. The second adjunction follows from \cite[Lemma~3.27]{aoki2023sheaves}.
\end{proof}

\begin{rem}
The Zariski frame functor can be identified with the composite $$\zar: \calg(\prl) \xrightarrow{(-)^\otimes_{/\mb{1}}}\calg(\prl)_{\op{ter}} \xrightarrow{\idl(-)}\calg(\prl_0)_{\op{ter}}\xrightarrow{(-)_{\op{rad}}}\calg(\prl_0)_{\op{Car}}\simeq \frm$$
where the first arrow is a right adjoint and the other two are left adjoints.
\end{rem}


\begin{prop}\label{prop:comparisoncidem}
The following commutative diagram is vertically left adjointable in the sense of \cite[Definition 4.7.4.13]{ha}:
$$
\begin{tikzcd}
\calg(\prl)_{\op{Car}} \arrow[d, "\idl(-)"', dashed, shift right] \arrow[r, hook] & \calg(\prl)_{\op{ter}} \arrow[d, "\idl(-)"', dashed, shift right] \\
\frm=\calg(\prl_0)_{\op{Car}} \arrow[r, hook] \arrow[u, hook, shift right]                   & \calg(\prl_0)_{\op{ter}}=\pfrm. \arrow[u, hook, shift right]                  
\end{tikzcd}$$

\end{prop}
\begin{proof}
This follows from \cref{ex:cartesianidl}.
\end{proof}

\begin{rem}
However, the following induced diagram is not horizontally right adjointable in the sense of \cite[Definition 4.7.4.13]{ha}: 
$$\begin{tikzcd}
\calg(\prl)_{\op{Car}} \arrow[d, "\idl(-)"'] \arrow[r, hook, shift left] & \calg(\prl)_{\op{ter}} \arrow[d, "\idl(-)"'] \arrow[l, "\op{cCAlg}(-)^\otimes", dashed, shift left] \\
\frm \arrow[r, hook, shift left]                     & \pfrm. \arrow[l, "\cidem(-)", dashed, shift left]                             
\end{tikzcd}$$
The Beck--Chevalley condition induces a functorial comparison
\[
\theta:\idl(\ccalg(-))\to \cidem(\idl(-)).
\]
By \cref{thm:cidem-zar-ccalg}, the map $\theta$ can be identified with $$\theta_\mcv:\idl(\ccalg(\mcv)) \simeq \cidemv \xrightarrow{\cidem(\im_\mcv)}\cidem(\idl(\mcv))$$
for $\mcv\in \calg(\prl)$.
\end{rem}

\section{Dualizable Categories and Smashing Frames}\label{sec:dualizable}
We now turn to the study of smashing ideals and smashing frames.

\subsection{Dualizable modules}
Throughout this subsection, let $\votimes \in \calg(\prl)$ be a presentably symmetric monoidal $\infty$-category. We collect the notation and results concerning dualizable $\mcv$-modules and internal left adjoints that will be used below. For a more complete treatment, we refer the reader to \cite{ramzi2024dualizable,ramzi2024locally}.

\begin{de}
Recall that
\[
\pr_{\mcv}^{L,\otimes}\coloneqq\modu_{\mcv}(\prl)^\otimes
\]
is the symmetric monoidal $\infty$-category of presentable
$\mcv$-modules and $\mcv$-linear colimit-preserving functors. A morphism $F\colon \mathcal M\to \mathcal N$ in $\prl_{\mcv}$ is called an \textbf{internal left adjoint} if its right adjoint $F^R$ is again $\mcv$-linear and colimit-preserving.
\end{de}

\begin{nota}\label{dualdef}
We denote by
\[
\pr_{\mcv}^{\op{dbl}} \coloneqq \modu_{\mcv}(\prl)^{\op{dbl}}
\]
the non-full subcategory of $\prl_{\mcv}$ with objects the dualizable $\mcv$-modules and morphisms the internal left adjoints. The category $\pr_{\mcv}^{\op{dbl}}$ inherits a symmetric monoidal structure from $\pr_{\mathcal{V}}^{L,\otimes}$, since tensoring with a fixed $\mcv$-module preserves internal left adjoints.
\end{nota}

\begin{nota}
We use the following notation for the stable and additive special cases:
\[
\prst^\otimes \coloneqq \pr_{\opsp}^{L,\otimes} = \modu_{\opsp}(\prl)^\otimes,
\qquad
\prad^\otimes \coloneqq \pr_{\opsp_{\ge0}}^{L,\otimes} = \modu_{\opsp_{\ge0}}(\prl)^\otimes,
\]
where $\opsp_{\ge0}$ denotes the symmetric monoidal $\infty$-category of connective spectra. Thus, $\prst^\otimes$ is the $\infty$-category of presentable stable $\infty$-categories and colimit-preserving functors, while $\prad^\otimes$ is the $\infty$-category of presentable additive $\infty$-categories and additive colimit-preserving functors. We write
\[
\prst^{\op{dbl},\otimes} \coloneqq \pr_{\opsp}^{\op{dbl}} = \modu_{\opsp}(\prl)^{\op{dbl},\otimes},
\qquad
\pr_{\op{ad}}^{\op{dbl},\otimes} \coloneqq \pr_{\opsp_{\ge0}}^{\op{dbl}} = \modu_{\opsp_{\ge0}}(\prl)^{\op{dbl},\otimes}.
\]
\end{nota}

\begin{thm}[{\cite[Theorem 3.1 and Corollary 3.14]{ramzi2024dualizable}}]\label{dualprl}
Let $\kappa$ be any uncountable regular cardinal and $\mcv^\otimes\in\calg(\prl_\kappa)$. Then we have $\pr^{\op{dbl},\otimes}_\mcv\in \calg(\prl_\kappa)$, and the natural (non-full) inclusion $\prvdbl\to \modu_{\mc{V}}(\prl_\kappa)$ preserves and reflects $\kappa$-compact objects. In particular, $\prst^{\op{dbl},\otimes}$ and $\prad^{\op{dbl},\otimes}$ are $\omega_1$-presentably symmetric monoidal, and the symmetric monoidal (non-full) inclusion $\prst^{\op{dbl},\otimes}\hookrightarrow \pr_{\op{st},\omega_1}^{L,\otimes}$ preserves and reflects $\omega_1$-compact objects.
\end{thm}

\begin{prop}[{\cite[Proposition 4.17 and Corollary 4.49]{ramzi2024locally}}]\label{appenb1}
Let $\mcv\to\mcw$ be a morphism in $\calg(\prl)$ such that $\mcw$ is locally rigid as a $\mcv$-algebra and let $\mcm\in \modu_\mcw(\prl)$ be a $\mcw$-module. Then $\mcm$ is dualizable over $\mcw$ if and only if it is dualizable over $\mcv$. If furthermore $\mcw$ is a rigid $\mcv$-algebra, then $\mcw^\otimes\in\calg(\pr_\mcv^{\op{dbl}})$ and there is a canonical symmetric monoidal equivalence $$\modu_\mcw(\prl)^{\op{dbl},\otimes}\simeq \modu_\mcw(\pr_\mcv^{\op{dbl}})^\otimes.$$
\end{prop}

\begin{de}
Let $\mcv\in\calg(\prl)$. We say that a map of \mcv-modules $\mcv\to\mcw$ is a \textbf{smashing localization} if it is both a \mcv-internal left adjoint and a localization.
\end{de}

\begin{prop}\label{rigidbarr}
Let $\mcv\in \calg(\prl)$. Let $F:\mcv\to \mcw$ be a smashing localization, with right adjoint~$G$. Then, $\mathcal W$ admits
a unique symmetric monoidal structure for which $F$ is a symmetric monoidal localization. Furthermore, there is a symmetric monoidal equivalence
$$\mcw\simeq \modu_{G(\mb{1}_{\mcw})}(\mcv).$$
\end{prop}

\begin{proof}
By \cite[Cor. 4.8.5.21]{ha}, it suffices to show that $G$ satisfies the projection formula over $\mcv$, that is, for every object $c\in \mcv$ and $d \in\mcw$, the canonical map $c \otimes G(d)\to G(F(c) \otimes d)$ is an equivalence. This is equivalent to $G$ being \mcv-linear and hence follows by assumption.
\end{proof}

\begin{prop}
Let $\mcv\in\calg(\prl)$. Let $\mcw\in \calg_\mcv$ and let $F:\mcw\to \mcu$ be a \mcw-module map that is a localization. If $\mcw$ is locally rigid over $\mcv$ and $F$ is a \mcv-internal left adjoint, then $F$ is a smashing localization. 
\end{prop}

\begin{proof}
This follows because the right adjoint $G$ is also \mcw-linear by \cite[Proposition 4.18]{ramzi2024locally}.
\end{proof}

\subsection{Smashing ideals and smashing frames}
In this section, we define smashing ideals and smashing frames for presentably symmetric monoidal $\infty$-categories. We also identify the smashing frame with the coidempotent frame (\cref{smidsmcoloc}).

\begin{de}
Let $\votimes\in \calg(\prl)$ and let $i:\mci\to \mcv$ be a morphism in $\linvdbl$. A \textbf{closed ideal} of $\mcv$ is a monomorphism $i:\mci\to \mcv$ in $\linvdbl$, that is, an object of $\idl(\linvdbl)$. If the closed ideal $i:\mci\to \mcv$ is moreover fully faithful, we call it a \textbf{smashing ideal} of $\mcv$.
\end{de}

\begin{rem}
By definition, a smashing ideal $\mci\to \mcv$ is always a closed ideal. We will see that the converse holds if either $\mcv$ is stable or $\mci$ is coidempotent.
\end{rem}

\begin{thm}\label{smideal}
Let $\votimes\in \calg(\prl)$ and let $i:\mci\to \mcv$ be a morphism in $\linvdbl$. Then $i:\mci\to \mcv$ is a smashing ideal of $\mcv$ if and only if it is a coidempotent object of $\linvdbl$.
\end{thm}
\begin{proof}
For the ``if'' direction, assume that $i\otimes 1:\mci\otimes_\mcv\mci \to\mcv\otimes_\mcv\mci=\mci$ is an equivalence. Then $i^R\otimes 1: \mci\to\mci\otimes_\mcv\mci$ is also an equivalence since it is a right adjoint to $i\otimes 1$. 
Then it follows from the following right adjointable diagram that the unit transformation $\op{id}_\mci\to i^Ri$ is an equivalence:
$$\begin{tikzcd}
\mci\otimes_\mcv\mci \arrow[d,  "1\otimes i","\rotatebox{90}{$\sim$}"'] \arrow[r,"i\otimes 1", "\sim"', shift left] & \mcv\otimes_\mcv\mci=\mci \arrow[d, "1\otimes i"] \arrow[l, "i^R\otimes 1", dashed, shift left=2] \\
\mci \arrow[r, "i", shift left]                                        & \mcv. \arrow[l, "i^R", dashed, shift left]                                               
\end{tikzcd}$$
Therefore $i$ is a fully faithful internal left adjoint.

For the ``only if'' direction, 
assume that $i:\mci\rightarrow\mcv$ is fully faithful.  
Note that the colimit-preserving right adjoint $i_R:\mcv\to\mci$ exhibits $\mci$ as a retract of $\mcv$ in $\lint$. Taking the dual, we get 
a \mcv-linear fully faithful functor $i_R^\vee: \mci^\vee \to \mcv$ and the retraction induces the following commutative diagrams in $\lint$
$$\begin{tikzcd}
\mci \arrow[r, "i"] \arrow[rd, Rightarrow, no head] & \mcv \arrow[d, "i^R"] \\
& \mci                 
\end{tikzcd}\quad
\begin{tikzcd}
\mci\otimes_\mcv\mci^\vee \arrow[d, "1\otimes i^\vee_R"'] \arrow[r, "e"] & \mcv \\
\mci\simeq\mci\otimes_\mcv\mcv.  \arrow[ru, "i"']                    &          
\end{tikzcd}$$
Hence, $e$ 
is fully faithful. 
Now the $\mcv$-linear dualizability provides a retraction in $\lint$:
$$\mci\xrightarrow{1\otimes c}\mci\otimes_{\mcv}\mci^\vee\otimes_\mcv\mci\xrightarrow{e\otimes 1}\mci.$$ 
Therefore, $e\otimes 1$ is essentially surjective and hence an equivalence. Moreover, $e\otimes 1$ factors as two inclusions
$$
\begin{tikzcd}
& \mci\otimes_{\mcv}\mcv\otimes_\mcv\mci \arrow[rd, "i\otimes 1",hook] &                                                  \\
\mci\otimes_{\mcv}\mci^\vee\otimes_\mcv\mci \arrow[rr, "e\otimes 1",hook] \arrow[ru, "1\otimes i_R^\vee \otimes 1",hook] &                                                                 & \mcv\otimes_\mcv\mci\simeq\mci.
\end{tikzcd}$$
So all arrows in this diagram are equivalences and hence $i\otimes 1:\mci\otimes_\mcv\mci \to\mcv\otimes_\mcv\mci=\mci$ is an equivalence.
\end{proof}

\begin{rem}
\cref{smideal} justifies the following definition.
\end{rem}

\begin{de}
Let $\votimes\in\calg(\prl)$. The \textbf{smashing frame} of $\mcv$ is defined to be the coidempotent frame of $\linvdbl$:
\[
\op{Sm}(\mcv)\coloneqq\cidem(\linvdbl).
\]
\end{de}

\begin{rem}
Let $\votimes\in \calg(\prl)$. Recall that a localizing ideal of $\mcv$ is a presentable full subcategory of $\mcv$ closed under small colimits and the action of $\mcv$. By definition, the smashing ideals of $\mcv$ can be identified with the localizing ideals $\mci \hookrightarrow \mcv$ whose right adjoint is $\mcv$-linear and colimit-preserving.
\end{rem}

\begin{de}
Let $\votimes\in \calg(\prl)$ and $X\subseteq\mcv$ a small collection of objects. The localizing ideal $\left<X \right>$ generated by $X$ is defined to be the smallest localizing ideal of $\mcv$ containing $X$.
\end{de}

\begin{thm}\label{smidsmcoloc}
Let $\votimes\in \calg(\prl)$. The assignment $x\mapsto \left<x \right>$ induces an equivalence
$$\left<- \right>: \cidemv\xrightarrow{\sim} \cidem(\linvdbl) = \op{Sm}(\mcv)$$
\end{thm}
\begin{proof}
We first check that the assignment is well-defined. For any $x\in\cidemv$, in order to show that $i\colon \left<x \right>\hookrightarrow\mcv$ is a smashing ideal, we need to show that $i^R$ is \mcv-linear and colimit-preserving. This follows by observing that $i\circ i^R\colon\mcv\to \mcv$ can be identified with $x\otimes (-)$. Also, if $x\leq y \in\cidemv$, then $\left<x \right>\hookrightarrow\left<y \right>$ is an internal left adjoint by an argument similar to \cite[Lemma 2.61]{ramzi2024dualizable}. To prove full faithfulness, it suffices to observe that for any $x,y\in\cidemv$, we have $$\left<x \right>\subseteq\left<y \right> \implies x\leq y.$$
For the essential surjectivity, let $i:\mci\hookrightarrow\mcv$ be a smashing ideal. Then  $i\circ i^R:\mcv\to \mcv$ is a map of $\mcv$-modules and hence can be identified with $x\otimes (-):\mcv\to \mcv$ for some $x\in\mcv$.
Then $x$ is coidempotent because $i$ is a colocalization. Consequently $\mci=\left<x \right>$.
\end{proof}

\subsection{The smashing frame as a Zariski frame}
We now prove that for a presentably symmetric monoidal stable $\infty$-category $\mct$, the smashing frame of $\mct$ coincides with the Zariski frame of dualizable $\mct$-modules.

\begin{lem}\label{efi187}
Let $\totimes\in \calg(\prlst)$. Consider a pair of \mct-linear internal left adjoints between dualizable \mct-modules $F: \mathcal{A} \rightarrow \mathcal{D}$ and $G: \mathcal{B} \rightarrow \mathcal{D}$. Then the fiber product $(\mathcal{A} \times_{\mathcal{D}} \mathcal{B})^{\mct\text{-}\mathrm{dbl}}$ in $\lintdbl$ can be identified with the  largest dualizable full subcategory $\mathcal{E} \subseteq \mathcal{A} \times_{\mathcal{D}} \mathcal{B}$, such that the functors $\mathcal{E} \rightarrow \mathcal{A}$ and $\mathcal{E} \rightarrow \mathcal{B}$ are \mct-linear internal left adjoints. Moreover, the natural inclusion $(\mathcal{A} \times_{\mathcal{D}} \mathcal{B})^{\mct\text{-}\mathrm{dbl}} \rightarrow \mathcal{A} \times_{\mathcal{D}} \mathcal{B}$ is a \mct-linear internal left adjoint. 
\end{lem}

\begin{proof}
This follows from the same argument as in the proof of \cite[Proposition 1.87]{efimov2024k}.
\end{proof}

\begin{prop}\label{cdblmono}
Let $\totimes\in \calg(\prlst)$ and let $F: \mathcal{M} \rightarrow \mathcal{N}$ be a morphism in $\lintdbl$. Then:
\enu{
\item $F$ is a monomorphism in $\lintdbl$ if and only if $F$ is fully faithful.
\item $F$ is an epimorphism in $\lintdbl$ if and only if $F$ is a localization functor.
}
\end{prop}
\begin{proof}
We follow the same argument as the proof of \cite[Proposition A.3]{efimov2024k}. The ``if'' direction of part (1) is clear. For the ``only if'' direction, suppose that a functor $F: \mathcal{A} \rightarrow \mathcal{B}$ is a monomorphism in $\prstdbl$. By \cref{efi187} the functor $(\mathcal{A} \times_{\mathcal{B}} \mathcal{A})^{\mct\text{-}\mathrm{dbl}}\rightarrow \mathcal{A} \times_{\mathcal{B}} \mathcal{A}$ is fully faithful. Hence, the functor $\mathcal{A} \rightarrow \mathcal{A} \times_{\mathcal{B}} \mathcal{A}$ is fully faithful. Arguing as in the proof of \cite[Proposition A.1]{efimov2024k}, we conclude that $F$ is fully faithful. Part (2) follows from \cite[Proposition A.1]{efimov2024k} since the functor $\prstdbl \rightarrow \prlst$ is conservative and commutes with colimits.
\end{proof}

\begin{rem}\label{charaidl}
Let $\totimes\in \calg(\prlst)$.
\enu{
\item By \cref{cdblmono}, the closed ideals and smashing ideals of \mct coincide. They can be identified with localizing ideals $\mci \hookrightarrow \mct$ whose right adjoint is \mct-linear and colimit-preserving.
\item If \totimes is locally rigid over $\opsp$ then the closed ideals of $\mct$, by \cite[Proposition 4.18]{ramzi2024locally}, can be identified with the localizing ideals $\mci\hookrightarrow\mct$ whose right adjoint is colimit-preserving.
\item If $\totimes$ is rigid over $\opsp$, then $\totimes$ lies in $\calg(\prstdbl)$, \ie it is a presentably symmetric monoidal stable \infcat such that the underlying \infcat $\mct$ is dualizable and that the unit $\op{Sp}\to \mct$ and the multiplication $\mct\otimes\mct\to \mct$ are strongly continuous. Furthermore, we have an identification $$\pr_{\mct}^{\op{dbl},\otimes}:=\modu_\mct(\prlst)^{\op{dbl},\otimes}\simeq \modu_\mct(\prstdbl)^\otimes$$ by \cref{appenb1}. 
}
\end{rem}

\begin{thm}[Smashing frame identification]\label{dblideal}
Let $\totimes\in\calg(\prlst)$. We have
\[
\op{Sm}(\mct) = \cidem(\lintdbl)\simeq\idl(\lintdbl)\simeq\idl(\lintdbl)_{\op{rad}}=\zar(\lintdbl).
\]
\end{thm}

\begin{proof}
The first equivalence follows from \cref{charaidl}(1). The second equivalence follows from the first equivalence and the fact that $\op{Sm}(\mct)$ is a frame.
\end{proof}

\begin{rem}[Telescope conjecture]
Let $\mck^\otimes \in \calg^{\op{rig}}(\catper)$ be a rigid small idempotent complete tt-\infcat. Ind-completion induces a fully faithful symmetric monoidal functor
\[
\modu_{\mck}(\catper)^\otimes\to \modu_{\op{Ind}(\mck)}(\prstdbl)^\otimes\simeq \pr^{\op{dbl},\otimes}_{\op{Ind}(\mck)}\]
where the last equivalence is explained in \cref{charaidl}(3). Applying the Zariski frame functor, we obtain an injective morphism of frames
\[
\zar(\modu_{\mck}(\catper))\hookrightarrow \zar(\pr^{\op{dbl}}_{\op{Ind(\mck)}})\simeq\op{Sm}(\op{Ind}(\mck))
\]
where the last equivalence is due to \cref{dblideal}. The telescope conjecture for $\mck$ holds precisely when this frame map is surjective, or equivalently, when every smashing ideal of $\op{Ind}(\mck)$ is generated by compact objects. This map need not be surjective; see \cite{burklund2023k} for the example $\mck=\opsp^\omega$.
\end{rem}

\subsection{$\omega_1$-coherence of smashing frames}
The smashing frame identification (\cref{dblideal}) allows us to apply the coherence results of \cref{sec2.3} to smashing frames. We first deduce that the smashing frame of any rigid $\opsp$-algebra is $\omega_1$-coherent. Then, we explicitly construct a set of $\omega_1$-compact generators, based on Efimov's Urysohn lemma for dualizable categories (\cite[Theorem D.1]{efimov2024k}).

\begin{prop}\label{prop:omega1}
For $\totimes\in\calg(\prl_{\op{st},\kappa})$, where $\kappa$ is an uncountable regular cardinal, the smashing frame $\op{Sm}(\mct)$ is $\kappa$-coherent. 
In particular, if $\totimes\in\calg(\prl_{\op{st},\omega_1})$ is 
$\omega_1$-presentably monoidal, then its smashing frame $\op{Sm}(\mct)$ is $\omega_1$-coherent.
\end{prop}

\begin{proof}
Combine \cref{dblideal}, \cref{kcohfrm}, and \cref{dualprl}.
\end{proof}

\begin{cor}
Let $\totimes\in\calg(\prlst)$ be rigid. Then the smashing frame $\op{Sm}(\mct)$ is $\omega_1$-coherent.
\end{cor}

\begin{proof}
By \cref{charaidl}(3), $\mct$ is dualizable and the tensor product $\mct\otimes\mct\to\mct$ is strongly continuous. Hence by \cite[Corollary 2.36]{ramzi2024dualizable} we have $\totimes\in\calg(\prl_{\op{st},\omega_1})$. We then can apply \cref{prop:omega1}.
\end{proof}

\begin{de}
Let $\gamma=\mathbb{R}_{\leqslant 0} \subseteq \mathbb{R}$ be the set of non-positive real numbers. The $\gamma$-topology 
on $\mathbb{R}$ is defined as follows: a subset $U \subseteq \mathbb{R}$ is $\gamma$-open if $U$ is open and $U+\gamma=U$. In other words, $U$ is $\gamma$-open if either $U=\varnothing$ or $U=\mathbb{R}$ or $U=(-\infty, a)$ for some $a \in \mathbb{R}$. We denote by $\mathbb{R}_\gamma$ the set $\mathbb{R}$ equipped with $\gamma$-topology. For a presentable $\infty$-category $\mct$, we denote by $\operatorname{Shv}_{\geqslant 0}(\mathbb{R};\mct)$ the category $\operatorname{Shv}(\mathbb{R}_\gamma ; \mct)$ of $\mct$-valued sheaves on $\mathbb{R}_\gamma$. For $a\in \mathbb{R}\cup \{\pm\infty\}$, we denote by $\mathbb{S}_{<a}$ the sheafification of the representable presheaf represented by $(-\infty,a)$.
\end{de}

\begin{prop}[{\cite[Proposition C.4]{efimov2024k}}]\label{efigenerator1}
Let $\mct\in \prstdbl$ be a dualizable presentable stable $\infty$-category. Denote by $\op{Fun}^{L L}(\operatorname{Shv}_{\geqslant 0}(\mathbb{R} ; \mathrm{Sp}), \mct)$ the full subcategory of $\op{Fun}^L(\operatorname{Shv}_{\geqslant 0}(\mathbb{R} ; \mathrm{Sp}), \mct)$
spanned by the functors whose right adjoints preserve colimits. Let $\mathbb{Q}_{\leqslant}$ denote the poset of rational numbers with the ordinary order. Denote by $\funct^{\op{sct}}(\mathbb{Q}_{\leqslant},\mct)$ the full subcategory of $\op{Fun}(\mathbb{Q}_{\leqslant},\mct)$ spanned by the functors
$\Phi: \mathbb{Q}_{\leqslant} \rightarrow \mct$ satisfying the following conditions:
\enu{
\item for any $a \in \mathbb{Q}$, we have $\Phi(a) \cong \underset{b<a}{\varinjlim} \Phi(b)$ and
\item for any $a<b$ with $a, b \in \mathbb{Q}$, the morphism $\Phi(a) \rightarrow \Phi(b)$ is compact.
}
Then the assignment $F\mapsto (a\to F(\mathbb{S}_{<a}))$ induces an equivalence
\[
\op{Fun}^{L L}(\operatorname{Shv}_{\geqslant 0}(\mathbb{R} ; \mathrm{Sp}), \mct)\xrightarrow{\sim}\funct^{\op{sct}}(\mathbb{Q}_{\leqslant},\mct).
\]
\end{prop}

\begin{rem}
In particular, $\mathrm{Shv}_{\geqslant 0}(\mathbb{R} ; \mathrm{Sp})$ is generated by $\{\mathbb{S}_{<a}\,|\, a\in \mathbb{Q}\}$ under small colimits. Note that $\operatorname{Fun}^{L L}(\operatorname{Shv}_{\geqslant 0}(\mathbb{R} ; \mathrm{Sp}), \mct)$ is an (essentially) small \infcat; see \cite[Lemma 1.21]{ramzi2024dualizable}.
\end{rem}

\begin{thm}[{\cite[Theorem D.1]{efimov2024k}}]\label{efigenerator2}
The category \prstdbl is generated under colimits by a single $\omega_1$-compact object, namely the \infcat $\shv_{\geq0}(\mathbb{R};\opsp)$.
\end{thm}

\begin{cor}\label{cor:omega1-coherent}
Let $\totimes\in\calg(\prlst)$. If $\mct$ is rigid over $\opsp$, then $\shv_{\geq0}(\mathbb{R};\mct)$ is an $\omega_1$-compact generator of $\lintdbl$.
\end{cor}

\begin{proof}
By \cref{charaidl}, we have $\lintdbl\simeq \modu_{\mct}(\prstdbl)$. Therefore the base change adjunction restricts to the following adjunction
$$ \begin{tikzcd}
\prstdbl  \arrow[r, "\mct\otimes(-)", shift right=-1ex]  & \arrow[l,"\perp"', "G", shift right=-1ex] \lintdbl \end{tikzcd}.$$
This implies that $\mct\otimes\shv_{\geq0}(\mathbb{R};\opsp)\simeq\shv_{\geq0}(\mathbb{R};\mct)$ is a generator of \lintdbl.
To prove that $\shv_{\geq0}(\mathbb{R};\mct)$ is~$\omega_1$-compact, it suffices to show that the forgetful functor $G$ preserves $\omega_1$-filtered colimits. This follows from the right adjointable diagram
$$
\begin{tikzcd}
\prstdbl \arrow[d] \arrow[r, "\mct\otimes(-)", shift left] & \lintdbl \arrow[d] \arrow[l, "G", shift left] \\
\prl_{\op{st},\omega_1} \arrow[r, "\mct\otimes(-)", shift left]              &  \arrow[l, shift left]    \modu_{\mct}(\prl_{\op{st},\omega_1})         
\end{tikzcd}
$$
and \cref{dualprl}.
\end{proof}

\begin{rem}
Let $\mct\in\calg^{\op{rig}}(\prlst)$ be a rigid $\opsp$-algebra and $\Phi\in\funct^{\op{sct}}(\mathbb{Q}_{\leqslant},\mct)$. The localizing ideal $\left<\im(\Phi)\right>$ is smashing by \cite[Lemma 2.61]{ramzi2024dualizable}.
\end{rem}

\begin{thm}
Let $\totimes\in\calg^{\op{rig}}(\prlst)$ be a rigid $\opsp$-algebra. Then $$\{\left<\im(\Phi)\right>\,|\, \Phi\in \funct^{\op{sct}}(\mathbb{Q}_{\leqslant},\mct)\}\subseteq \op{Sm}(\mct)^{\omega_1}$$ is a set of $\omega_1$-compact generators for the smashing frame $\op{Sm}(\mct)$.
\end{thm}

\begin{proof}
Let $$l:\funct^{\op{sct}}(\mathbb{Q}_{\leqslant},\mct)\xrightarrow{\sim}\operatorname{Fun}^{L L}(\operatorname{Shv}_{\geqslant 0}(\mathbb{R} ; \mathrm{Sp}), \mct)$$ denote the inverse equivalence in \cref{efigenerator1}.
Recall the identification $\op{Sm}(\mct)\simeq\idl(\lintdbl)$ from \cref{dblideal}. According to \cref{cor:omega1-coherent}, $\shv_{\geq0}(\mathbb{R};\mct)$ is an $\omega_1$-compact generator of $\lintdbl$. Hence, by \cref{idlccompgene} and \cref{rem:generatingset},
$$\{\im_{\lintdbl}(l(\Phi))\,|\, \Phi\in \funct^{\op{sct}}(\mathbb{Q}_{\leqslant},\mct)\}\subseteq \op{Sm}(\mct)^{\omega_1}$$
is a generating set of $\omega_1$-compact objects for the smashing frame $\op{Sm}(\mct)$. By \cite[Lemma 2.61]{ramzi2024dualizable} we actually have $$\im_{\lintdbl}(l(\Phi))\simeq\left<\im(l(\Phi))\right>.$$ Since the image of the map
$$\mathbb{Q}_{\le}\xrightarrow{a\mapsto(\mathbb{S}_{<a}\otimes \mb{1}_\mct)} \shv_{\geq0}(\mathbb{R};\mct)$$
generates the target under small colimits and the action of $\mct$, the following two localizing ideals of \mct are identical
$$\left<\im(l(\Phi))\right>\simeq\left<\im(\Phi)\right>\subseteq \mct.$$
Consequently,
$$\{\left<\im(\Phi)\right>\,|\, \Phi\in \funct^{\op{sct}}(\mathbb{Q}_{\leqslant},\mct)\}\subseteq \op{Sm}(\mct)^{\omega_1}$$
is a generating set of $\omega_1$-objects.
\end{proof}

\begin{nota}
Let $R\in \calg(\op{Sp})$ be an $\einf$-ring. We denote $\pr_R^{\op{dbl},\otimes}:=\modu_{\modu_R(\opsp)}(\prlst)^{\op{dbl},\otimes}$. Thus $\pr_R^{\op{dbl}}$ is the $\infty$-category whose objects are the presentable stable $R$-linear $\infty$-categories that are dualizable as $\modu_R(\opsp)$-modules, and whose morphisms are $R$-linear strongly continuous functors.
\end{nota}

\begin{cor}
Let $R\in \calg(\op{Sp})$ be an $\einf$-ring. The smashing frame $\op{Sm}(\modu_R(\opsp))$ 
can be identified with the Zariski frame of $\pr_R^{\op{dbl}}$  $$\op{Sm}(\modu_R(\opsp))\simeq\zar(\pr_R^{\op{dbl}}).$$
Moreover, $\op{Sm}(\modu_R(\opsp))$ is an $\omega_1$-coherent frame generated by the $\omega_1$-compact objects $$\{\left<\im(\Phi)\right>\,|\, \Phi\in \funct^{\op{sct}}(\mathbb{Q}_{\leqslant},\modu_R(\opsp))\}\subseteq \op{Sm}(\modu_R(\opsp))^{\omega_1}.$$
In particular, taking $R=\mathbb{S}$, the sphere spectrum, we obtain a generating set of $\omega_1$-compact objects $$\{\left<\im(\Phi)\right>\,|\, \Phi\in \funct^{\op{sct}}(\mathbb{Q}_{\leqslant},\opsp)\}\subseteq \op{Sm}(\opsp)^{\omega_1}$$
for the smashing frame $\op{Sm}(\opsp)$.
\end{cor}

\section{$\Sigma$-Triviality and $\Sigma$-Exactness}\label{sec:sigma}
In classical commutative algebra, the quotient of a ring by an ideal canonically inherits a ring structure. In contrast, within higher categorical frameworks, the analogous construction---such as forming the quotient of an $\mathbb{E}_\infty$-ring by an ideal---is more subtle and requires additional care, due in part to underlying homotopical considerations. In order to obtain a well-behaved theory of quotient by ideals in our setting, we introduce the notion of $\Sigma$-triviality and $\Sigma$-exactness in this section.

\subsection{$\Sigma$-trivial $\infty$-categories}

\begin{de}\label{de:Sigma-trivial}
Let $\mcc$ be a pointed $\infty$-category. We say that $\mcc$ is \textbf{$\Sigma$-trivial} if for every object $X\in\mcc$, the map $X\to0$ is an epimorphism, \textbf{$\Omega$-trivial} if for any object $X\in\mcc$, the map $0\to X$ is a monomorphism, \textbf{shift-trivial} if it is both $\Sigma$-trivial and $\Omega$-trivial.
\end{de}

\begin{rem}\label{sigtriveqshtriv}
Let $\mcc$ be a pointed $\infty$-category.
\enu{\item If $\mcc$ has cofibers, then $\mcc$ is $\Sigma$-trivial if and only if the suspension functor $\Sigma:\mcc\to \mcc$ is zero.
\item If $\mcc$ has fibers, then $\mcc$ is $\Omega$-trivial if and only if the loop functor $\Omega:\mcc \to\mcc$ is zero.
\item If $\mcc$ has cofibers and fibers, then $\mcc$ is $\Sigma$-trivial if and only if $\Omega$-trivial, since the adjoint of a zero functor is zero. Therefore, in this case, the conditions of being $\Sigma$-trivial, $\Omega$-trivial, and shift-trivial are all equivalent.}
\end{rem}

\begin{prop}\label{ptsubcat}
Let $\mc{C}$ be a pointed $\infty$-category. Let $\mcc'$ be a pointed (not necessarily full) subcategory of $\mcc$. 
The following hold:
\enu{
\item If $\mcc$ is $\Omega$-trivial, then so is $\mcc'$.
\item If $\mcc$ is $\Sigma$-trivial, then so is $\mcc'$.
\item If $\mcc$ is shift-trivial, then so is $\mcc'$.
}
\end{prop}
\begin{proof}
For part (1), we consider the following diagram
$$
\begin{tikzcd}
{\mapp_{\mcc'}(X,0)} \arrow[d] \arrow[r, hook] & {\mapp_{\mcc}(X,0)} \arrow[d, hook] \\
{\mapp_{\mcc'}(X,Y)} \arrow[r, hook]           & {\mapp_{\mcc}(X,Y)}
\end{tikzcd}
$$ 
for any $X,Y\in \mcc'$. Since the top, bottom, and right arrows are all $(-1)$-truncated, so is the left arrow $\mapp_{\mcc'}(X,0)\to \mapp_{\mcc'}(X,Y)$. This indicates that for any $Y\in \mcc'$, the map $0\to Y$ is a monomorphism in $\mcc'$ and hence $\mcc'$ is $\Omega$-trivial. The proof for (2) is similar to (1). Part (3) follows from (1) and (2).
\end{proof}

\begin{ex}\label{sigmatrivexamples}
The following are all examples of shift-trivial $\infty$-categories:
\enu{
\item All pointed $1$-categories;
\item The \infcat $\Cat_{\infty}^*$ of small pointed $\infty$-categories and point-preserving functors;
\item  The \infcat $\Cat_{\infty}^{+}$ of small preadditive $\infty$-categories and finite-product-preserving functors;
\item  The \infcat $\op{Cat}_{\infty}^{\op{ad}}$ of small additive $\infty$-categories  and finite-product-preserving functors;
\item The \infcat $\op{Cat}_{\infty}^{\op{st}}$ of small stable $\infty$-categories  and exact functors;
\item The \infcat $\Cat^{\op{perf}}$ of small idempotent complete stable $\infty$-categories  and exact functors.
\item The \infcat $\prl_\mcv$ of \mcv-linear presentable categories for any presentably symmetric monoidal $\infty$-category $\mcv^\otimes\in\calg(\prl)$.
\item The \infcat $\prvdbl$ of dualizable $\mcv$-linear presentable categories, where $\mcv^\otimes\in\calg(\prl_*)$ is a pointed presentably \syminfcat.
}
\end{ex}

\begin{rem}
A stable \infcat $\mcc$ is $\Sigma$-trivial if and only if $\mc{C}=0$, since any epimorphism in $\mcc$ is an equivalence.
\end{rem}

\begin{lem}
Let $\mcc$ be a pointed $\infty$-category.
\enu{
\item If $\mcc$ is $\Sigma$-trivial then for any cofiber sequence $A\to B\to C$ in $\mcc$, the map $B\to C$ is an epimorphism.
\item If $\mcc$ is $\Omega$-trivial then for any fiber sequence $A\to B\to C$ in $\mcc$, the map $A\to B$ is a monomorphism.
}
\end{lem}

\begin{proof}
This follows from the fact that epimorphisms (resp. monomorphisms) are closed under pushout (resp. pullbacks).
\end{proof}

\begin{rem}
Usually, a triangle
in an $\infty$-category consists of the following data: a pair of morphisms
$f \colon x \to y$ and $g \colon y \to z$, together with a \textit{choice} of nullhomotopy
$g f \simeq 0$.
However, in a $\Sigma$-trivial $\infty$-category $\mcc$, the space of nullhomotopies $g f \simeq 0$ is contractible. This is because we have $\Omega\mapp_{\mcc}(x,z) \simeq \mapp_{\mcc}(\Sigma x, z) = 0$ by $\Sigma$-triviality. Hence, a triangle in $\mcc$ can be identified with a pair of composable morphisms whose composition is homotopic to zero. In particular, there is no need to keep track of the nullhomotopy, as it is unique up to a contractible choice. Note that this phenomenon never occurs in a stable $\infty$-category unless it is trivial.
\end{rem}

\begin{rem}
For the next result, we refer the reader to \cite[Definition~5.5.4.5 and Remark~5.5.4.7]{htt} for the definition of strongly saturated class and small generation.
\end{rem}

\begin{lem}\label{Ssigma}
Let $\mc{C}\in\prl_{*}$ be a pointed presentable $\infty$-category. Let $S_\Sigma$ be the smallest strongly saturated class containing $\SET{\Sigma X\to 0}{X\in\mcc}$. Then $S_\Sigma$ is of small generation.
\end{lem}

\begin{proof}
Let $S=\{ X_\alpha\}$ be a set of generators of $\mc{C}$. Then $\overline{\SET{\Sigma X\to 0}{X\in\mcc}}=\overline{\SET{\Sigma X_{\alpha}\to 0}{X_{\alpha}\in\mcc}}$, which is of small generation.
\end{proof}

\begin{de}
Let $\mc{C}\in\prl_{*}$ be a pointed presentable $\infty$-category.
\enu{
\item We call the accessible localization $\mcc \to \mc{C}_{\Sigma} \coloneqq \mc{C}[S^{-1}_{\Sigma}]$ the \textbf{$\Sigma$-trivialization} of $\mc{C}$.
\item We define the \textbf{$\Sigma$-core} of $\mc{C}$ to be the full subcategory $\mc{C}^{\Sigma}$ spanned by the objects $X$ such that $\Sigma X\simeq0$.
\item We define the \textbf{$\Omega$-core} of $\mc{C}$ to be the full subcategory $\mc{C}^{\Omega}$ spanned by the objects $X$ such that $\Omega X\simeq 0$.
} 
\end{de}

\begin{prop}\label{sigmacore}
Let $\mcc\in\prl_{*}$ be a pointed presentable $\infty$-category. 
Then:
\enu{
\item The $\Sigma$-trivialization $\mc{C}_\Sigma$ can be identified with the $\Omega$-core $\mc{C}^\Omega$.
\item The subcategory $\mc{C}^\Omega\subseteq\mcc$ is closed under subobjects, \ie whenever $A\to X$ is a monomorphism in $\mcc$ with $X\in\mc{C}^\Omega$, we have $A\in\mc{C}^\Omega$.
} 
\end{prop}

\begin{proof}
For part (1), it suffices to observe that an object $X$ is $S_\Sigma$-local if and only if $\Omega X=0$. Part (2) follows by applying the $n=-1$ case of \cite[Lemma 5.5.6.14]{htt}.
\end{proof}

\begin{prop}
Let $\prl_{\Sigma}$ denote the full subcategory of $\prl_{*}$ spanned by the $\Sigma$-trivial presentable $\infty$-categories. Then we have the following adjunctions
$$\begin{tikzcd}
\prl_{*} \arrow[r, "(-)_{\Sigma}", shift left=2ex] \arrow[r, "(-)^{\Sigma}"', shift right=2ex]  & \arrow[l,"\perp","\perp"', hook'] \prl_{\Sigma}
\end{tikzcd}  .$$
\end{prop}

\begin{proof}
This follows from the definition of $\Sigma$-trivialization and $\Sigma$-core.
\end{proof}

\begin{prop}
Let $\mc{S}_*$ denote the \infcat of pointed spaces.
\enu{
\item Let $L_\Sigma: \mc{S}_* \rightleftarrows \mc{S}_{*,\Sigma}$ denote the $\Sigma$-trivialization of $\mc{S}_*$. Then $\mc{S}_{*,\Sigma} \subseteq \mc{S}_*$ is spanned by the pointed spaces whose connected component containing the base point is contractible. In other words, $\mc{S}_{*,\Sigma}$ is the essential image of the functor adjoining a base point $\mc{S}\to \mc{S}_*$.
\item A map of pointed spaces $f: X\to Y$ is an $L_\Sigma$-equivalence if and only if for any point $x\in X$ \textbf{not} in the connected component of the base point, the induced map $\pi_n(X,x)\to \pi_n(Y,f(x))$ is bijective for each $n\geq0$.
\item A pointed space $X$ lies in the $\Sigma$-core $\mc{S}_*^\Sigma$ of $\mc{S}_*$ if and only if $X$ is an acyclic space.
}
\end{prop} 

\begin{proof}
For part (1), by \cref{sigmacore}, a pointed space $X$ lies in $\mc{S}_{*,\Sigma}$ if and only if $\Omega X=0$, \ie its base component is contractible. For part (2), it suffices to observe that the $L_\Sigma$ is the operation which contracts the base component into a point while keeping the other components unchanged. See \cref{acyc} for part (3).
\end{proof}

\begin{de}
We call an object of $\mc{S}_{*,\Sigma}$ a \textbf{unital pointed space}.
\end{de}

\begin{thm}\label{thm:mode}
For any pointed presentable \infcat $\mc{C}$, the $\Sigma$-trivialization $\mcc \to \mcc_{\Sigma}$ can be identified with the natural map $\mc{C}\simeq \mc{C}\otimes \mc{S}_*\to \mc{C}\otimes\mc{S}_{*,\Sigma}$. In particular, $\mc{S}_{*,\Sigma}$ is a mode, which we call the \textbf{$\Sigma$-trivial mode}.
\end{thm}

\begin{proof}
It suffices to show that the $\Sigma$-trivialization $\prl_* \to \prl_{\Sigma}$ is a smashing localization. This follows from the observation
\[
\mc{C}\otimes\mc{S}_{*,\Sigma}=\funct^R((\mc{S}_{*,\Sigma})^{\op{op}},\mcc)=\funct^R(\mc{S}_{*}[(S^1\to 0)^{-1}]^{\op{op}},\mcc)=\mcc^\Omega
\]
and \cref{sigmacore}.
\end{proof}

\begin{ex}\label{ex:sigma-trivialization}
The following examples follow from the identification $\mcc_{\Sigma}\simeq\mcc^{\Omega}$ of 
\cref{sigmacore}:
\enu{
\item The $\Sigma$-trivialization of the category of discrete spaces is the category of pointed discrete spaces: $\mc{S}_{\leq0}\otimes\mc{S}_{*,\Sigma}=\mc{S}_{*,\leq0}$.
\item The $\Sigma$-trivialization of the category of $\einf$-spaces is the category of unital $\einf$-spaces, that is, the $\einf$-spaces whose identity component is contractible: $\cmon(\mc{S})\otimes \mc{S}_{*,\Sigma}=\cmon(\mc{S})_\Sigma$.
\item The $\Sigma$-trivialization of the category of connective spectra is the category of abelian groups: $\spgeq\otimes\mc{S}_{*,\Sigma}=\op{Sp}^\heartsuit=\op{Ab}$.
\item The $\Sigma$-trivialization of the category of spectra is trivial: $\op{Sp}\otimes\mc{S}_{*,\Sigma}=0$.
}
\end{ex}

\subsection{Quotients by ideals}\label{sec4}

\begin{nota}
Let $\mcc$ be a pointed $\infty$-category which admits finite colimits and finite limits and $N\to M$ be a morphism in $\mcc$. We write $M/N$ for $\coker(N\to M)$. Hence, we have the following pushout diagram
$$
\begin{tikzcd}
N \arrow[d] \arrow[r] & M \arrow[d] \\
0 \arrow[r]           & M/N     .   
\end{tikzcd}$$
\end{nota}

\begin{de}
Let $\cotimes\in \calg(\prl_{*})$. We say that an ideal $I\in \idlc$ is \textbf{quotientable} if $I\to 0$ is an epimorphism in $\mcc$.
\end{de}

\begin{rem}
By definition, if $\cotimes\in \calg(\prl_{*})$ is $\Sigma$-trivial, then every ideal of $\mcc$ is quotientable.
\end{rem}

\begin{lem}\label{cokerof0}
Let $\mcc$ be a pointed $\infty$-category which admits finite colimits, and let $N\xrightarrow{0} M$ be the zero morphism in  $\modu_R(\mcc)$. If $N\to 0$ is an epimorphism, then $M\to M/N$ is an equivalence.
\end{lem}

\begin{proof}
Consider the pushout square
$$
\begin{tikzcd}
N \arrow[d] \arrow[r] & 0 \arrow[d] \arrow[r] & M \arrow[d] \\
0 \arrow[r]           & 0 \arrow[r]           & M/N.
\end{tikzcd}$$
Since $N\to 0$ is an epimorphism, the left square is pushout and so is the right square. Therefore, $M/N \simeq M \coprod 0 \simeq M$.
\end{proof}

\begin{thm}\label{quotidl}
Let $\cotimes\in\calg(\prl_*)$. Let $R\in\calg(\mcc)$ and let $I\subseteq R$ be a quotientable ideal of $\modu_{R}(\mcc)$. Then there exists a unique commutative $R$-algebra structure on $R/I$. Moreover, the map $R\to R/I$ is an epimorphism in $\calg(\mc{C})$ satisfying the following universal property: For any $B\in \calg(\mc{C})$, the induced map $$\mapp_{\calg(\mc{C})}(R/I,B)\to \mapp_{\calg(\mc{C})}(R,B)$$ is a $(-1)$-truncated map whose image on $\pi_0$ consists of the maps $R\to B$ such that the composite $I\to R\to B$ is homotopic to $0$ in $\modu_R(\mcc)$.
\end{thm}

\begin{proof}
By \cref{calgepi2}, it suffices to show that $R\otimes_RR \to R\to R/I$ factors through $R/I\otimes_R R/I$. We claim that $R/I \simeq R/I\otimes_RR\to R/I\otimes_R R/I$ is actually an equivalence. Indeed, by the pushout diagram
$$
\begin{tikzcd}
R/I\otimes_RI \arrow[d] \arrow[r] & R/I\otimes_RR \arrow[d] \\
0 \arrow[r]         & R/I\otimes_RR/I      
\end{tikzcd}$$
and \cref{cokerof0}, it suffices to show that the map $R/I\otimes_RI \to R/I\otimes_RR$ is zero. This follows from the observation that in the following commutative diagram $$
\begin{tikzcd}
R\otimes_RI \arrow[d] \arrow[r] & R\otimes_RR \arrow[d] \\
R/I\otimes_RI \arrow[r]         & R/I\otimes_RR        
\end{tikzcd}$$
the composite of the top-right corner is zero and the left vertical map is an epimorphism.
\end{proof}

\begin{rem}
In particular, if $\mcc$ is $\Sigma$-trivial then \cref{quotidl} holds for all ideals $I \subseteq R$ of $\modu_{R}(\mcc)$.
\end{rem}

\begin{rem}
Note that the composite $I\to R\to B$ being null in $\modu_R(\mcc)$ is not equivalent to being null in $\mcc$. This is because, unlike in the $1$-categorical case, the forgetful functor $\modu_R(\mcc)\to\mcc$ does not necessarily induce an injective map on $\pi_0$ of the mapping spaces.
\end{rem}

\subsection{Quotients of $\einf$-semirings}
We saw in \cref{ex:sigma-trivialization} that the $\Sigma$-trivialization of the universal additive \infcat  $\spgeq\simeq\op{CGrp}(\mcs)$ is the category of abelian groups. Hence, all $\Sigma$-trivial additive \infcats are in fact $1$-categories.
However, this is not the case in the semi-additive case, as we now explain. 

\begin{de}
Since $\mcs_{*,\Sigma}$ is a mode, the $\infty$-category $\cmon(\mcs)_\Sigma:= \cmon(\mcs)\otimes \mcs_{*,\Sigma}$ of unital $\einf$-spaces naturally inherits a symmetric monoidal structure. This makes the $\Sigma$-trivialization into a symmetric monoidal localization
$$L_\Sigma^{\cmon(\mcs)}\colon\cmon(\mcs)^\otimes\rightleftarrows\cmon(\mcs)_\Sigma^\otimes$$
and therefore induces an adjunction between \textbf{$\einf$-semirings} and \textbf{unital $\einf$-semirings}
$$\calg(L_\Sigma^{\cmon(\mcs)})\colon\calg(\cmon(\mcs))\rightleftarrows\calg(\cmon(\mcs)_\Sigma).$$
\end{de}

\begin{rem}
Note that a unital \ein-semiring is unital at the component containing $0$, instead of $1$. 
\end{rem}

\begin{prop}
Let $X\in \cmon(\mcs)$ be an \ein-space. Then $L_\Sigma^{\cmon(\mcs)}(X)$ can be (functorially) identified with the $\einf$-space $X/X_0 = \coker(X_0 \hookrightarrow X)$, where $X_0$ denotes the identity component of $X$.
\end{prop}

\begin{proof}
We write $L$ for $L_\Sigma^{\cmon(\mcs)}$. First we claim that $L(X_0)$ is trivial, so that $X\to X/X_0$ is an $L$-equivalence. Indeed, for any $Y\in \cmon(\mcs)_\Sigma$, we have $$\mapp_{\cmon(\mcs)_\Sigma}(L(X_0),Y)\simeq \mapp_{\cmon(\mcs)}(X_0,Y)\simeq \lim_{f\in \op{Tw}(\op{Fin}_*)} \mapp_{\mcs_*}(X_{0}^m,Y^n)$$
where the second equivalence follows from $\cmon(\mcs)\simeq\cmon(\mcs_*)\simeq \funct^{\op{seg}}(\op{Fin}_*,\mcs_*)$ and \cite[\href{https://kerodon.net/tag/03P7}{Example~03P7}]{kerodon}. Moreover, since each $\mapp_{\mcs_*}(X_{0}^m,Y^n)$ is contractible by the triviality of the identity component of $Y$, we have $\mapp_{\cmon(\mcs)_\Sigma}(L(X_0),Y)=0$ and therefore $L(X_0)=0$.

It remains to show that $X/X_0$ is $L$-local, that is, the identity component of $X/X_0$ is trivial. Considering the Cartesion monoidal structure on $\cmon(\mcs)$, we have $\cmon(\mcs)\simeq \calg(\mcs)$. Hence, the space $X/X_0$ can be identified with the base change $0\otimes_{X_0}X$. Since $X$ can be decomposed as two $X_0$-submodules $X_0\sqcup X_1$ by \cref{modumono}, we have $$0\otimes_{X_0}X\simeq 0\otimes_{X_0}(X_0\sqcup X_1)\simeq *\sqcup X_1/X_0,$$
which means $X/X_0$ is unital.
\end{proof}

\begin{rem}
Note that the $\Sigma$-trivialization of $\cmon(\mcs)$ is not compatible with that of $\mcs$. To see this, taking the $\einf$-space $X=S^1\times \mathbb{N}\simeq K(\mathbb{Z},1)\times\mathbb{N}$, we then have $L_\Sigma^{\cmon(\mcs)}(X)\simeq \mathbb{N}$
but $L_\Sigma^\mcs(X)\simeq *\sqcup (S^1\times \mathbb{N}_{\geq1})$.
\end{rem}

\begin{prop}
The following statements hold:
\begin{enumerate}[label=(\arabic*),font=\normalfont]
\item For any \ein-space $X$, we have natural equivalences
$$
\begin{tikzcd}
\idl(\modu_{X}(\mcs)) \arrow[r, "\sim"] \arrow[d, "\sim"] & \idl(\modu_{X_+}(\mcs_*)) \arrow[d, "\sim"] \arrow[r, "\sim"]              & \idl(\modu_{L^{\mcs}_\Sigma(X_+)}(\mcs_{*,\Sigma})) \arrow[d, "\sim"] \\
\idl(\modu_{\pi_0X}(\op{Sets})) \arrow[r, "\sim"] & \idl(\modu_{\pi_0X_+}(\op{Sets}_*)) \arrow[r, Rightarrow, no head] & \idl(\modu_{\pi_0X_+}(\op{Sets}_*))                           
\end{tikzcd}
$$
\item For any \ein-semiring space $R$, we have natural equivalences
$$
\begin{tikzcd}
\idl(\modu_{R}(\cmon(\mcs))) \arrow[r, "\sim"] \arrow[d, "\sim"]              & \idl(\modu_{L(R)}(\cmon(\mcs)_\Sigma)) \arrow[d, "\sim"] \\
\idl(\modu_{\pi_0R}(\cmon(\op{Sets}))) \arrow[r,  Rightarrow, no head] & \idl(\modu_{\pi_0R}(\cmon(\op{Sets})))          
\end{tikzcd}
$$

\end{enumerate}

\end{prop}

\begin{proof}

Since monomorphisms in \mcs are the component inclusions, all ideals are in one-one corresponding to ideals of its $\pi_0$. 
\end{proof}

\begin{lem}
Let $X$ be an \ein-space. The map $p:X\to 0$ is an epimorphism in $\cmon(\mcs)$ if and only if $X$ has trivial $\mathbb{E}_1$-group completion.
\end{lem}

\begin{proof}
By definition, $p$ is epimorphic if and only if $0/X\simeq BX$ is contractible. 	According to the proof of \cite[Theorem 1]{nikolaus2017group}, we see that  $X\to \Omega BX$ exhibits $\Omega BX$ as the $\mathbb{E}_1$-group completion of $X$. However, $BX$ is connected and hence $BX$ is contractible if and only if $\Omega BX$ is.
\end{proof}

\begin{thm}[Quotient $\einf$-semiring]\label{thm:quotient-einf-semiring}
Let $R\in\calg(\cmon(\mcs))$ be an $\einf$-semiring and let $I\subseteq R$ be an element of $ \idl(\modu_{R}(\cmon(\mcs)))\simeq \op{SIdl}(\pi_0R)$, where we denote $\op{SIdl}(\pi_0R):=\idl(\modu_{\pi_0R}(\cmon(\op{Sets})))$. If $I$ as an \ein-space has trivial $\mathbb{E}_1$-group completion, then there exists an $\einf$-semiring $R/I$ and an $\einf$-semiring map $R\to R/I$ satisfying the following universal property: For any $\einf$-semiring $S$, the induced map $$\mapp_{\calg(\cmon(\mcs))}(R/I,S)\to \mapp_{\calg(\cmon(\mcs))}(R,S)$$ is $(-1)$-truncated and its image on $\pi_0$ consists those maps $R\to S$ such that the composition $I\to R\to S$ is null in $\modu_R(\cmon(\mcs))$.
\end{thm}
\begin{proof}
Since the $\mathbb{E}_1$-group completion of $I$ is trivial, the previous lemma shows that the map $I\to 0$ is an epimorphism in $\cmon(\mcs)$. We can then apply \cref{quotidl}.
\end{proof}

\begin{thm}[Quotient unital $\einf$-semiring]\label{thm:quotient-unital-einf-semiring}
Let $R\in\calg(\cmon(\mcs)_\Sigma)$ be a unital $\einf$-semiring and $I\subseteq R$ be an element in $ \idl(\modu_{R}(\cmon(\mcs)_\Sigma))\simeq \op{SIdl}(\pi_0R)$. Then there exists a unital $\einf$-semiring $R/I$ and an $\einf$-semiring map  $R\to R/I$ satisfying the following universal property: For any unital $\einf$-semiring $S$, the induced map $$\mapp_{\calg(\cmon(\mcs))}(R/I,S)\to \mapp_{\calg(\cmon(\mcs))}(R,S)$$ is a $(-1)$-truncated map whose image on $\pi_0$ consists those maps $R\to S$ such that the composition $I\to R\to S$ is null in $\modu_R(\cmon(\mcs)_\Sigma)$. Note that in this case the composition $I\to R\to S$ is null in $\modu_R(\cmon(\mcs)_\Sigma)$ if and only if the composition $\pi_0I\to \pi_0R\to \pi_0S$ is zero of $\pi_0R$-semimodules.
\end{thm}

\begin{proof}
The first statement is a direct consequence of \cref{quotidl} since $\cmon(\mcs)_\Sigma$ is by definition $\Sigma$-trivial. The last statement follows from the assumption that the component of $S$ containing 0 is contractible.
\end{proof}

\subsection{$\Sigma$-exact $\infty$-categories}
In commutative algebra, there is a one-to-one correspondence between the ideals of a commutative ring $R$ and the epimorphisms out of $R$. Our next goal is to establish such a correspondence in our setting.

\begin{de}Let $\mc{C}$ be a $\Sigma$-trivial  $\infty$-category which admits finite colimits and finite limits.
\enu{
\item Let $f: X\to Y$ be a morphism in $\mcc$. We say that $f$ is a \textbf{normal monomorphism} if it is the kernel of some morphism, a \textbf{normal epimorphism} if it is the cokernel of some morphism.
\item We say that $\mcc$ is \textbf{left $\Sigma$-exact} if every monomorphism in $\mcc$ is normal, \textbf{right $\Sigma$-exact} if every epimorphism in $\mcc$ is normal, and \textbf{$\Sigma$-exact} if it is both left and right $\Sigma$-exact.
}
\end{de}

\begin{prop}\label{kerofcoker}
Let $\mc{C}$ be a $\Sigma$-trivial $\infty$-category which admits finite colimits and finite limits. Let $f: X\to Y$ be a morphism in $\mcc$. Then
\begin{enumerate}
\item $f$ is a normal monomorphism if and only if $X$ is the kernel of the map $Y\to \coker(f)$;
\item $f$ is a normal epimorphism if and only if $Y$ is the cokernel of the map $\ker(f)\to X$.
\end{enumerate}
\end{prop}
\begin{proof}
We will only prove (1) since the proof of (2) is similar. The ``if'' part is clear. For the ``only if'' part, suppose that $f: X\to Y$ be a normal monomorphism such that $X=\ker(g)$ for some morphism $g:Y\to Z$. We need to show $X\simeq\ker(\coker(f))$.
The evident commutative diagram
$$
\begin{tikzcd}
X \arrow[rd] \arrow[d, dashed]              &                                            &                             \\
\ker(\coker(f)) \arrow[r] \arrow[d, dashed] & Y \arrow[r] \arrow[d,  Rightarrow, no head] & \coker(f) \arrow[d, dashed] \\
X \arrow[r]                                 & Y \arrow[r]                                & Z                          
\end{tikzcd}$$
asserts that $X$ is a retract of $\ker(\coker(f))$ over $Y$. This implies $X\simeq\ker(\coker(f))$, since both are $(-1)$-truncated over $Y$.
\end{proof}
\begin{cor}
Let $\mc{C}^\otimes\in\calg(\prl_{\Sigma})$ be a presentably symmetric monoidal $\Sigma$-trivial $\infty$-category. If $\mcc$ is left (resp. right) $\Sigma$-exact, then so is $\modu_{R}(\mcc)$ for any $R\in\calg(\mcc)$.
\end{cor}

\begin{proof}
It suffices to observe that the forgetful functor $\modu_{R}(\mcc)\to\mcc$ creates colimits and limits.
\end{proof}
\begin{lem}\label{lem:satura}
Let $\mc{C}$ be a pointed $\infty$-category which admits finite colimits and finite limits. Then
\enu{
\item The cokernel of any epimorphism is zero.
\item The kernel of any monomorphism is zero.
\item If $X\to Y$ is an epimorphism in $\mcc$, then for any map $Y\to Z$, we have
\[
\coker(X\to Z)\simeq \coker(Y\to Z).
\]
\item If $X\to Y$ is a monomorphism in $\mcc$, then for any map $A\to X$, we have
\[
\ker(A\to X)\simeq \ker(A\to Y).
\]
}
\end{lem}

\begin{proof}
For part (1), let $X \to Y$ be an epimorphism so we have the following two pushout squares
$$
\begin{tikzcd}
X \arrow[d] \arrow[r] & Y \arrow[d, equals] \\
Y \arrow[r, equals] \arrow[d] & Y \arrow[d] \\
0 \arrow[r]           & 0  .        
\end{tikzcd}
$$
Therefore, the outer square is also a pushout, that is, $\coker(X\to Y) = 0$. For part (3), it suffices to show that the left square in the diagram
$$
\begin{tikzcd}
X \arrow[d] \arrow[r] & Y \arrow[d] \arrow[r] & Z \arrow[d] \\
0 \arrow[r]           & 0 \arrow[r]           & Z/Y        
\end{tikzcd}
$$
is a pushout. This follows from the diagram in (1). Parts (2) and (4) follow from (1) and (3) by symmetry.
\end{proof}

\begin{prop}
Let $\mcc$ be a $\Sigma$-exact $\infty$-category. Then a morphism $f$ in $\mcc$ is an equivalence if and only if it is both a monomorphism and an epimorphism.
\end{prop}

\begin{proof}
It follows directly from the definition of $\Sigma$-exact and \cref{lem:satura}.
\end{proof}

\begin{de}\label{def:c-s-mono-epi}
Let $\mc{C}$ be a pointed $\infty$-category which admits finite colimits and finite limits. We denote by $\mcc^{\Delta^1\times\Delta^1}_{\op{cof-fib}}$ the full subcategory of $\mcc^{\Delta^1\times\Delta^1}$ spanned by all fiber-cofiber diagrams
$$
\begin{tikzcd}
X_{00} \arrow[d] \arrow[r] & X_{10} \arrow[d] \\
X_{01}=0 \arrow[r]           & X_{11}        
\end{tikzcd}
$$in $\mcc$. We define:
\begin{enumerate}
\item $\mcc^{\op{mono}}\coloneqq$ the full subcategory of $\mcc^{\Delta^1}$ spanned by monomorphisms;
\item $\mcc^{\op{n-mono}}\coloneqq$ the full subcategory of $\mcc^{\Delta^1}$ spanned by normal monomorphisms;
\item $\mcc^{\op{epi}}\coloneqq$ the full subcategory of $\mcc^{\Delta^1}$ spanned by epimorphisms;
\item $\mcc^{\op{n-epi}}\coloneqq$ the full subcategory of $\mcc^{\Delta^1}$ spanned by normal epimorphisms.
\end{enumerate}
Taking kernels and cokernels induces the following two functors over $\mcc$:
$$%
\begin{tikzcd}
\mcc^{\op{mono}} \arrow[rr, "\coker"] \arrow[rd, "ev_1"] &      & \mcc^{\Delta^1\times\Delta^1}_{\op{cof-fib}} \arrow[ld, "ev_{10}"'] \\
& \mcc &                                                                    
\end{tikzcd}
\text{  and  }
\begin{tikzcd}
\mcc^{\op{epi}} \arrow[rr, "\ker"] \arrow[rd, "ev_0"] &      & \mcc^{\Delta^1\times\Delta^1}_{\op{cof-fib}} \arrow[ld, "ev_{10}"'] \\
& \mcc &                                                                    
\end{tikzcd}.$$
\end{de}

\begin{prop}\label{prop:c-s-mono-epi}
Let $\mc{C}$ be a $\Sigma$-trivial  $\infty$-category which admits finite colimits and finite limits. The projections  
$$
\begin{tikzcd}
\mcc^{\op{mono}} \arrow[rd, "ev_1"'] & \mcc^{\Delta^1\times\Delta^1}_{\op{cof-fib}} \arrow[l, "ev_{00\to10}"'] \arrow[r, "ev_{10\to11}"] \arrow[d, "ev_{10}"] & \mcc^{\op{epi}} \arrow[ld, "ev_0"] \\
& \mcc                                                                                                                   &                                   
\end{tikzcd}
$$
restrict to the equivalences
$$
\mcc^{\op{n-mono}}\xleftarrow{\sim}\mcc^{\Delta^1\times\Delta^1}_{\op{cof-fib}}\xrightarrow{\sim}\mcc^{\op{n-epi}}
$$
whose inverses are given by $\coker$ and $\ker$, respectively.
\end{prop}

\begin{proof}
This is immediate by \cref{def:c-s-mono-epi} and \cref{kerofcoker}.
\end{proof}

\begin{rem}
By \cref{prop:c-s-mono-epi}, $\mcc$ is left $\Sigma$-exact if and only if the left projection
\begin{equation}\label{eq:c-mono}
\mcc^{\op{mono}}\xleftarrow{ev_{00\to10}}\mcc^{\Delta^1\times\Delta^1}_{\op{cof-fib}}    
\end{equation}
is an equivalence. Similarly, $\mcc$ is right $\Sigma$-exact if and only if the right projection
\begin{equation}\label{eq:c-epi}
\mcc^{\Delta^1\times\Delta^1}_{\op{cof-fib}}\xrightarrow{ev_{10\to11}}\mcc^{\op{epi}}
\end{equation}
is an equivalence.
\end{rem}

\begin{cor}\label{monoepiequi}
Let $\mcc$ be a $\Sigma$-trivial $\infty$-category which admits finite colimits and finite limits. Then for any object $X\in\mcc$ the equivalences \eqref{eq:c-mono} and \eqref{eq:c-epi} restrict to the equivalences on fibers over $X$:
$$\mcc_{/X}^{\op{n-mono}}\xleftarrow{\sim}\mcc^{\Delta^1\times\Delta^1}_{\op{cof-fib}, X}\xrightarrow{\sim}\mcc_{X/}^{\op{n-epi}} $$
where  $\mcc_{/X}^{\op{n-mono}}\subseteq \mcc_{/X}$ denotes the full subcategory of normal monomorphisms to $X$ and $\mcc_{X/}^{\op{epi}}\subseteq \mcc_{X/}$ denotes the full subcategory of normal epimorphisms from $X$.
\end{cor}

\begin{lem}\label{epiidem}
Let $\mcc^\otimes\in\calg(\prl_{\Sigma})$ be a presentably symmetric monoidal $\Sigma$-trivial $\infty$-category. If $f:\mb{1}\to S$ is a normal epimorphism in $\mcc$, then $\mb{1}\to S$ is idempotent.
\end{lem}

\begin{proof}
Since $\mcc$ is $\Sigma$-trivial, $\ker(f)$ is quotientable. Moreover, $S\simeq \mb{1}/\ker(f)$ since $f$ is normal. The result then follows from the proof of \cref{quotidl}.
\end{proof}

\begin{thm}[Ideal-epimorphism correspondence]\label{thm:ideal-epi-iden}
Let $\mcc^\otimes\in\calg(\prl_{\Sigma})$ be a presentably symmetric monoidal $\Sigma$-trivial $\infty$-category. Then there exist natural equivalences $$\idlc^n\xrightarrow{\sim}\mcc_{\mb{1}/}^{\op{n-epi}}\xleftarrow{\sim}\calg(\mcc)^n$$ where $\idlc^n\subseteq \idlc$ denotes the full subcategory of normal ideals and $\calg(\mcc)^n\subseteq\calg(\mcc)$ denotes the full subcategory spanned by the $\einf$-algebras $S$ such that $\mb{1}\to S$ is a normal epimorphism in $\mcc$.
\end{thm}

\begin{proof}
The first equivalence follows from \cref{monoepiequi}. By \cite[Proposition 4.8.2.9]{ha}, the forgetful functor $\calg(\mcc)^{\op{idem}}\xrightarrow{\sim} \mcc_{\mb{1}/}^{\op{idem}}$ is an equivalence. This restricts to the second equivalence, since normal epimorphisms from the unit $\mb{1}$ are automatically idempotent by \cref{epiidem}.
\end{proof}

\begin{cor}\label{cor:ideal-epi-iden}
Let $\mcc^\otimes\in\calg(\prl_{\Sigma})$ be a presentably symmetric monoidal $\Sigma$-trivial $\infty$-category.
\begin{enumerate}
\item If $\mcc$ is left $\Sigma$-exact then we have natural equivalences $$\idlc\xrightarrow{\sim}\mcc_{\mb{1}/}^{\op{n-epi}}\xleftarrow{\sim}\calg(\mcc)^n.$$ 
\item If $\mcc$ is $\Sigma$-exact then we have natural equivalences $$\idlc\xrightarrow{\sim}\mcc_{\mb{1}/}^{\op{epi}}\xleftarrow{\sim}\calg(\mcc)^n.$$ 
\end{enumerate}
\end{cor}

\begin{proof}
Part (1) is immediate from the definition. For part (2), note that the fully faithful functor that appears in \cref{calgepi2} is an equivalence, since any epimorphism $\mb{1}\to X$ in $\mcc$ is idempotent by \cref{epiidem}.
\end{proof}

\begin{rem}
Our next goal is to introduce a parametrized version of \cref{thm:ideal-epi-iden}.
\end{rem}

\begin{de}
Let $\mcc^\otimes\in\calg(\prl)$. We define $\calg(\mcc^{\op{mono}})$  to be the \infcat of pairs $(R,I)$ such that $R\in \calg(\mcc)$ and $I$ is an ideal of $\modu_R(\mcc)^\otimes$. This can be viewed as a cocartesian fibration $$\calg(\mcc^{\op{mono}}) \to \calg(\mcc)$$ classifying the assignment $R\mapsto \idl(R)$. We define $\funct(\Delta^1,\calg(\mcc))^n$ to be the $\einf$-maps $A\to B$ such that the underlying map is a normal epimorphism in \mcc. 
\end{de}

\begin{rem}
In fact, there exists a symmetric monoidal structure on $\mcc^{\Delta^1}$ which makes any commutative monoid object of $\mcc^{\Delta^1}$ into a Smith ideal; see \cite[Section~3.1]{mao2021revisiting} for more detail. Then an ideal pair $(R,I)$ can be identified with a Smith ideal such that the underlying map is a monomorphism. This explains our choice of notation $\calg(\mcc^{\op{mono}})$.
\end{rem}

\begin{prop}
Let $\mcc^\otimes\in\calg(\prl_\Sigma)$ be a presentably symmetric monoidal $\Sigma$-trivial $\infty$-category. Then  the assignment $$(R,I)\mapsto (R\to R/I)$$ induces a functor $$\calg(\mcc^{\op{mono}})\xrightarrow{}\funct(\Delta^1,\calg(\mcc))^n.$$
\end{prop}

\begin{proof}
First, by \cref{modumono} the category $\calg(\mcc^{\op{mono}})$ can be identified with a full subcategory of $\calg(\mcc)\times_{\mcc}\mcc^{\op{mono}}$. Then, note that the functor over $\calg(\mcc)$ $$\funct(\Delta^1,\calg(\mcc))^n\xrightarrow{(ev_0, \funct(\Delta^1,p))} \calg(\mcc)\times_{\mcc} \mcc^{\op{n-epi}}$$ is fully faithful since it is fiberwise fully faithful by \cref{calgepi2}, where $p$ denotes the forgetful functor $p:\calg(\mcc)\to\mcc$. Since taking cokernels induces the following functor over \mcc,
$$
\begin{tikzcd}
\mcc^{\op{mono}} \arrow[rd, "ev_1"'] \arrow[r, "\coker"] & \mcc^{\Delta^1\times\Delta^1}_{\op{cof-fib}} \arrow[r, "ev"] \arrow[d, "ev_{10}"] & \mcc^{\op{n-epi}} \arrow[ld, "ev_0"] \\
& \mcc                                                                              &                                   
\end{tikzcd}
$$
it produces a functor $\calg(\mcc)\times_{\mcc}\mcc^{\op{mono}}\to \calg(\mcc)\times_{\mcc} \mcc^{\op{n-epi}}$, which restricts to a functor
$$\calg(\mcc^{\op{mono}})\xrightarrow{}\funct(\Delta^1,\calg(\mcc))^n$$  by \cref{quotidl}.
\end{proof}

\begin{thm}\label{thm:parametrized-ideal-epi}
Let $\mcc^\otimes\in\calg(\prl_{\Sigma})$ be a presentably symmetric monoidal $\Sigma$-trivial $\infty$-category. Then the assignment $(R,I)\mapsto (R\to R/I)$ restricts to an equivalence $$\calg(\mcc^{\op{n-mono}})\xrightarrow{\sim}\funct(\Delta^1,\calg(\mcc))^n$$ where $\calg(\mcc^{\op{n-mono}})\subseteq \calg(\mcc^{\op{mono}})$ denotes the full subcategory spanned by the normal ideals.
\end{thm}

\begin{proof}
Since both sides are cocartesian fibrations over $\calg(\mcc)$, it suffices to establish the equivalence fiberwise, but this follows from \cref{thm:ideal-epi-iden}.
\end{proof}

\begin{cor}\label{cor:parametrized-corr}
Let $\mcc^\otimes\in\calg(\prl_{\Sigma})$ be a presentably symmetric monoidal $\Sigma$-trivial $\infty$-category. If $\mcc$ is left $\Sigma$-exact, then the assignment $(R,I)\mapsto (R\to R/I)$ induces an equivalence $$\calg(\mcc^{\op{mono}})\xrightarrow{\sim}\funct(\Delta^1,\calg(\mcc))^n.$$
\end{cor}

\begin{ex}
For $\mcc=\catper$, \cref{cor:parametrized-corr} gives a equivalence  \[
\calg\big((\catper)^{\op{mono}}\big)=\left\{
\substack{
\text{thick ideal pairs} \\
(\mck^\otimes,\mci)
}
\right\}
\simeq 
\left\{
\substack{
\text{symmetric monoidal Karoubi} \\
\text{projections }\mck^\otimes\to \votimes
}
\right\}=\funct(\Delta^1,\calg(\catper))^n.
\]
See \cite[Definition~A.3.6]{hermitionkii} for the definition of Karoubi projection.
\end{ex}

\subsection{Examples}
We now discuss examples of $\Sigma$-exact $\infty$-categories.

\begin{ex}\label{ex:Sigma-exact}
\hspace{1em}
\enu{
\item Any abelian category is $\Sigma$-exact.

\item The \infcat $\op{Cat}^{\op{perf}}$ is left $\Sigma$-exact since any monomorphism can be identified with the kernel of its cokernel. It is, however, not right $\Sigma$-exact, since not every homological epimorphism is a Karoubi projection. Recall from \cref{monoepicatperf} that the epimorphisms in $\catper$ are exactly the homological epimorphisms, but the normal epimorphisms are the Karoubi projections.

\item The \infcat $\op{Cat}_{\infty}^{\op{st}}$ of (small) stable \infcats is neither left nor right $\Sigma$-exact, since for any momomorphism $f:\mcc\to\mcd$, the kernel of $\mcd\to\coker(f)$ is the retraction closure of $\mcc$ in $\mcd$, instead of $\mcc$ itself.

\item The \infcat $\mcs_{*,\Sigma}$ of unital pointed spaces is left $\Sigma$-exact, since any monomorphism $f$ in it is a component inclusion and hence can be identified with the kernel of $\coker(f)$. Note that the cokernel in $\mcs_{*,\Sigma}$ is different from the cokernel in $\mcs_{*}$. On the other hand, $\mcs_{*,\Sigma}$ is not right $\Sigma$-exact: Let $X$ be a non-trivial acyclic space (e.g. a punctured Poincaré homology 3-sphere). Then $X_+\to *_+=S^0$ is an epimorphism in $\mcs_{*,\Sigma}$, but it has a trivial kernel and hence is not normal.
\item The 1-category $\cmon(\op{Sets})$ of commutative monoids is Barr-exact; see \cite[Remark 4.3]{abuhlail2014exact}. However, $\cmon(\op{Sets})$ is neither left nor right $\Sigma$-exact: The map $i:\mathbb{N}\to\mathbb{Z}$ is both a monomorphism and an epimorphism in $\cmon(\op{Sets})$, but we have $\coker(i)=0=\ker(i)$ and hence $i$ is not normal regarded as either a monomorphism or an epimorphism.
\item 
The \infcat $\cmon(\mcs)_\Sigma$ of unital $\einf$-spaces is neither left nor right $\Sigma$-exact. This follows from the observation that the map
$$\pi_0 \colon \cmon(\mcs)_\Sigma \to \cmon(\op{Sets})$$
preserves small colimits, fibers and monomorphisms.
}
\end{ex}

\begin{rem}
A more striking example of a $\Sigma$-exact $\infty$-category is the $\infty$-category $\lintdbl$ of dualizable $\mct$-modules for any $\totimes\in \calg(\prlst)$. To prove this, we need the following generalization of \cite[Lemma~2.62]{ramzi2024dualizable}.
\end{rem}

\begin{lem}\label{dualker}
Let $\mcv\in\calg(\prl)$ and let $f: \mathcal{M} \rightarrow \mathcal{N}$ and $i:\mcn_0\hookrightarrow\mcn$ be maps in $\linvdbl$ such that $i$ is fully faithful. The dualizable pullback $(\mcm\times_\mcn\mcn_0)^{\mcv\text{-}\mathrm{dbl}}$ in $\linvdbl$ can be identified with the largest dualizable $\mcv$-submodule of $\mcm$ contained in the pullback $\mcm\times_\mcn\mcn_0$ in $\prl_{\mcv}$ and for which the inclusion into $\mcm$ is an internal left adjoint.
\end{lem}

\begin{proof}
Let $\kappa$ be an uncountable regular cardinal such that $\votimes\in\calg(\prl_{\kappa})$. We denote by $P$ the poset of all dualizable $\mct$-submodules of $\mcm$ contained in $\mcm\times_\mcn\mcn_0$ and for which the inclusion functor into $\mcm$ is an internal
left adjoint. Then all submodules in $P$ are $\kappa$-compactly generated by $\kappa$-compacts of \mcm and hence $P$ is small. Note that, for any $\mck_0\leq \mck_1$ in $P$, the correspondening inclusion $i_0:\mck_0\hookrightarrow \mck_1$ is an internal left adjoint, since its right adjoint $i_0^R$ can be identified with $j_0^R\circ j_1$ by the commutative diagram
$$
\begin{tikzcd}
\mck_0 \arrow[rd, "j_0", hook] \arrow[r, "i_0", hook] & \mck_1 \arrow[d, "j_1", hook] \\
& \mcm   .                      
\end{tikzcd}
$$
Hence, it suffices to show that $P$ has a maximal element. Now let $\mck\subseteq\mcm$ denote the full subcategory of $\mathcal{M}$ generated under colimits and $\mct$-tensors by the images of all $\mck_\alpha\in P$. Applying \cite[Lemma 2.61]{ramzi2024dualizable} to the map $\bigoplus_{\alpha}\mck_\alpha \to \mcm$, we see that $\mck$ is a dualizable $\mcv$-module and that the inclusion $\mck \to \mcm$ is an internal left adjoint. That is, $\mck \in P$. Therefore, $\mck$ is a maximal element of $P$.
\end{proof}

\begin{rem}\label{rem:prdbl-non-sigma-exact}
Let $\votimes\in\calg(\prl_*)$ and $\mci\rightarrow\mcv\in \linvdbl$ be a smashing ideal (hence a closed ideal). Then, $\mcv\to\mcv/\mci$ is a smashing localization. Conversely, for a smashing localization $\mcv \to \mcw$, its dualizable kernel $\mck^d \to \mcv$ is a smashing ideal; see \cref{dualker}. However, the dualizable kernel of the smashing localization $\mcv \to \mcv / \mci$ need not coincide with $\mci$. This implies that the smashing ideals and the smashing localizations are not in one-to-one correspondence. In other words, $\linvdbl$ is in general not $\Sigma$-exact. 
\end{rem}

\begin{ex}
For $\mcv=\spgeq$, we have $\linvdbl \simeq \praddbl$, the $\infty$-category of dualizable additive \infcats. It is not right $\Sigma$-exact: The dualizable additive kernel of the localization $$p:\mcd(\mathbb{Z})_{\geq0}\to \mcd(\mathbb{Z}[p^{-1}])_{\geq0}$$ in $\praddbl$ is $0$, instead of $\mcd(\mathbb{Z})_{\geq0}^{p\text{-nil}}$; see \cite{levy-liang-sosnilo:dualizable} for more detail.
\end{ex}

\begin{rem}
Nonetheless, in the stable case, we do have a one-to-one correspondence between the smashing localizations and the smashing ideals, as shown below.
\end{rem}

\begin{thm}\label{sigmaexact}
Let $\totimes\in \calg(\prlst)$. The $\infty$-category $\lintdbl$ is $\Sigma$-exact.
\end{thm}

\begin{proof}
By \cref{sigmatrivexamples}(8), $\lintdbl$ is $\Sigma$-trivial. It remains to show that any monomorphism and epimorphism in $\lintdbl$ is normal. 

Let $\mck\to\mca$ be a monomorphism in $\lintdbl$, \ie a fully faithful internal left adjoint between dualizable \mct-modules. By \cite[Proposition 1.62]{ramzi2024dualizable}, $\mca/\mck$ can be calculated in $\lint$. It then follows from \cref{dualker} that \mck is exactly the dualizable kernel of $\mca\to\mca/\mck$ in $\lintdbl$. Hence, $\mck\to\mca$ is normal by \cref{kerofcoker}.

Let $p:\mca\to \mcb$ be an epimorphism in $\lintdbl$. By \cref{kerofcoker}, it suffices to show that its dualizable kernel $\mck^d$ in $\lintdbl$ coincides with its usual Verdier kernel $i:\mck \to \mca$ in $\lint$. Indeed, by \cite[Example~1.18]{ramzi2024dualizable}, $i$ is an internal left adjoint. It then follows from \cite[Theorem~1.49]{ramzi2024dualizable} that $\mck$ is a dualizable $\mct$-module. This completes the proof.
\end{proof}

\begin{cor}\label{idlfullyfai}
Let $\totimes\in\calg(\prlst)$. Then
the Morita embedding $\calg(\mct)\hookrightarrow \calg(\lint)$ constructed in \cite[Corollary 4.8.5.21]{ha} factors through $\calg(\mct)\hookrightarrow \calg(\lintdbl)$
and induces an equivalence $$
\calg(\mct)^{\op{idem}}\xrightarrow{\sim}\calg(\lintdbl)^n$$
where $\calg(\lintdbl)^n\subseteq\calg(\lintdbl)$ denotes the full subcategory spanned by the commutative dualizable $\mct$-algebras $\totimes \to \mcs^\otimes$  such that the underlying functor is a normal epimorphism\footnote{Note that, in the stable case, all epimorphisms in $\lintdbl$ are normal and identical with localizations by \cref{cdblmono,sigmaexact}.} in $\lintdbl$.
\end{cor}

\begin{proof}

The module category $\modu_R(\mct)^\otimes$ over any \ein-algebra $R\in \calg(\mct)$ is rigid over \mct (\cite[Example 4.38]{ramzi2024locally}), hence $\modu_R(\mct)^\otimes$ lies in $\calg(\lintdbl)$ and the Morita embedding factors through $\calg(\mct)\hookrightarrow \calg(\lintdbl)$ by \cite[Observation 4.51]{ramzi2024locally}.

For an idempotent $\einf$-algebra $R$, the base change functor $\mcc\to \modu_R(\mcc)$ is always a localization, so the Morita embedding restricts to an embedding $\calg(\mcc)^{\op{idem}}\hookrightarrow\calg(\lintdbl)^n$. It then suffices to show that this embedding is essentially surjective. This follows from \cref{rigidbarr}.

\end{proof}

\begin{rem}
Let $\totimes\in\calg(\prlst)$. The equivalence in \cref{idlfullyfai} is compatible with the one built in \cref{smidsmcoloc}, as shown by the following commutative diagram
$$
\begin{tikzcd}
\calg(\mct)^{\op{idem}} \arrow[d, "\sim"] \arrow[rr, "\sim"] &                                    & \calg(\lintdbl)^n \arrow[d, "\sim"] \\
\cidemt \arrow[r, "\left<- \right>"]                         & \cidem(\lintdbl) \arrow[r, "\sim"] & \idl(\lintdbl)                     
\end{tikzcd}$$
where the right vertical equivalence comes from \cref{cor:ideal-epi-iden}.
\end{rem}

\appendix
\numberwithin{thm}{section}
\section{Stone duality}\label{app:stone-duality}
For this appendix, we refer the reader to \cite{Johnstone82} or \cite[Appendix A.1]{sag}.

\begin{de}
A (small) poset is said to be a \textbf{suplattice} if it admits all joins. By the adjoint functor theorem, a suplattice also admits all meets and is thus a complete lattice. A morphism of suplattices is a map preserving joins.
\end{de}

\begin{de}
A \textbf{frame}\footnote{A frame is also equivalent to a $0$-topos, see \cite[\textsection6.4.2]{htt}.} is a suplattice in which finite meets distribute over arbitrary joins:
\[
x \wedge (\bigvee_{i\in I}y_i) = \bigvee_{i\in I}(x\wedge y_i).
\]
A morphism of frames is a map preserving arbitrary joins and finite meets. We write $\frm$ for the category of frames.
\end{de}

\begin{rem}
Given a topological space $X$, the poset $\mathcal{O}(X)$ of open subsets of $X$ is a frame. This provides a functor from topological spaces to frames:
\[
\mathcal{O}\colon \Top \to \frm.
\]
\end{rem}

\begin{de}\label{de:pt}
A \textbf{point} of a frame $F$ is a frame morphism $f\colon F \to \{0,1\} = \mathcal{O}(\ast)$. The set $\op{pt}(F)$ of points of $F$ admits a topology with open subsets given by $U_y \coloneqq \SET{f\in \op{pt}(F)}{f(y)=1}$ ranging over the elements $y$ of $F$. This induces a functor
\[
\op{pt}\colon \frm \to \Top
\]
which is left adjoint to $\mathcal{O}$.
\end{de}

\begin{rem}
The points of a frame $F$ are in one-to-one correspondence with the prime elements of $F$, 
by the assignment $f\mapsto \bigvee f^{-1}(0) $. See \cref{thm:specQ} for a similar statement in the more general setting of preframes. Note that the prime elements of a frame are also called the indecomposable elements in \cite{sag}.
\end{rem}

\begin{de}
Let $\eta$ and $\varepsilon$ denote the unit and counit for the adjunction
\[
\begin{tikzcd}
\frm  \arrow[r, "\op{pt}", shift right=-1ex]  & \arrow[l,"\perp"', "\mathcal{O}", shift right=-1ex] \Top \end{tikzcd}.
\]
We say that a frame $F$ is \textbf{spatial} if $\eta_F\colon F\to\mathcal{O}(\op{pt}(F))$ is an isormophism. We say that a topological space $X$ is \textbf{sober} if $\varepsilon_X\colon X\to\op{pt}(\mathcal{O}(X))$ is an isomorphism. Therefore, the adjunction restricts to an equivalence of spatial frames and sober spaces:
\[
\frmsp \simeq \topsob.
\]
\end{de}

\begin{rem}
A frame is spatial if and only if it has enough points, meaning that the collection of all points $\{f:F\to\{0,1\}\}$ is jointly conservative. The ``only if'' direction is obvious. For the ``if'' direction, see \cite[Example 1.5.4.5]{sag} for example. 
\end{rem}

\begin{de}\label{def:frame-finite-elt}
Let $F$ be a frame. An element $x \in F$ is called \textbf{finite} if whenever $x \le \bigvee S$ for some subset $S \subseteq F$, there exists a finite subset $K\subseteq S$ such that $x \le \bigvee K$. We denote by $F^{\omega}$ the set of finite elements of $F$.
\end{de}

\begin{rem}
When regarded as categories, $F^{\omega}$ is exactly the subcategory of compact objects of $F$.
\end{rem}

\begin{rem}\label{rem:stone-duality}
Let $F$ be a spatial frame. Since $F\cong\mathcal O(\op{pt}(F))$, the elements of $F$  correspond bijectively with the open subsets of $\op{pt}(F)$. Note that the finite elements of $F$ correspond exactly with the quasi-compact open subsets of $\op{pt}(F)$.
\end{rem}

\begin{de}
We say that a frame $F$ is \textbf{coherent} if $F^{\omega}$ is closed under finite joins and finite meets and $F^{\omega}$ generates $F$ under joins. A morphism of coherent frames is a morphism of frames which preserves finite elements. We denote by $\frmcoh$ the category of coherent frames.
\end{de}

\begin{rem}
Every coherent frame is spatial. Thus, $\frmcoh$ is a nonfull subcategory of $\frmsp \simeq \topsob$.
\end{rem}

\begin{de}
The topological spaces corresponding to the coherent frames under the equivalence $\frmsp \simeq \topsob$ are called \textbf{coherent spaces} or \textbf{spectral spaces}. We write $\op{Top}_{\op{coh}}$ for the category of coherent spaces.
\end{de}

\begin{rem}
The category $\frmcoh$ is also equivalent to the category $\dlat$ of distributive lattices. Indeed, taking finite elements induces an equivalence
\[
(-)^{\omega}\colon \frmcoh \simeq \dlat
\]
with the inverse sending a distributive lattice $L$ to its ind-category $\op{Ind}(L)$. See \cite[Proposition~II.3.2]{Johnstone82}.
\end{rem}

\begin{rem}
Any lattice $L$ has a dual lattice $L^{\vee}$ obtained by simply reversing the order. If $L$ is distributive then so is $L^{\vee}$.
\end{rem}

\begin{de}
Let $F$ be a coherent frame. The coherent frame $F^{\vee}\coloneqq\op{Ind}((F^{\omega})^\vee)$ is called the \textbf{Hochster dual} of $F$. For any coherent space $X$, we call $X^\vee\coloneqq \op{pt}(\mathcal O(X)^{\vee})$ the \textbf{Hochster dual} of $X$.
\end{de}

\begin{rem}
One readily checks that $F^{\vee\vee} \cong F$ and $X^{\vee\vee} \cong X$.
\end{rem}

\section{Epimorphisms and monomorphisms}
\begin{de}
Let $\mc{C}$ be an $\infty$-category and $f:X\to Y$ be a morphism in it. We say $f$ is an \textbf{epimorphism} if the following diagram is a pushout:
$$\begin{tikzcd}
X \arrow[d] \arrow[r] & {Y} \arrow[d, double, no head] \\ 
{Y} \arrow[r, double, no head] & {Y} 
\end{tikzcd}.$$
We say $f$ is a \textbf{monomorphism} if $f$ is an epimorphism in $\mcc^{\op{op}}$.
\end{de}
\begin{rem}
Monomorphisms can be identified with the $(-1)$-truncated morphisms; see \cite[\textsection5.5.6]{htt} for more detail.
\end{rem}
\begin{rem}
Note that, $\infty$-categorically, a retract $X\to Y \to X$ does not produce a monomorphism or epimorphism. For example $*\to S^1 \to *$ is a retract in the $\infty$-category of spaces, but $*\to S^1$ is not $(-1)$-truncated.
\end{rem}
\begin{prop}
Let $\mc{C}$ be an $\infty$-category. Consider a diagram $$\begin{tikzcd}
& A \arrow[rd, "f"] \arrow[ld, "h"'] &   \\
B \arrow[rr, "g"] &                                    & C
\end{tikzcd}$$
in $\mc{C}$ where $f$ is an epimorphism. Then $g$ is an epimorphism if and only if $h$ is an epimorphism.
\end{prop}
\begin{proof}
See \cite[Lemma 5.5.6.14]{htt}.
\end{proof}

\begin{prop}\label{calgepi}
Let $\mc{C}^\otimes$ be a symmetric monoidal $\infty$-category such that the underlying $\infty$-category $\mcc$ admits countable colimits, and that tensoring with any fixed object preserves such colimits. Let $f\colon A\to B \in \calg(\mcc)$ be a map of $\einf$-algebras. If $f$ is an epimorphism as a map in $\mcc$, then $f$ is also an epimorphism in $\calg(\mcc)$.
\end{prop}

\begin{proof}
Let $U:\mathrm{CAlg}(\mathcal C)\to \mathcal C$ denote the forgetful functor, and let $F=\mathrm{Sym}^*$ be its left adjoint, the free commutative algebra functor. Write $T=UF$ for the associated monad on $\mathcal C$. By \cite[Barr-Beck-Lurie Theorem 4.7.3.5]{ha}, the forgetful functor
\[
U:\mathrm{CAlg}(\mathcal C)\to \mathcal C
\]
is monadic. Hence, for every $C\in \mathrm{CAlg}(\mathcal C)$, the map between mapping spaces
\[
\operatorname{Map}_{\mathrm{CAlg}(\mathcal C)}(B,C)
\longrightarrow
\operatorname{Map}_{\mathrm{CAlg}(\mathcal C)}(A,C)
\]
can be identified, via the standard monadic resolution, with
\[
\lim_{[n]\in \Delta}
\mapp_{\calg(\mcc)}
\bigl(F(T^nUB),C\bigr)
\longrightarrow
\lim_{[n]\in \Delta}
\mapp_{\calg(\mcc)}
\bigl(F(T^nUA),C\bigr).
\]
Using the adjunction \(F\dashv U\), each term in this map is naturally identified with
\[
\mapp_{\mcc}(T^nUB,UC)
\longrightarrow
\mapp_{\mcc}(T^nUA,UC).
\]
By assumption, $UA\to UB$ is an epimorphism in $\mcc$. Since $T=UF=\mathrm{Sym}^*$ is the free commutative algebra monad, it is built from tensor powers, symmetric coinvariants, and countable coproducts. By the compatibility assumption on $\mcc^\otimes$, these operations preserve epimorphisms. Therefore
\[
T^nUA\longrightarrow T^nUB
\]
is an epimorphism in $\mcc$ for every $n\geq 0$. It follows that, for each $n\geq 0$, the induced map
\[
\mapp_{\mcc}(T^nUB,UC)
\longrightarrow
\mapp_{\mcc}(T^nUA,UC)
\]
is a monomorphism of spaces. Therefore the induced map on limits is also a monomorphism. Hence
\[
\mapp_{\calg(\mcc)}(B,C)
\longrightarrow
\mapp_{\calg(\mcc)}(A,C)
\]
is a monomorphism for every $C\in \calg(\mcc)$. Thus $f$ is an epimorphism in $\calg(\mcc)$.
\end{proof}

\begin{ex}\label{acyc}
A map $f \colon X\to Y$ in the \infcat of spaces is an epimorphism if and only if for each $y\in Y$ the fiber $f_y$ is an acyclic space, that is, the reduced homology groups $\widetilde{H}_*(X)$ vanish. See \cite{hoyois2019quillen,Raptis2019} for more discussion on acyclic spaces.
\end{ex}

\begin{ex}\label{monoepicatperf}
Let $F:\mcc\to \mcd$ be a morphism in $\catper$.
\enu{
\item $F$ is a monomorphism if and only if it is fully faithful.
\item  $F$ is an epimorphism if and only if it is a homological epimorphism, meaning that the induced functor $\op{Ind}(F) \colon \op{Ind}(\mcc) \to \op{Ind}(\mcd)$ is a localization.}
\end{ex}
\begin{proof}
(1): The ``if'' part is obvious. For the ``only if'' part, note that the map of groupoids
\[
\mathcal{C}^{\simeq} \simeq \operatorname{Fun}^{\op{ex}}(\mathrm{Sp}^\omega, \mathcal{C})^{\simeq} \rightarrow \operatorname{Fun}^{\op{ex}}(\mathrm{Sp}^\omega, \mathcal{D})^{\simeq} \simeq \mathcal{D}^{\simeq}
\]
is a monomorphism. By the usual stable-category argument used in \cite[Proposition A.1(1)]{efimov2024k} (\cf \cite[Lemma~7.5]{hattt}), a monomorphism on maximal subgroupoids induced by an exact functor between stable $\infty$-categories forces fully faithfulness. Therefore, $F$ is fully faithful.

(2): This follows from \cite[Proposition A.1(2)]{efimov2024k} and the fact that the functor $$\op{Ind}\colon\catper\rightarrow\prlst$$ is conservative and preserves small colimits, hence detects the codiagonal criterion for epimorphisms.
\end{proof}

\begin{lem}\label{cartmono}
Let $p:\mce\to\mcb$ be a categorical fibration of $\infty$-categories. If $f:x\to y$ is a monomorphism in $\mcb$, then any $p$-Cartesian lifting of $f$ is a monomorphism in $\mce$.
\end{lem}
\begin{proof}
Let $u\in\mce$ and let $\bar{f}:\bar{x}\to\bar{y}$ be a $p$-Cartesian lifting of $f$. Then we have a pullback diagram of spaces 
\[
\begin{tikzcd}
{\mapp_{\mc{E}}(u,\bar{x})} \arrow[d] \arrow[r] & {\mapp_{\mc{E}}(u,\bar{y})} \arrow[d] \\
{\mapp_{\mc{B}}(p(u),x)} \arrow[r]              & {\mapp_{\mc{B}}(p(u),y)}.           
\end{tikzcd}
\]
Since monomorphisms are stable under pullback, it follows that $\bar{f}$ is a monomorphism in $\mce$.
\end{proof}

\begin{prop}[{\cite[Corollary A.9 and Proposition A.11]{ramzi2023elementaryproofnaturalityyoneda}}]\label{monolem}\,
\begin{enumerate}[label=(\arabic*),font=\normalfont]

\item Let $\mcc$ be an $\infty$-category, and let $\mcc_{\op{mono}} \subseteq \mcc^{\Delta^1}$ be the full subcategory of the arrow category spanned by monomorphisms. Let $\mci$ be an arbitrary $\infty$-category with an essential surjection $p: \mci_0 \rightarrow \mci$. Then the space of diagonal fillers in any commutative diagram is empty or contractible:
\[
\begin{tikzcd}
\mci_0 \arrow[d, "p"] \arrow[r, "m"]        & \mcc_{\op{mono}} \arrow[d, "ev_1"] \\
\mci \arrow[r, "f"] \arrow[ru, dashed] & \mcc.                              
\end{tikzcd}
\]
Here, the map $\mcc_{\op{mono}} \rightarrow \mcc$ is evaluation at $1$, which picks out the target of an arrow. This space of diagonal fillers is nonempty if and only if for every arrow $p(i) \rightarrow p(j)$ in $\mci$, the corresponding map $f(p(i)) \rightarrow f(p(j))$ is compatible with the monomorphisms $m(i)=\left(h(i) \rightarrow f(p(i))\right)$ and $m(j)=\left(h(j) \rightarrow f(p(j))\right)$ in the sense that the composite $h(i) \rightarrow f(p(i)) \rightarrow f(p(j))$ factors through $h(j)$.
\item  Let $f, g, h: \mci \rightarrow \mcc$ be functors between $\infty$-categories and let $\eta: f \rightarrow g$ be a transformation which is pointwise a monomorphism. In this case, $\eta$ is a monomorphism, and a transformation $\alpha: h \rightarrow g$ factors through $\eta$ if and only if $\alpha_i: h(i) \rightarrow g(i)$ factors through $\eta_i: f(i) \rightarrow g(i)$ for all $i \in \mci$.
\end{enumerate}
\end{prop}

\begin{cor}[Relative version]\label{monocor}
Let $e:\mc{E}\to\mcb$ be a categorical fibration between $\infty$-categories and $p:\mci_0\to\mci$ an essentially surjective functor between $\infty$-categories over $\mcb$. Consider a Cartesian lifting problem
\[
\begin{tikzcd}
\Delta^0 \arrow[d, "\{1\}"'] \arrow[r, "f"]           & {\funct_{/\mcb}(\mci,\mc{E})} \arrow[d, "q"] \\
\Delta^1 \arrow[r, "\theta"'] \arrow[ru, dashed, "\theta'"] & {\funct_{/\mcb}(\mci_0,\mce)}
\end{tikzcd}
\]
where $\theta: h\to f\circ p$ is a pointwise monomorphism in $\mce$. The following hold:
\begin{enumerate}
\item If for every arrow $p(i) \rightarrow p(j)$ in $\mci$ the composition $h(i) \rightarrow f(p(i)) \rightarrow f(p(j))$ factors through $h(j)$, then $\theta$ admits a Cartesian lifting $\theta':\bar{h}\to f$ with respect to  $q:\funct_{/\mcb}(\mci,\mc{E})\xrightarrow{} \funct_{/\mcb}(\mci_0,\mce)$.
\item Moreover, $\theta'$ is a monomorphism in $\funct_{/\mcb}(\mci,\mc{E})$, and a transformation $\alpha: g \rightarrow f \in \funct_{/\mcb}(\mci,\mc{E})$ over $\mcb$ factors through $\bar{h}$ if and only if $\alpha_i\colon g(p(i)) \rightarrow f(p(i))$ factors through $h(i) \rightarrow f(p(i))$ for all $i \in \mci_0$.  
\end{enumerate}
\end{cor}

\begin{proof}
First we note that $q$ can be decomposed as $$\funct_{/\mcb}(\mci,\mc{E})\xrightarrow{q_1}\funct(\mci,\mc{E})\times_{\funct(\mci_0,\mc{E})}\funct_{/\mcb}(\mci_0,\mc{E})\xrightarrow{q_2} \funct_{/\mcb}(\mci_0,\mc{E}).$$
A $q_2$-Cartesian lifting $\theta_2'$ can be found by \cref{monolem}. It suffices to show that $\theta_2'$ admits a Cartesian lifting. Next we observe that $q_1$ can be identified with
\[
\funct(\mci,\mc{E})\times_{\funct(\mci,\mc{B})}\{*\}\xrightarrow{q_1}\funct(\mci,\mc{E})\times_{\funct(\mci_0,\mc{B})}\{*\}
\]
which is a pullback from
\[
\funct(\mci,\mc{B})\xrightarrow{\delta}\funct(\mci,\mc{B})\times_{\funct(\mci_0,\mc{B})}\funct(\mci,\mc{B})=\funct(\mci\sqcup_{\mci_0}\mci,\mc{B}).
\]
However, by construction the image of $\theta_2'$ in $\funct(\mci\sqcup_{\mci_0}\mci,\mc{B})$ is an identity map, so $\theta_2'$ automatically admits a $q_1$-Cartesian lifting $\theta'$ as desired. The monomorphism property of $\theta'$ also follows from \cref{monolem,cartmono}.
\end{proof}

\begin{lem}\label{monoperad}
Let $\mcc^\otimes$ be a symmetric monoidal $\infty$-category. Any monomorphism
\[
(X_1,\dots,X_n)\to (Y_1,\dots,Y_n)
\]
in $\mcc^n\simeq\mcc^{\otimes}_{\left<n \right>}$ is also a monomorphism in $\mcc^\otimes$. Moreover, if $\mcc^\otimes$ is compatible with epimorphisms, meaning that the tensor product preserves epimorphisms separately in each variable, then any epimorphism in $\mcc^n\simeq\mcc^{\otimes}_{\left<n \right>}$ is also an epimorphism in $\mcc^\otimes$.
\end{lem}

\begin{proof}
For a pair of objects $C=\left(\langle m\rangle,\left(C_1, \ldots, C_m\right)\right)$ to $C^{\prime}=\left(\langle n\rangle,\left(C_1^{\prime}, \ldots, C_n^{\prime}\right)\right)$ in $\mathcal{C}^{\otimes}$, the mapping space $\operatorname{Map}_{\mathcal{C}^{\otimes}}(C, C^{\prime})$ can be written as  
$$
\coprod_{\alpha:\langle m\rangle \rightarrow\langle n\rangle} \prod_{1 \leq j \leq n} \operatorname{Map}_{\mathcal{C}}\left(\bigotimes_{\alpha(i)=j}C_i, C_j^{\prime}\right).
$$

Now let
$$
f=(f_1,\ldots,f_n):(X_1,\ldots,X_n)\to (Y_1,\ldots,Y_n)
$$
be a monomorphism in $\mcc^n\simeq \mcc^\otimes_{\langle n\rangle}$. Thus each $f_j:X_j\to Y_j$ is a monomorphism in $\mcc$. We need to show that $f$ is a monomorphism in $\mcc^\otimes$. By definition, it is enough to show that for every object
$$
C=\left(\langle m\rangle,\left(C_1,\ldots,C_m\right)\right)
$$
of $\mcc^\otimes$, the induced map
$$
\mapp_{\mcc^\otimes}(C,X)\to \mapp_{\mcc^\otimes}(C,Y)
$$
is a monomorphism of spaces. Using the above description of mapping spaces, this map is the coproduct over $\alpha:\langle m\rangle\to\langle n\rangle$ of products of the maps
$$
\mapp_{\mcc}\left(\bigotimes_{\alpha(i)=j}C_i,X_j\right)
\to
\mapp_{\mcc}\left(\bigotimes_{\alpha(i)=j}C_i,Y_j\right).
$$
For each fixed $\alpha$ and $j$, this map is induced by $f_j:X_j\to Y_j$, and hence is a monomorphism of spaces. Since monomorphisms between spaces are closed under products and coproducts, the map
$$
\mapp_{\mcc^\otimes}(C,X)\to \mapp_{\mcc^\otimes}(C,Y)
$$
is a monomorphism. Therefore $f$ is a monomorphism in $\mcc^\otimes$.

The proof for epimorphisms is similar. Let
$$
f=(f_1,\ldots,f_n):(X_1,\ldots,X_n)\to (Y_1,\ldots,Y_n)
$$
be an epimorphism in $\mcc^n\simeq \mcc^\otimes_{\langle n\rangle}$. Thus each $f_i:X_i\to Y_i$ is an epimorphism in $\mcc$. To show that $f$ is an epimorphism in $\mcc^\otimes$, it suffices to show that for every
$$
D=\left(\langle m\rangle,\left(D_1,\ldots,D_m\right)\right)
$$
the induced map
$$
\mapp_{\mcc^\otimes}(Y,D)\to \mapp_{\mcc^\otimes}(X,D)
$$
is a monomorphism of spaces. By the same description of mapping spaces, this map is the coproduct over $\alpha:\langle n\rangle\to\langle m\rangle$ of products of the maps
$$
\mapp_{\mcc}\left(\bigotimes_{\alpha(i)=k}Y_i,D_k\right)
\to
\mapp_{\mcc}\left(\bigotimes_{\alpha(i)=k}X_i,D_k\right).
$$
For fixed $\alpha$ and $k$, this map is induced by
$$
\bigotimes_{\alpha(i)=k}X_i
\to
\bigotimes_{\alpha(i)=k}Y_i .
$$
This morphism is an epimorphism by our assumption that the tensor product preserves epimorphisms separately in each variable. Hence, the induced map on mapping spaces is a monomorphism. Since monomorphisms of spaces are closed under products and coproducts, $f$ is an epimorphism in $\mcc^\otimes$.
\end{proof}

\begin{prop}
Let $\mcc^\otimes$ be a symmetric monoidal $\infty$-category such that $\mcc$ admits pullbacks and let $\mc{O}^\otimes$ be an $\infty$-operad. Given an $\mc{O}$-algebra $A\in\alg_{\mc{O}}(\mcc)$, the forgetful functor $\alg_{\mc{O}}(\mcc)\to \funct(\mc{O},\mc{C})$ preserves and reflects monomorphisms and the induced functor
$$(\alg_{\mc{O}}(\mcc)_{/A})_{\leq-1}\to (\funct(\mc{O},\mc{C})_{/A_{\left<1\right>}})_{\leq-1}$$ is fully faithful. The essential image consists of the pointwise monomorphisms $B\to A_{\left<1\right>}$ such that for any $\alpha: T=\left(\langle m\rangle,\left(T_1, \ldots, T_m\right)\right)\to T^{\prime}=\left(\langle 1\rangle,(T')\right)$ in $ \mc{O}^\otimes$ the map $\bigotimes_{i}B_{T_i}\to \bigotimes_{i}A_{T_i}\to A_{T'}$ factors through $B_{T'}$.
\end{prop}

\begin{proof}
Take
\[
\mcb=\mc{O}^\otimes,\qquad 
(\mci_0\to\mci)=
(\mc{O}^\otimes\times_{\op{Fin}_*}\op{Triv}^\otimes\to \mc{O}^\otimes),
\qquad
\mce=\mcc^\otimes\times_{\op{Fin}_*}\mc{O}^\otimes,
\]
where $\op{Triv}^\otimes$ denotes the trivial $\infty$-operad; see \cite[Example 2.1.1.20]{ha}. Under the usual identification of $\mc{O}$-algebras in $\mcc$ with sections of $\mce\to\mc{O}^\otimes$ carrying inert morphisms to Cartesian morphisms, restriction along $\mci_0\to\mci$ is precisely the forgetful functor
\[
\alg_{\mc{O}}(\mcc)\to \funct(\mc{O},\mc{C}).
\]
By \cref{monoperad}, monomorphisms of such sections are detected after
restriction to $\mci_0$, hence the forgetful functor preserves and reflects monomorphisms. Applying \cref{monocor} to the above choice of $\mcb,\mci_0,\mci,\mce$ gives the fully faithfulness of
\[
(\alg_{\mc{O}}(\mcc)_{/A})_{\leq-1}\to
(\funct(\mc{O},\mc{C})_{/A_{\left<1\right>}})_{\leq-1}.
\]
Moreover, the factorization criterion in \cref{monocor} says exactly that, for every morphism
\[
\alpha:T=\left(\langle m\rangle,\left(T_1,\ldots,T_m\right)\right)
\to
T'=\left(\langle 1\rangle,(T')\right)
\]
in $\mc{O}^\otimes$, the composite
\[
\bigotimes_i B_{T_i}\to \bigotimes_i A_{T_i}\to A_{T'}
\]
factors through $B_{T'}$. This is precisely the stated description of the essential image.
\end{proof}

\begin{prop}
Let $\mcc^\otimes$ be a symmetric monoidal $\infty$-category which is compatible with epimorphisms and let $\mc{O}^\otimes$ be an $\infty$-operad. Given an $\mc{O}$-algebra $A\in\alg_{\mc{O}}(\mcc)$, let $\alg_{\mc{O}}(\mcc)_{A/}^{\prime}\subseteq \alg_{\mc{O}}(\mcc)_{A/}$ denote the full subcategory spanned by the morphisms $A\to D$ such that $A_T\to D_T$ is epimorphic in $\mcc$ for every $T\in \mc{O}^\otimes_{\left<1\right>}$.
Then we have $$\alg_{\mc{O}}(\mcc)_{A/}^{\prime}\subseteq \alg_{\mc{O}}(\mcc)_{A/}^{\op{epi}}$$ and the forgetful functor $$\alg_{\mc{O}}(\mcc)_{A/}^{\prime}\to \funct(\mc{O},\mc{C})_{A_{\left<1\right>}/}^{\op{epi}}$$ is fully faithful. The essential image consists of the maps $ A_{\left<1\right>}\to B$ such that for any
\[
\alpha: T=\left(\langle m\rangle,\left(T_1, \ldots, T_m\right)\right)\to T^{\prime}=\left(\langle 1\rangle,(T')\right)
\]
in $\mc{O}^\otimes$, the map $\bigotimes_{i}A_{T_i}\to  A_{T'} \to B_{T'}$ factors through $\bigotimes_{i}A_{T_i}\to\bigotimes_{i}B_{T_i}$.
\end{prop}

\begin{proof}
This follows from the dual version of \cref{monocor,monoperad} by taking
\[
\mcb=\mc{O}^\otimes,\qquad 
(\mci_0\to\mci)=
(\mc{O}^\otimes\times_{\op{Fin}_*}\op{Triv}^\otimes\to \mc{O}^\otimes),
\qquad
\mce=\mcc^\otimes\times_{\op{Fin}_*}\mc{O}^\otimes
\]
as in the proof of previous proposition.
\end{proof}

\begin{rem}
Unlike in the monomorphic case, the forgetful functor $\alg_{\mc{O}}(\mcc)\to \funct(\mc{O},\mc{C})$ does not necessarily preserve epimorphisms. For instance, taking $\mc{O}^\otimes=\op{Fin}_*$ and $\mcc^\otimes=\op{Ab}^\otimes$, we see that $\mathbb{Z}\to \mathbb{Z}[p^{-1}]$ is an epimorphism in $\calg(\op{Ab})$ but not an epimorphism in $\op{Ab}$.
\end{rem}

As a corollary of the previous proposition, we obtain a stronger version of \cref{calgepi} as follows:
\begin{cor}\label{calgepi2}
Let $\mc{C}^\otimes$ be a symmetric monoidal $\infty$-category which is compatible with epimorphisms. Let $A \in \calg(\mc{C})$. If $A\xrightarrow{} B$ is an epimorphism in $\mc{C}$ such that $A\otimes A\xrightarrow{} A\to B$ factors through $A\otimes A\to B\otimes B$, then there exists a unique commutative $A$-algebra structure on $B$ such that $A\to B$ is an epimorphism in $\calg(\mc{C})$ satisfying the following universal property: a commutative $A$-algebra $A\to C$ admits a factorization $A\to B\to C$ in $\calg(\mcc)$ if and only if it admits such a factorization in $\mcc$. Moreover, the forgetful functor $$\calg(\mcc)_{A/}^{\prime}\to \mc{C}_{A/}^{\op{epi}}$$ is fully faithful with essential image consisting of the epimorphisms $ A\to B$ in $\mcc$ such that the map $A\otimes A \to A\to B$ factors through $A\otimes A\to B\otimes B$.
\end{cor}

\begin{prop}\label{modumono}
Let $\mcc^\otimes$ be a symmetric monoidal $\infty$-category. Let $R\in\calg(\mcc)$ and $M\in\modu_R(\mcc)$. 
\enu{
\item If $\mcc$ admits pullbacks then the forgetful functor $(\modu_R(\mcc)_{/M})_{\leq-1}\to (\mc{C}_{/M})_{\leq-1}$ is fully faithful with essential image spanned by the monomorphisms $N\to M$ in $\mcc$ such that $R\otimes N \to R\otimes M \to M$ factors through $N$.
\item If $\mcc^\otimes$ is compatible with epimorphisms then the forgetful functor $\modu_R(\mcc)_{M/}^{\op{epi}}\to \mc{C}_{M/}^{\op{epi}}$ is fully faithful with essential image spanned by the epimorphisms $M\to N$ in $\mcc$ such that $R\otimes M \to  M \to N$ factors through $R\otimes M\to R\otimes N$.} 
\end{prop}

\begin{proof}
This follows from \cref{monocor,monoperad} by taking
\[
\mcb=\op{LM}^\otimes,\qquad 
(\mci_0\to\mci)=
(\op{LM}^\otimes\times_{\op{Fin}_*}\op{Triv}^\otimes\to \op{LM}^\otimes),
\qquad
\mce=\mcc^\otimes\times_{\op{Fin}_*}\op{LM}^\otimes
\]
where $\op{LM}^\otimes$ denotes the $\infty$-operad of left module action (see \cite[\textsection 4.2.1]{ha}).
\end{proof}	

\clearpage
\printbibliography[heading=bibintoc]

\bigskip

\textsc{Jiacheng Liang, Mathematics Department, Johns Hopkins University, Baltimore, MD 21218, USA}

\textit{Email address:} \href{mailto:jliang66@jhu.edu}{\texttt{jliang66@jhu.edu}}

\bigskip

\textsc{Changhan Zou, Mathematics Department, University of California, Santa Cruz, CA 95064, USA}

\textit{Email address:} \href{mailto:czou3@ucsc.edu}{\texttt{czou3@ucsc.edu}}
\end{document}